\documentclass[11pt,twoside,letter]{article}
\usepackage{amsthm, amsmath, amssymb, amsfonts, graphicx, epsfig}
\usepackage{accents}

\usepackage{bbm}
\usepackage{mathrsfs}

\usepackage{enumerate}

\usepackage[colorlinks=true, pdfstartview=FitV, linkcolor=blue,
            citecolor=blue, urlcolor=blue]{hyperref}
\usepackage[usenames]{color}

\usepackage{url}
\makeatletter\def\url@leostyle{%
 \@ifundefined{selectfont}{\def\UrlFont{\sf}}{\def\UrlFont{\scriptsize\ttfamily}}} \makeatother

\usepackage[margin=1.0in, letterpaper]{geometry}

\newtheorem{theorem}{Theorem}
\newtheorem{proposition}[theorem]{Proposition}
\newtheorem{lemma}[theorem]{Lemma}
\newtheorem{corollary}[theorem]{Corollary}
\newtheorem{assumption}[theorem]{Assumption}
\theoremstyle{definition}

\theoremstyle{remark}
\newtheorem{remark}[theorem]{Remark}

\numberwithin{equation}{section}
\numberwithin{theorem}{section}

\definecolor{Red}{rgb}{1.0,0,0.0}
\definecolor{Blue}{rgb}{0,0.0,1.0}

\def\cA{\mathcal{A}}
\def\cB{\mathcal{B}}

\def\cD{\mathcal{D}}

\def\cG{\mathcal{G}}
\def\cH{\mathcal{H}}
\def\cI{\mathcal{I}}
\def\cJ{\mathcal{J}}

\def\cM{\mathcal{M}}

\def\cP{\mathcal{P}}
\def\cQ{\mathcal{Q}}

\def\bE{\mathbb{E}}
\def\bF{\mathbb{F}}

\def\bN{\mathbb{N}}

\def\bP{\mathbb{P}}

\def\bR{\mathbb{R}}

\def\sD{\mathscr{D}}

\def\sF{\mathscr{F}}
\def\sG{\mathscr{G}}

\def\sL{\mathscr{L}}

\def\sX{\mathscr{X}}
\def\sY{\mathscr{Y}}
\def\sZ{\mathscr{Z}}

\def\mA{\mathsf{A}}
\def\mB{\mathsf{B}}
\def\mC{\mathsf{C}}
\def\mD{\mathsf{D}}

\def\mI{\mathsf{I}}

\def\mQ{\mathsf{Q}}

\def\mS{\mathsf{S}}
\def\mT{\mathsf{T}}
\def\mU{\mathsf{U}}
\def\mV{\mathsf{V}}

\newcommand{\wt}{\widetilde}
\newcommand{\wh}{\widehat}

\newcommand{\1}{\mathbbm{1}}            
\newcommand{\set}[1]{\{#1\}}            

\DeclareMathOperator{\dif}{d \!}        

\DeclareMathOperator{\supp}{supp}          

\title{Wiener-Hopf Factorization for Time-Inhomogeneous Markov Chains}

\author{
    Tomasz R. Bielecki \\[-0.3ex]
    \url{tbielecki@iit.edu} \\[-0.9ex]
    \url{http://math.iit.edu/\~bielecki}
 \and
    Ziteng Cheng \\[-0.3ex]
    \url{zcheng7@hawk.iit.edu} \\[-0.9ex]
 \and
    Igor Cialenco \\[-0.3ex]
    \url{cialenco@iit.edu}  \\[-0.9ex]
    \url{http://math.iit.edu/\~igor}
 \and
     Ruoting Gong \\[-0.3ex]
    \url{rgong2@iit.edu}  \\[-0.9ex]
    \url{http://mypages.iit.edu/\~rgong2}\\[-0.9ex]
 \and \\
        {\footnotesize Department of Applied Mathematics, Illinois Institute of Technology} \\
        {\footnotesize W 32nd Str, John T. Rettaliata Engineering Center, Room 208, Chicago, IL 60616, USA}\\
        }

\date{
First Circulated: February 24, 2019 }

\begin{document}

\maketitle




\smallskip

{\footnotesize
\begin{tabular}{l@{} p{350pt}}
  \hline \\[-.2em]
  \textsc{Abstract}: \ & This work contributes to the theory of Wiener-Hopf type factorization for finite Markov chains. This theory originated in the seminal paper \cite{BarlowRogersWilliams1980}, which treated the case of finite time-homogeneous Markov chains. Since then, several works extended the results of \cite{BarlowRogersWilliams1980} in many directions. However, all these extensions were dealing with  time-homogeneous Markov case. The first work dealing with the  time-inhomogeneous situation was \cite{BieCiaGonHua2018}, where Wiener-Hopf type factorization for time-inhomogeneous finite Markov chain with piecewise constant generator matrix function was derived. In the present paper we go further: we derive and study Wiener-Hopf type factorization for time-inhomogeneous finite Markov chain with the generator matrix function being a fairly general matrix valued function of time.
    \\[0.5em]
\textsc{Keywords:} \ &  Time-inhomogeneous finite Markov chain, Markov family, Feller semigroup, time homogenization, Wiener-Hopf factorization. \\
\textsc{MSC2010:} \ & 60J27, 60J28, 60K25 \\[1em]
  \hline
\end{tabular}
}


\section{Introduction}\label{sec:Intro}
The main goal of this paper is to develop a Wiener-Hopf type factorization for finite time-inhomo\-geneous Markov chains. In order to motivate this goal, we first provide a brief account of the Wiener-Hopf factorization for time-homogeneous Markov chains based on \cite{BarlowRogersWilliams1980}.

Towards this end, consider a finite state space $\mathbf{E}$ with cardinality $m$, and let $\mathsf{\Lambda}$ be a sub-Markovian generator matrix of dimension $m\times m$, that is, $\mathsf{\Lambda}(i,j)\geq 0$, $i\neq j$, and $\sum_{j\in\mathbf{E}}\mathsf{\Lambda}(i,j)\leq 0$. Next, let $v$ be a real valued function on $\mathbf{E}$, such that $v(i)\neq 0$ for all $i\in\mathbf{E}$, and define
\begin{align*}
\mathbf{E}_{\pm}:=\left\{i\in\mathbf{E}:\,\pm\,v(i)>0\right\}.
\end{align*}
We also denote by $m_{\pm}$ cardinality of $\mathbf{E}_{\pm}$, and we let $\mV:=\textrm{diag}\{v(i)\,:\,i\in\mathbf{E}\}$ be the diagonal matrix of dimension $m\times m$. Finally, let $\mI$ and $\mI^{\pm}$ denote the identity matrices of dimensions $m\times m$ and $m_{\pm}\times m_{\pm}$, respectively. Using probabilistic methods, the following result was proved in \cite{BarlowRogersWilliams1980}.
\begin{theorem}[{\cite[Theorem I]{BarlowRogersWilliams1980}}]\label{thm:BRW80WH1}
For any $c>0$, there exists a unique pair of matrices $(\mathsf{\Pi}^{+}_{c},\mathsf{\Pi}^{-}_{c})$ of dimensions $m_{-}\times m_{+}$ and $m_{+}\times m_{-}$m respectively, such that the matrix
\begin{align*}
\mS=\begin{pmatrix} \mI^{+} & \mathsf{\Pi}^{-}_{c} \\ \mathsf{\Pi}^{+}_{c} & \mI^{-} \end{pmatrix}
\end{align*}
is invertible and the following factorization holds true
\begin{align}\label{eq:TimeHomoWH}
\mV^{-1}(\mathsf{\Lambda}-c\,\mI)=\mS\begin{pmatrix} \mQ^{+}_{c} & 0 \\ 0 & \mQ^{-}_{c} \end{pmatrix}\mS^{-1},
\end{align}
where $\mQ^{\pm}_{c}$ are $m_{\pm}\times m_{\pm}$ sub-Markovian generator matrices. Moreover, $\mathsf{\Pi}^{\pm}_{c}$ are strictly substochastic.
\end{theorem}

The right-hand side of \eqref{eq:TimeHomoWH} is said to constitute the Wiener-Hopf factorization of the matrix $\mV^{-1}(\mQ-c\mI)$. While the factorization \eqref{eq:TimeHomoWH} is algebraic in its nature, it admits a very important probabilistic interpretation, which leads to very efficient computation of some useful expectations. More precisely, let $X$ be a time-homogeneous Markov chain taking values in $\mathbf{E}\cup\partial$, where $\partial$ is a coffin state, with generator $\mathsf{\Lambda}$. For $t\geq 0$, we define the additive functional
\begin{align*}
\phi(t):=\int_{0}^{t}v(X_{u})\,du,
\end{align*}
and two stopping times
\begin{align*}
\tau^{\pm}_{t}:=\inf\left\{u\geq 0:\,\pm\,\phi(u)>t\right\}.
\end{align*}

\begin{theorem}[{\cite[Theorem II]{BarlowRogersWilliams1980}}]\label{thm:BRW80WH2}
For any $i\in\mathbf{E}_{\mp}$ and $j\in\mathbf{E}_{\pm}$,
\begin{align}\label{eq:PicPlusMinus}
\bE\bigg(e^{-c\,\tau_{0}^{\pm}}\1_{\big\{X_{\tau_{0}^{\pm}}=j\big\}}\,\Big|\,X_{0}=i\bigg)=\mathsf{\Pi}_{c}^{\pm}(i,j).
\end{align}
For any $i,j\in\mathbf{E}_{\pm}$ and $t\geq 0$,
\begin{align}\label{eq:QcPlusMinus}
\bE\bigg(e^{-c\,\tau_{t}^{\pm}}\1_{\big\{X_{\tau_{t}^{\pm}}=j\big\}}\,\Big|\,X_{0}=i\bigg)=e^{t\,\mQ_{c}^{\pm}}(i,j).
\end{align}
\end{theorem}

Both Theorems \ref{thm:BRW80WH1} and \ref{thm:BRW80WH2} have been studied for more general classes of Markov process, as well as for various types of stopping times, that naturally occur in applications {(cf. \cite{KennedyWilliams1990}, \cite{AvramPistoriusUsabel2003}, \cite{Williams2008}, \cite{MijatovicPistorius2011}, and references therein)}. However, in all these studies the Markov processes have been assumed to be time-homogeneous.

As it turns out, the time-inhomogeneous case is more intricate, and direct (naive) generalizations or applications of the time-homogenous case to the non-homogenous case can not be done in principle. Specifically, let now $X$ be a finite state time-inhomogeneous Markov chain taking values in $\mathbf{E}\cup\partial$, with generator function $\mathsf{\Lambda}_{s}$, $s\geq 0$. The first observation that one needs to make is that the Wiener-Hopf factorization of the matrix $\mV^{-1}(\mathsf{\Lambda}_{s}-c\mI)$ can be done for each $s\geq 0$ separately, exactly as described in Theorem \ref{thm:BRW80WH1}. However, the resulting matrices $\mathsf{\Pi}^{\pm}_{c}(s)$ and $\mQ^{\pm}_{c}(s)$, $s\geq 0$, are not useful for computing the expectations of the form
\begin{align*}
\bE\bigg(e^{-c\,\tau^{\pm}_{t}(s)}\1_{\big\{X_{\tau^{\pm}_{t}(s)}=j\big\}}\,\Big|\,X_{s}=i\bigg),
\end{align*}
where
\begin{align*}
\tau^{\pm}_{t}(s):=\inf\left\{u\geq s:\,\pm\int_{s}^{s+u}v(X_{r})\,dr>t\right\}.
\end{align*}
This makes the study of the time-inhomogeneous case a highly nontrivial and novel enterprise. As it will be seen from the discussion presented below, an entirely new theory needs to be put forth for this purpose. The research effort in this direction has been originated in \cite{BieCiaGonHua2018}. This work contributes to the continuation of  the research endeavor in this direction.


\section{Setup and the main goal of the paper}\label{sec:Setup}

\subsection{Preliminaries}\label{subsec:Notations}

Throughout this paper we let $\mathbf{E}$ be a finite set, with $|\mathbf{E}|=m>1$. We define $\overline{\mathbf{E}}:=\mathbf{E}\cup\{\partial\}$, where $\partial$ denotes the coffin state isolated from $\mathbf{E}$. Let $(\mathsf{\Lambda}_{s})_{s\in\bR_{+}}$, where $\bR_{+}:=[0,\infty)$, be a family of $m\times m$ generator matrices, i.e., their off-diagonal elements are non-negative, and the entries in their rows sum to zero. We additionally define $\mathsf{\Lambda}_{\infty}:=\mathsf{0}$, the $m\times m$ matrix with all entries equal to zero.

\medskip
We make the following standing assumption:
\begin{assumption}\label{assump:GenLambda}\mbox{}
\vspace{-0.5em}
\begin{itemize}
\item [(i)] There exists a universal constant $K\in(0,\infty)$, such that $|\mathsf{\Lambda}_{s}(i,j)|\leq K$, for all $i,j\in\mathbf{E}$ and $s\in\bR_{+}$.
\item [(ii)] $(\mathsf{\Lambda}_{s})_{s\in\bR_{+}}$, considered as a mapping from $\bR_{+}$ to the set of $m\times m$ generator matrices,  is continuous with respect to $s$.
\end{itemize}
\end{assumption}
Let $v:\overline{\mathbf{E}}\rightarrow\bR$ with $v(i)\neq 0$ for any $i\in\mathbf{E}$ and $v(\partial)=0$, $\mV:=\text{diag}\{v(i):i\in\mathbf{E}\}$, $\overline{v}:=\max_{i\in\mathbf{E}}|v(i)|$, and $\underline{v}:=\min_{i\in\mathbf{E}}|v(i)|$. We will use the following partition of the set $\mathbf{E}$
\begin{align*}
\mathbf{E}_{+}:=\left\{i\in\mathbf{E}:\,v(i)>0\right\}\quad\text{and}\quad\mathbf{E}_{-}:=\left\{i\in\mathbf{E}:\,v(i)<0\right\}.
\end{align*}
We assume that both $\mathbf{E}_{+}$ and $\mathbf{E}_{-}$ are non-empty, and that the indices of the first $m_+=|\mathbf{E}_{+}|$ (respectively, last $m_-=|\mathbf{E}_{-}|$) rows and columns of any $m\times m$ matrix correspond to the elements in $\mathbf{E}_{+}$ (respectively, $\mathbf{E}_{-}$). Accordingly, we write $\mathsf{\Lambda}_{s}$ and $\mV$ in the block form
\begin{align}\label{eq:MatrixBlocks}
\mathsf{\Lambda}_{s}=\bordermatrix{~ & \mathbf{E}_{+} & \mathbf{E}_{-} \cr \mathbf{E}_{+} & \mA_{s} & \mB_{s} \cr \mathbf{E}_{-} & \mC_{s} & \mD_{s} \cr},\quad\mathsf{\mV}=\bordermatrix{~ & \mathbf{E}_{+} & \mathbf{E}_{-} \cr \mathbf{E}_{+} & \mV_{+} & \mathsf{0} \cr \mathbf{E}_{-} & \mathsf{0} & \mV_{-} \cr}.
\end{align}

In what follows we let $\sX:=\bR_{+}\times\mathbf{E}$, and  $\sX_{\pm}:=\bR_{+}\times\mathbf{E}_{\pm}$). The Borel $\sigma$-field on $\sX$ (respectively, $\sX_{\pm}$) is denoted by $\cB(\sX):=\cB(\bR_{+})\otimes 2^{\mathbf{E}}$ (respectively, $\cB(\sX_{\pm}):=\cB(\bR_{+})\otimes 2^{\mathbf{E}_{\pm}}$). Accordingly, we let $\overline{\sX}:=\sX\cup(+\infty,\partial)$ (respectively, $\overline{\sX_{\pm}}:=\sX_{\pm}\cup(+\infty,\partial)$) be the one-point completion of $\sX$ (respectively, $\sX_{\pm}$), and let $\cB(\overline{\sX}):=\sigma(\cB(\sX)\cup\{(\infty,\partial)\})$ (respectively, $\cB(\overline{\sX_{\pm}}):=\sigma(\cB(\sX_{\pm})\cup\{(\infty,\partial)\})$). A pair $(s,i)\in \sX$ consists of the time variable $t$ and the space variable $s$.

\medskip
We will also use the following notations for various spaces of real-valued functions:
\begin{itemize}
\item $L^{\infty}(\overline{\sX})$ is the space of $\cB(\overline{\sX})$-measurable, and bounded functions $f$ on $\overline{\sX}$, with $g(+\infty,\partial)=0$.
\item $C_{0}(\overline{\sX})$ is the space of functions $f\in L^{\infty}(\overline{\sX})$ such that $f(\cdot,i)\in C_{0}(\bR_{+})$ for all $i\in\mathbf{E}$, where $C_{0}(\bR_{+})$ is the space {of} functions vanishing at infinity.
\item $C_{c}(\overline{\sX})$ is the space of functions $f\in L^{\infty}(\overline{\sX})$ such that $f(\cdot,i)\in C_{c}(\bR_{+})$ for all $i\in\mathbf{E}$, where $C_{c}(\bR_{+})$ is the space of functions with compact support.
\item $C_{0}^{1}(\overline{\sX})$ is the space of functions $f\in C_{0}(\overline{\sX})$ such that, for any $i\in\mathbf{E}$, $\partial f(\cdot,i)/\partial s$ exists and belongs to $C_{0}(\bR_{+})$ ({for convenience,} we stipulate that $\partial f(\infty,\partial)/\partial s=0$).
\item $C_{c}^{1}(\overline{\sX})$ is the space of functions $f\in C_{c}(\overline{\sX})$ such that, for any $i\in\mathbf{E}$, $\partial f(\cdot,i)/\partial s$ exists ({for convenience,} we stipulate that $\partial f(\infty,\partial)/\partial s=0$).
\end{itemize}

Sometimes $\overline{\sX}$ will be replaced by $\overline{\sX_{+}}$ or $\overline{\sX_{-}}$ when the functions are defined on these spaces, in which case the set $\mathbf{E}$ will be replaced by $\mathbf{E}_{+}$ or $\mathbf{E}_{-}$, respectively, in the above definitions.  Note that each function on $\overline{\sX}$ can be viewed as a time-dependent vector of size $m$, which can be split into a time-dependent vector of size $m_{+}$ (a function on $\sX_{+}$) and a time-dependent vector of size $m_{-}$ (a function on $\sX_{-}$).

\medskip
We conclude this section by introducing some more notations, this time for operators:
\begin{itemize}
\item $\wt{\mathsf{\Lambda}}:L^{\infty}(\overline{\sX})\rightarrow L^{\infty}(\overline{\sX})$ is the multiplication operator associated with $(\mathsf{\Lambda}_{s})_{s\in\bR_{+}}$, defined by
    \begin{align}\label{eq:DefTildeLambda}
    (\wt{\mathsf{\Lambda}}\,g)(s,i):=(\mathsf{\Lambda}_{s}\,g(s,\cdot))(i),\quad (s,i)\in\overline{\sX}.
    \end{align}
\item Similarly, we define multiplication operators $\wt{\mA}:L^{\infty}(\overline{\sX_{+}})\rightarrow L^{\infty}(\overline{\sX_{+}})$, $\wt{\mB}:L^{\infty}(\overline{\sX_{-}})\rightarrow L^{\infty}(\overline{\sX_{+}})$, $\wt{\mC}:L^{\infty}(\overline{\sX_{+}})\rightarrow L^{\infty}(\overline{\sX_{-}})$, and $\wt{\mD}:L^{\infty}(\overline{\sX_{-}})\rightarrow L^{\infty}(\overline{\sX_{-}})$, associated with the blocks $(\mA_{s})_{s\in\bR_{+}}$, $(\mB_{s})_{s\in\bR_{+}}$, $(\mC_{s})_{s\in\bR_{+}}$, and $(\mD_{s})_{s\in\bR_{+}}$ given in \eqref{eq:MatrixBlocks}, respectively.
\end{itemize}
Given the above, for any\footnote{The superscript $T$ will be used to denote the transpose of a vector or matrix.} $g=(g^{+},g^{-})^{T}\in L^{\infty}(\overline{\sX})$, where $g^{\pm}\in L^{\infty}(\overline{\sX_{\pm}})$, we have
\begin{align}\label{eq:TildeLambdaBlocks}
\wt{\mathsf{\Lambda}}\,g=\begin{pmatrix} \wt{\mA}\,g^{+}+\wt{\mB}\,g^{-} \\ \wt{\mC}\,g^{+}+\wt{\mD}\,g^{-} \end{pmatrix}.
\end{align}

\subsection{A time-inhomogeneous Markov family corresponding to  $(\mathsf{\Lambda}_{s})_{s\in\bR_{+}}$ and related passage times}\label{subsec:Markov}

We start with introducing a time-inhomogeneous Markov Family corresponding to $(\mathsf{\Lambda}_{s})_{s\in\bR_{+}}$. Then, we proceed with a study of some passage times related to this family.

\subsubsection{A time-inhomogeneous Markov family $\cM^{*}$ corresponding to $(\mathsf{\Lambda}_{s})_{s\in\bR_{+}}$}\label{subsubsec:Markov}

We take $\Omega^{*}$ as the collection of $\mathbf{E}$-valued functions $\omega^{*}$ on $\bR_{+}$, and $\sF^{*}:=\sigma\{X^{*}_{t},\,t\in\bR_{+}\}$, where $X$ is the coordinate mapping $X^{*}_\cdot(\omega^{*}):=\omega^{*}(\cdot)$. Sometimes we may need the value of $\omega^{*}\in\Omega^{*}$ at infinity, and in such case we set $X_{\infty}^{*}(\omega^{*})=\omega^{*}(\infty)=\partial$, for any $\omega^{*}\in\Omega^{*}$. We endow the space $(\Omega^{*},\sF^{*})$ with a family of filtrations $\bF_{s}^{*}:=\{\sF^{s,*}_{t},\,t\in[s,\infty]\}$, $s\in\overline{\bR}_{+}$, where, for $s\in\bR_{+}$,
\begin{align*}
\sF^{s,*}_{t}:=\bigcap_{r>t}\sigma\left(X^{*}_{u},\,\,u\in[s,r]\right),\,\,\,t\in[s,\infty);\quad\sF^{s,*}_{\infty}:=\sigma\bigg(\bigcup_{t\geq s}\sF^{s,*}_{t}\bigg),
\end{align*}
and $\sF^{\infty,*}_{\infty}:=\{\emptyset,\Omega^{*}\}$. We denote by
\begin{align*}
\mathcal{M}^{*}:=\big\{\big(\Omega^{*},\sF^{*},\bF_{s}^{*},(X^{*}_t)_{t\in[s,\infty]},\bP_{s,i}^{*}\big),\,(s,i)\in\overline{\sX}\big\}
\end{align*}
a canonical {\it time-inhomogeneous} Markov family. That is,
\begin{itemize}
\item $\bP^{*}_{s,i}$ is a probability measure on $(\Omega^{*},\sF^{s,*}_{\infty})$ for $(s,i)\in\overline{\sX}$;
\item the function $P^{*}:\overline{\sX}\times\overline{\bR}_{+}\times 2^{\overline{\mathbf{E}}}\rightarrow [0,1]$ defined for $0\leq s\leq t\leq\infty$ as
    \begin{align*}
    P^{*}(s,i,t,B):=\bP^{*}_{s,i}\!\left(X^{*}_t\in B\right)
    \end{align*}
    is measurable with respect to $i$ for any fixed $s\leq t$ and $B\in 2^{\overline{\mathbf{E}}}$;
\item $\bP^{*}_{s,i}(X^{*}_{s}=i)=1$ for any $(s,i)\in\overline{\sX}$;
\item for any $(s,i)\in\overline{\sX}$, $s\leq t\leq r\leq\infty$, and $B\in 2^{\overline{\mathbf{E}}}$, it holds that
    \begin{align*}
    \bP^{*}_{s,i}\!\left(X^{*}_r\in B\,|\,\sF^{s,*}_{t}\right)=\bP^{*}_{t,X^{*}_{t}}\!\left(X^{*}_r\in B\right),\quad\bP^{*}_{s,i}-\text{a.s.}\,.
    \end{align*}
\end{itemize}

Let $\mU^{*}:=(\mU_{s,t}^{*})_{0\leq s\leq t<\infty}$ be the evolution system (cf. \cite{Bottcher2014}) corresponding to $\mathcal{M}^{*}$ defined by
\begin{align}\label{eq:DefEvolSytXStar}
\mU_{s,t}^{*}f(i):=\bE_{s,i}^{*}\left(f(X_{t}^{*})\right),\quad 0\leq s\leq t<\infty,\quad i\in\mathbf{E},
\end{align}
for all functions (column vectors) $f:\mathbf{E}\rightarrow\bR$.\footnote{Note that for $t\in\bR_{+}$, $X_{t}^{*}$ takes values in $\mathbf{E}$.} We assume that
\begin{align}\label{eq:DefGenXStar}
\lim_{h\downarrow 0}\frac{1}{h}\left(\mU_{s,s+h}^{*}f(i)-f(i)\right)=\mathsf{\Lambda}_{s}f(i),\quad\text{for any }\,(s,i)\in\sX,
\end{align}
for all $f:\mathbf{E}\rightarrow\bR$.

It is well known that a standard version of the Markov family $\mathcal{M}^{*}$ (cf. \cite[Definition I.6.6]{GikhmanSkorokhod2004}) can be constructed. This is done by first constructing via Peano-Baker series the evolution system $\mU^{*}=(\mU^{*}_{s,t})_{0\leq s\leq t<\infty}$ that solves
\begin{align}\label{eq:EvolDE}
\frac{\dif\mU^{*}_{s,t}}{\dif t}=\mathsf{\Lambda}_{t}\,\mU^{*}_{s,t},\quad\mU^{*}_{s,s}=\mI,\quad 0\leq s\leq t<\infty.
\end{align}
Since $\mathsf{\Lambda}_{t}$ is a generator matrix, $\mU^{*}_{s,t}$ is positive preserving and contracting with $\mU^{*}_{s,t}\mathsf{1}_{m}=\mathsf{1}_{m}$. In addition, due to Assumption \ref{assump:GenLambda}-(i) and \eqref{eq:EvolDE}, it holds for any $0\leq s<t$ and $r\in(0,t-s)$ that
\begin{align*}
\left\|\mU^{*}_{s,t+r}-\mU^{*}_{s,t}\right\|_{\infty}=\left\|\mU^{*}_{s,t}\left(\mU^{*}_{t,t+r}-\mI\right)\right\|_{\infty}\leq Cr,
\end{align*}
and
\begin{align*}
\left\|\mU^{*}_{s+r,t}-\mU^{*}_{s,t}\right\|_{\infty}=\left\|\left(\mI-\mU^{*}_{s,s+r}\right)\mU^{*}_{s+r,t}\right\|_{\infty}\leq Cr,
\end{align*}
for some positive constant $C$, so that $\mU^{*}_{s,t}$ is strongly continuous in $s$ and $t$. The above, together with the finiteness of the state space, implies that $\mU^*$ is a Feller evolution system. The corresponding standard version can then be constructed (cf. \cite[Theorem I.6.3]{GikhmanSkorokhod2004}).

In view of the above, we will consider the standard version of $\cM^{*}$ in what follows, and, for simplicity, we will preserve the notation $\cM^{*}=\set{(\Omega^{*},\sF^{*},\bF_{s}^{*},(X^{*}_t)_{t\ge s},\bP_{s,i}^{*})$, $(s,i)\in\overline{\sX}}$, in which $\Omega^{*}$ is restricted to the collection of $\mathbf{E}$-valued c\`{a}dl\`{a}g functions $\omega^{*}$ on $\bR_{+}$ with $\omega^{*}(\infty)=\partial$.

\subsubsection{Passage times related to $\cM^{*}$}\label{subsubsec:PassageTime}

For any $s\in\overline{\bR}_{+}$, we define an additive functional $\phi_{\cdot}^{*}(s)$ as
\begin{align*}
\phi_{t}^{*}(s):=\int_{s}^{t}v(X_{u}^{*})\,du,\quad t\in[s,\infty],
\end{align*}
and we stipulate $\phi_{\infty}^{*}(s,\omega^{*})=\infty$ for every $\omega^{*}\in\Omega^{*}$. In addition, for any $s\in\overline{\bR}_{+}$ and $\ell\in\bR_{+}$, we define associated passage times
\begin{align*}
\tau_{\ell}^{+,*}(s):=\inf\left\{t\in[s,\infty]:\,\phi_{t}^{*}(s)>\ell\right\}\quad\text{and}
\quad\tau_{\ell}^{-,*}(s):=\inf\left\{t\in[s,\infty]:\,\phi_{t}^{*}(s)<-\ell\right\}.
\end{align*}
Both $\tau_{\ell}^{+,*}(s)$ and $\tau_{\ell}^{-,*}(s)$ are $\bF_{s}^{*}$-stopping times since, $\phi_{\cdot}^{*}(s)$ is $\bF^{*}_{s}$-adapted, has continuous sample paths, and $\bF_{s}^{*}$ is right-continuous (cf. \cite[Proposition 1.28]{JacodShiryaev2003}). For notational convenience, if no confusion arises, we will omit the parameter $s$ in $\phi_{t}^{*}(s)$ and $\tau_{\ell}^{\pm,*}(s)$.

The following result will be used later in the paper.

\begin{lemma}\label{lem:RangeXTaupm}
For any $s\in\overline{\bR}_{+}$, $\ell\in\bR_{+}$, and $\omega^{*}\in\Omega^{*}$, $X_{\tau_{\ell}^{\pm,*}(s)}^{*}(\omega^{*})\in\mathbf{E}_{\pm}\cup\{\partial\}$. In particular, if $\tau_{\ell}^{\pm,*}(s,\omega^{*})<\infty$, then $X_{\tau_{\ell}^{\pm,*}(s)}^{*}(\omega^{*})\in\mathbf{E}_{\pm}$.
\end{lemma}

\begin{proof}
We will only prove the ``+" version of the lemma; a proof of the ``$-$" version proceeds in an analogous way.

To begin with, for any $\omega^{*}\in\Omega^{*}$ such that $\tau_{\ell}^{\pm,*}(s,\omega^{*})=\infty$, clearly we have $X_{\tau_{\ell}^{\pm,*}(s)}^{*}(\omega^{*})=X_{\infty}^{*}(\omega^{*})=\omega^{*}(\infty)=\partial$. Next, suppose that for some $\omega_{0}^{*}\in\Omega^{*}$ and $s_{0},\ell_{0}\in\bR_{+}$, $\tau_{\ell_{0}}^{+,*}(s_{0},\omega_{0}^{*})<\infty$ and $X_{\tau_{\ell_{0}}^{+,*}(s_{0})}^{*}(\omega_{0}^{*})\in\mathbf{E}_{-}$. By the definition of $\phi^{*}$, we have $\phi^{*}_{\tau^{+,*}_{\ell_{0}}(s_{0})}(\omega_{0}^{*})=\ell_{0}$. Moreover, since $X_{\cdot}(\omega_{0}^{*})$ is right-continuous, there exists $\varepsilon_{0}>0$, such that for any $t\in[\tau_{\ell_{0}}^{+,*}(s_{0},\omega_{0}^{*}),\tau_{\ell_{0}}^{+,*}(s_{0},\omega_{0}^{*})+\varepsilon]$, $X_{t}^{*}(\omega_{0}^{*})\in\mathbf{E}_{-}$. Hence, for any $t\in[\tau_{\ell_{0}}^{+,*}(s_{0},\omega_{0}^{*}),\tau_{\ell_{0}}^{+,*}(s_{0},\omega_{0}^{*})+\varepsilon]$, $\phi^{*}_{t}(\omega_{0}^{*})\leq\ell_{0}$, which contradicts the definition of $\tau_{\ell}^{+,*}$.
\end{proof}

\begin{remark}
Here is an example where $\{\tau^{+,*}_{\ell}=\infty\}$ has a positive probability. Consider $\mathbf{E}=\{1,-1\}$, $v(\pm 1)=\pm 1$, and
\begin{align*}
\mathsf{\Lambda}_{s}=\bordermatrix{~ & +1 & -1 \cr +1 & -1 & 1 \cr -1 & 0 & 0 \cr},\quad s\in\bR_{+}.
\end{align*}
Then, for any $s\in\bR_{+}$ and $\ell>0$,
\begin{align*}
\bP^{*}_{s,1}\left(\tau^{+,*}_{\ell}=\infty\right)=\bP^{*}_{s,1}\left(\left\{\text{the first jump of $X^{*}$ occurs before time $\ell$}\right\}\right)=1-e^{-\ell}>0.
\end{align*}
\end{remark}

\subsection{The main goal of the paper}\label{subsec:MainGoal}

Our main interest is to derive a Wiener-Hopf type method for computing expectations of the following form
\begin{align}\label{eq:ExpTauXTau}
\bE^{*}_{s,i}\Big(g^{\pm}\Big(\tau_{\ell}^{\pm,*},X_{\tau_{\ell}^{\pm,*}}^{*}\Big)\Big)
\end{align}
for $g^{\pm}\in L^{\infty}(\overline{\sX_{\pm}})$, $\ell\in\bR_{+}$, and $(s,i)\in\sX$. In view of Lemma \ref{lem:RangeXTaupm}, it is enough to compute the expectation in \eqref{eq:ExpTauXTau} for $g^{\pm}\in L^{\infty}(\overline{\sX_{\pm}})$ in order to compute the analogous expectation for $g\in L^{\infty}(\overline{\sX})$.

The Wiener-Hopf type method derived in this paper generalizes the Wiener-Hopf type method of \cite{BarlowRogersWilliams1980} that was developed for the time-homogeneous Markov chains.

\begin{remark}
The time-homogeneous version of the problem of computing the expectation of the type given in \eqref{eq:ExpTauXTau} appears frequently in time-homogeneous fluid models (see e.g. \cite{Rogers1994} and the references therein). Time inhomogeneous extensions of such models is important and natural due to temporal (seasonal) effect, for example. This is one practical motivation for the study presented in this paper.
\end{remark}

In order to proceed, we introduce the following operators:
\begin{itemize}
\item $J^{+}:L^{\infty}(\overline{\sX_{+}})\rightarrow L^{\infty}(\overline{\sX_{-}})$ is defined as
    \begin{align}\label{eq:DefJPlus}
    \big(J^{+}g^{+}\big)(s,i):=\bE_{s,i}^{*}\Big(g^{+}\Big(\tau_{0}^{+,*},X_{\tau_{0}^{+,*}}^{*}\Big)\Big),\quad (s,i)\in\overline{\sX_{-}}.
    \end{align}
    Clearly, for any $g^{+}\in L^{\infty}(\overline{\sX_{+}})$, $|(J^{+}g^{+})(s,i)|\leq\|g^{+}\|_{L^{\infty}(\overline{\sX_{+}})}<\infty$ for any $(s,i)\in\sX_{-}$, and $(J^{+}g^{+})(\infty,\partial)=0$, so that $J^{+}g^{+}\in L^{\infty}(\overline{\sX_{-}})$.
\item $J^{-}:L^{\infty}(\overline{\sX_{-}})\rightarrow L^{\infty}(\overline{\sX_{+}})$ is defined as,
    \begin{align}\label{eq:DefJMinus}
    \big(J^{-}g^{-}\big)(s,i):=\bE_{s,i}^{*}\Big(g^{-}\Big(\tau_{0}^{-,*},X_{\tau_{0}^{-,*}}^{*}\Big)\Big),\quad (s,i)\in\overline{\sX_{+}}.
    \end{align}
\item For any $\ell\in\bR_{+}$, $\cP_{\ell}^{+}:L^{\infty}(\overline{\sX_{+}})\rightarrow L^{\infty}(\overline{\sX_{+}})$ is defined as
    \begin{align}\label{eq:DefPPlus}
    \big(\cP^{+}_{\ell}g^{+}\big)(s,i):=\bE_{s,i}^{*}\Big(g^{+}\Big(\tau_{\ell}^{+,*},X_{\tau_{\ell}^{+,*}}^{*}\Big)\Big),\quad (s,i)\in\overline{\sX_{+}}.
    \end{align}
\item For any $\ell\in\bR_{+}$, $\cP_{\ell}^{-}:L^{\infty}(\overline{\sX_{-}})\rightarrow L^{\infty}(\overline{\sX_{-}})$ is defined as,
    \begin{align}\label{eq:DefPMinus}
    \big(\cP^{-}_{\ell}g^{-}\big)(s,i):=\bE_{s,i}^{*}\Big(g^{-}\Big(\tau_{\ell}^{-,*},X_{\tau_{\ell}^{-,*}}^{*}\Big)\Big),\quad (s,i)\in\overline{\sX_{-}}.
    \end{align}
\item For any $(s,i)\in\overline{\sX_{+}}$, we define
    \begin{align}\label{eq:DefGenGPlus}
    \big(G^{+}g^{+}\big)(s,i):=\lim_{\ell\rightarrow 0+}\frac{1}{\ell}\big(\cP^{+}_{\ell}g^{+}(s,i)-g^{+}(s,i)\big),
    \end{align}
    for any $g^{+}\in C_{0}(\overline{\sX_{+}})$ such that the limit in \eqref{eq:DefGenGPlus} exists and is finite.
\item For any $(s,i)\in\overline{\sX_{-}}$, we define
    \begin{align}\label{eq:DefGenGMinus}
    \big(G^{-}g^{-}\big)(s,i):=\lim_{\ell\rightarrow 0+}\frac{1}{\ell}\big(\cP^{-}_{\ell}g^{-}(s,i)-g^{-}(s,i)\big),
    \end{align}
    for all $g^{-}\in C_{0}(\overline{\sX_{+}})$ such that the above limit in \eqref{eq:DefGenGMinus} exists and is finite.
\end{itemize}

\begin{remark}\label{rem:Expgpmsimp}
For $g^{+}\in L^{\infty}(\overline{\sX_{+}})$, $\ell\in(0,\infty)$, and $(s,i)\in\sX_{-}$, it can be shown that
\begin{align}\label{eq:ExpgpsimJPPlus}
\bE^{*}_{s,i}\Big(g^{+}\Big(\tau_{\ell}^{+,*},X_{\tau_{\ell}^{+,*}}^{*}\Big)\Big)=\big(J^{+}\cP_{\ell}^{+}g^{+}\big)(s,i).
\end{align}
Similarly, for $g^{-}\in L^{\infty}(\overline{\sX_{-}})$, $\ell\in(0,\infty)$, and $(s,i)\in\sX_{+}$, we have
\begin{align}\label{eq:ExpgmsipJPMinus}
\bE^{*}_{s,i}\Big(g^{-}\Big(\tau_{\ell}^{-,*},X_{\tau_{\ell}^{-,*}}^{*}\Big)\Big)=\big(J^{-}\cP_{\ell}^{-}g^{-}\big)(s,i).
\end{align}
The identity \eqref{eq:ExpgpsimJPPlus} will be verified in Remark \ref{rem:ProofExpgpsimJPMinus} below, while \eqref{eq:ExpgmsipJPMinus} can be proved in an analogous way with $v$ replaced by $-v$.

In view of \eqref{eq:DefJPlus}$-$\eqref{eq:DefPMinus} and \eqref{eq:ExpgpsimJPPlus}$-$\eqref{eq:ExpgmsipJPMinus}, the expectation of the form \eqref{eq:ExpTauXTau} for any $g^{\pm}\in L^{\infty}(\overline{\sX_{\pm}})$, $\ell\in\bR_{+}$, and $(s,i)\in\sX$, can be represented in terms of the operators $J^{\pm}$ and $\cP_{\ell}^{\pm}$.
\end{remark}

\section{Main Results}\label{sec:MainResults}

We now state the main results of this paper, Theorem \ref{thm:WHExistUniq} and Theorem \ref{thm:WHProbInterpr}. Theorem \ref{thm:WHExistUniq} is analytical in nature, and it provides the {\it Wiener-Hopf factorization} for the generator $\mV^{-1}(\partial/\partial s+\wt{\mathsf{\Lambda}})$. This factorization is given in terms of operators $(S^{+},H^{+},S^{-},H^{-})$  showing in the statement of the theorem. Theorem \ref{thm:WHProbInterpr} is probabilistic in nature, and provides the probabilistic interpretation of the operators $(S^{+},H^{+},S^{-},H^{-})$, which is key for various applications of our  Wiener-Hopf factorization.

\begin{theorem}\label{thm:WHExistUniq}
Let $(\mathsf{\Lambda}_{s})_{s\in\bR_{+}}$ be a family of $m\times m$ generator matrices satisfying Assumption \ref{assump:GenLambda}, and let $\wt{\mathsf{\Lambda}}$ be the associated multiplication operator defined as in \eqref{eq:DefTildeLambda}. Let $v:\overline{\mathbf{E}}\rightarrow\bR$ with $v(i)\neq 0$ for any $i\in\mathbf{E}$, $v(\partial)=0$, and $\mV=\textnormal{diag}\,\{v(i):i\in\mathbf{E}\}$. Then, there exists a unique quadruple of operators $(S^{+},H^{+},S^{-},H^{-})$ which solves the following operator equation
\begin{align}\label{eq:WH}
\mV^{-1}\bigg(\frac{\partial}{\partial s}+\wt{\mathsf{\Lambda}}\bigg) \begin{pmatrix} I^{+} & S^{-} \\ S^{+} & I^{-} \end{pmatrix} \begin{pmatrix} g^{+} \\ g^{-} \end{pmatrix} = \begin{pmatrix} I^{+} & S^{-} \\ S^{+} & I^{-} \end{pmatrix} \begin{pmatrix} H^{+} & \mathsf{0} \\ \mathsf{0} & -H^{-} \end{pmatrix} \begin{pmatrix} g^{+} \\ g^{-} \end{pmatrix},\quad g^{\pm}\in C_{0}^{1}(\overline{\sX_{\pm}}),
\end{align}
subject to the conditions below:
\begin{itemize}
\item [$(\textnormal{a}^{\pm})$] $S^{\pm}:C_{0}(\overline{\sX_{\pm}})\rightarrow C_{0}(\overline{\sX_{\mp}})$ is a bounded operator such that
    \begin{itemize}
    \item[(i)] for any $g^{\pm}\in C_{c}(\overline{\sX_{\pm}})$ with $\supp g^{\pm}\subset[0,\eta_{g^{\pm}}]\times\mathbf{E}_{\pm}$ for some constant $\eta_{g^{\pm}}\in(0,\infty)$, we have $\supp S^{\pm}g^{\pm}\subset[0,\eta_{g^{\pm}}]\times\mathbf{E}_{\mp}$;
    \item[(ii)] for any $g^{\pm}\in C_{0}^{1}(\overline{\sX_{\pm}})$, we have $S^{\pm}g^{\pm}\in C_{0}^{1}(\overline{\sX_{\mp}})$.
    \end{itemize}
\item [$(\textnormal{b}^{\pm})$] $H^{\pm}$ is the strong generator of a strongly continuous positive contraction semigroup $(\cQ_{\ell}^{\pm})_{\ell\in\bR_{+}}$ on $C_{0}(\overline{\sX_{\pm}})$ with domain $\sD(H^{\pm})=C_{0}^{1}(\overline{\sX_{\pm}})$.
\end{itemize}
\end{theorem}

\begin{theorem}\label{thm:WHProbInterpr}
For any $g^{\pm}\in C_{0}(\overline{\sX_{\pm}})$, we have
\begin{align*}
S^{\pm}g^{\pm}=J^{\pm}g^{\pm}\quad\text{and}\quad\cQ_{\ell}^{\pm}g^{\pm}=\cP_{\ell}^{\pm}g^{\pm},\quad\text{for any }\,\ell\in\bR_{+},
\end{align*}
where $J^{\pm}$ and $(\cP_{\ell}^{\pm})_{\ell\in\bR_{+}}$ are defined in \eqref{eq:DefJPlus}$-$\eqref{eq:DefPMinus}. Moreover, $G^+$ given in \eqref{eq:DefGenGPlus} is the (strong) generator of $(\cP^{+}_{\ell})_{\ell\geq 0}$ with $\cD(G^+) = C_0^1(\overline{\sX}_{+})$, and  $G^-$ given in \eqref{eq:DefGenGMinus} is the (strong) generator of $(\cP^{-}_{\ell})_{\ell\geq 0}$ with $\cD(G^-) = C_0^1(\overline{\sX}_{-})$.
\end{theorem}
The proofs of these two theorems is deferred to Section~\ref{sec:Proofs}.

By Theorems \ref{thm:WHExistUniq} and \ref{thm:WHProbInterpr}, we are able to compute $J^{\pm}g^{\pm}$ and $\cP_{\ell}^{\pm}g^{\pm}$, for any $g^{\pm}\in C_{0}^{1}(\overline{\sX_{\pm}})$ and $\ell\in\bR_{+}$, by solving equation~\eqref{eq:WH} subject to the conditions $(a^{\pm})$ and $(b^{\pm})$. In view of Remark~\ref{rem:Expgpmsimp}, these functions lead to the expectation of the form \eqref{eq:ExpTauXTau} for any $g^{\pm}\in C_{0}^{1}(\overline{\sX_{\pm}})$. In particular, for any $c>0$ and $j\in\mathbf{E}_{\pm}$, by taking $g^{\pm}_{j}\in C_{0}^{1}(\overline{\sX_{\pm}})$ with
\begin{align}\label{eq:Funtgjpm}
g^{\pm}_{j}(s,i):=e^{-cs}\,\1_{\{j\}}(i),\quad (s,i)\in\sX_{\pm},
\end{align}
we obtain the following Laplace transform for $(\tau_{\ell}^{\pm,*},X_{\tau_{\ell}^{\pm,*}}^{*})$
\begin{align*}
\bE^{*}_{s,i}\bigg(e^{-c\tau_{\ell}^{\pm,*}}\1_{\big\{X_{\tau_{\ell}^{\pm,*}}^{*}=j\big\}}\bigg),
\end{align*}
for any $c\in(0,\infty)$, $\ell\in\bR_{+}$, and $(s,i)\in\sX$. We then perform the inverse Laplace transform with respect to $c$ to obtain the join distribution of $(\tau_{\ell}^{\pm,*},X_{\tau_{\ell}^{\pm,*}}^{*})$ under $\bP_{s,i}^{*}$, which enables us to compute the expectations \eqref{eq:ExpTauXTau} for any $g^{\pm}\in L^{\infty}(\overline{\sX_{\pm}})$.

Note that the equation \eqref{eq:WH} can be decomposed into the following two uncoupled equations
\begin{align}\label{eq:WHPlus}
\mV^{-1}\bigg(\frac{\partial}{\partial s}+\wt{\mathsf{\Lambda}}\bigg)\begin{pmatrix} I^{+} \\ S^{+} \end{pmatrix} g^{+}&=\begin{pmatrix} I^{+} \\ S^{+} \end{pmatrix} H^{+}g^{+},\quad g^{+}\in C_{0}^{1}(\overline{\sX_{+}}),\\
\label{eq:WHMinus} \mV^{-1}\bigg(\frac{\partial}{\partial s}+\wt{\mathsf{\Lambda}}\bigg)\begin{pmatrix} I^{-} \\ S^{-} \end{pmatrix} g^{+}&= -\begin{pmatrix} I^{-} \\ S^{-} \end{pmatrix} H^{-}g^{-},\quad g^{-}\in C_{0}^{1}(\overline{\sX_{-}}).
\end{align}
Hence, one can compute $J^{+}g^{+}$ and $G^{+}g^{+}$ (and thus $\cP_{\ell}^{+}g^{+}$) separately from $J^{-}g^{-}$ and $G^{-}g^{-}$ (and thus $\cP_{\ell}^{-}g^{-}$) by solving \eqref{eq:WHPlus} and \eqref{eq:WHMinus} subject to $(a^{+})$ and $(b^{+})$,  and $(a^{-})$ and $(b^{-})$, respectively.

\begin{remark}\label{rem:Riccati}
By \eqref{eq:MatrixBlocks}, \eqref{eq:TildeLambdaBlocks}, and Theorems \ref{thm:WHExistUniq} and \ref{thm:WHProbInterpr}, we see that $(J^{+},G^{+})$ is the unique solution, subject to $(a^{+})$ and $(b^{+})$, to the following two operator equations,
\begin{align}\label{eq:WHPlus1}
\big(\mV^{+}\big)^{-1}\bigg(\frac{\partial}{\partial s}+\wt{\mA}+\wt{\mB}\,S^{+}\bigg)g^{+}&=H^{+}g^{+},\\
\label{eq:WHPlus2} \big(\mV^{-}\big)^{-1}\bigg(\frac{\partial}{\partial s}+\wt{\mC}+\wt{\mD}\,S^{+}\bigg)g^{+}&=S^{+}H^{+}g^{+},
\end{align}
where $g^{+}\in C_{0}^{1}(\overline{\sX_{+}})$.
By plugging \eqref{eq:WHPlus1} into \eqref{eq:WHPlus2}, we obtain the operator Riccati equation of the form
\begin{align*}
\bigg(S^{+}\big(V^{+}\big)^{-1}\wt{\mB}\,S^{+}+S^{+}\big(\mV^{+}\big)^{-1}\bigg(\frac{\partial}{\partial s}+\wt{\mA}\bigg)-\big(\mV^{-}\big)^{-1}\bigg(\frac{\partial}{\partial s}+\wt{\mD}\bigg)S^{+}-\big(\mV^{-}\big)^{-1}\wt{\mC}\bigg)\,g^{+}=0.
\end{align*}
Hence, in order to compute $(J^{+},G^{+})$ from \eqref{eq:WHPlus}, one needs first to compute $J^{+}$ by solving the above operator equation subject to $(a^{+})$, and {then} $G^{+}$ is given in terms of $J^{+}$ by \eqref{eq:WHPlus1}. Similarly, one can compute $(J^{-},G^{-})$ from \eqref{eq:WHMinus} in an analogous way.
\end{remark}

\begin{remark}\label{rem:OperPsiInv}
The operator
\begin{align*}
\Psi:=\begin{pmatrix} I^{+} & J^{-} \\ J^{+} & I^{-} \end{pmatrix}:\,C_{0}(\overline{\sX})\rightarrow C_{0}(\overline{\sX})
\end{align*}
is {the} counterpart  of the matrix $\mS$ given in Theorem \ref{thm:BRW80WH1}. It can be shown that the operator $\Psi$ is injective. However, unlike  the matrix $\mS$ which is invertible, the operator $\Psi$ is not invertible in general. In fact, the surjectivity of $\Psi$ may fail, even when restricted to $C_{0}^{1}(\overline{\sX})$ (recall the condition $(a^{+})$(ii)). Nevertheless, the potential lack of invertibility of $\Psi$ does not affect the existence and uniqueness of our Wiener-Hopf factorization. It only affects the form of equality \eqref{eq:WH}, with $S^\pm$ replaced with $J^\pm$.
\end{remark}

\begin{remark}\label{rem:RealityCheck}
When the Markov family $\cM^{*}$ is time-homogeneous, namely, $\mathsf{\Lambda}_{s}=\mathsf{\Lambda}$ for all $s\in\bR_{+}$, where $\mathsf{\Lambda}$ is an $m\times m$ generator matrix, the equation \eqref{eq:WH} reduces to the time-homogeneous Wiener-Hopf factorization \eqref{eq:TimeHomoWH}, which, in light of the invertibility of $\mS$, can be rewritten as
\begin{align}\label{eq:TimeHomoWHReform}
\mV^{-1}(\mathsf{\Lambda}-c\,\mI) \begin{pmatrix} \mI^{+} & \mathsf{\Pi}^{-}_{c} \\ \mathsf{\Pi}^{+}_{c} & \mI^{-} \end{pmatrix} = \begin{pmatrix} \mI^{+} & \mathsf{\Pi}^{-}_{c} \\ \mathsf{\Pi}^{+}_{c} & \mI^{-} \end{pmatrix} \begin{pmatrix} \mQ^{+}_{c} & 0 \\ 0 & \mQ^{-}_{c} \end{pmatrix}.
\end{align}
In what follows, we will only check the ``+" part of the above equality.

Towards this end, for any $c\in(0,\infty)$ and $j\in\mathbf{E}_{+}$, take $g^{+}_{j}\in C_{0}^{1}(\overline{\sX_{+}})$ as in \eqref{eq:Funtgjpm}. Since $(J^{+},G^{+})$ is the unique solution to \eqref{eq:WHPlus} subject to $(\textnormal{a}^{+})$ and $(\textnormal{b}^{+})$, we have
\begin{align}\label{eq:WHPlusgj}
\mV^{-1}\bigg(\frac{\partial}{\partial s}+\mathsf{\Lambda}\bigg) \begin{pmatrix} I^{+} \\ J^{+} \end{pmatrix} g^{+}_{j}= \begin{pmatrix} I^{+} \\ J^{+} \end{pmatrix} G^{+}g^{+}_{j}.
\end{align}
Since $\cM^{*}$ is a time-homogeneous Markov family, for any $s,\ell\in\bR_{+}$ and $i\in\mathbf{E}$, the distribution of $(\tau^{+,*}_{\ell}(s)-s,X_{\tau^{+,*}_{\ell}(s)})$ under $\bP_{s,i}^{*}$ is the same as that of $(\tau^{+,*}_{\ell}(0),X_{\tau^{+,*}_{\ell}(0)})$ under $\bP_{0,i}^{*}$. Hence, for any $s\in\bR_{+}$ and $i\in\mathbf{E}_{+}$, we have
\begin{align}
\big(G^{+}\!g^{+}_{j}\big)(s,i)&=\lim_{\ell\rightarrow 0+}\frac{1}{\ell}\Big(\big(\cP^{+}_{\ell}g^{+}_{j}\big)(s,i)\!-\!g^{+}_{j}(s,i)\Big)\!=\!\lim_{\ell\rightarrow 0+}\frac{1}{\ell}\bigg(\bE^{*}_{s,i}\bigg(e^{-c\tau^{+,*}_{\ell}(s)}\1_{\big\{X^{*}_{\tau^{+,*}_{\ell}(s)}\!=j\!\big\}}\bigg)\!-\!\1_{\{j\}}(i)\bigg)\nonumber\\
\label{eq:GPlusgjPlus} &=\lim_{\ell\rightarrow 0+}\frac{e^{-cs}}{\ell}\bigg(\bE^{*}_{0,i}\bigg(e^{-c\tau^{+,*}_{\ell}(0)}\1_{\big\{X^{*}_{\tau^{+,*}_{\ell}(0)}=j\big\}}\bigg)-\1_{\{j\}}(i)\bigg)=e^{-cs}\,\mQ_{c}^{+}(i,j),
\end{align}
where we recall that the matrix $\mQ_{c}^{+}$ is defined in \eqref{eq:QcPlusMinus}. Similarly, for any $s\in\bR_{+}$ and $i\in\mathbf{E}_{-}$,
\begin{align}
\big(J^{+}g^{+}_{j}\big)(s,i)&=\bE^{*}_{s,i}\Big(g^{+}_{j}\Big(\tau^{+,*}_{0}(s),X^{*}_{\tau^{+,*}_{0}(s)}\Big)\Big)=\bE^{*}_{s,i}\bigg(e^{-c\tau^{+,*}_{0}(s)}\1_{\big\{X^{*}_{\tau^{+,*}_{0}(s)}=j\big\}}\bigg)\nonumber\\
\label{eq:JPlusgjPlus} &=e^{-cs}\,\bE^{*}_{0,i}\bigg(e^{-c\tau^{+,*}_{0}(0)}\1_{\big\{X^{*}_{\tau^{+,*}_{0}(0)}=j\big\}}\bigg)=e^{-cs}\,\mathsf{\Pi}^{+}_{c}(i,j),
\end{align}
where the matrix $\mathsf{\Pi}^{+}_{c}$ is defined by \eqref{eq:PicPlusMinus}, and
\begin{align}
\big(J^{+}G^{+}g^{+}_{j}\big)(s,i)&=\bE^{*}_{s,i}\Big(\big(G^{+}g^{+}_{j}\big)\Big(\tau^{+,*}_{0}(s),X^{*}_{\tau^{+,*}_{0}(s)}\Big)\Big)=\bE_{s,i}^{*}\Big(e^{-c\tau^{+,*}_{0}(s)}\,\mQ_{c}^{+}\Big(X^{*}_{\tau^{+,*}_{0}(s)},j\Big)\Big)\nonumber\\
&=e^{-cs}\,\bE_{0,i}^{*}\Big(\!e^{-c\tau^{+,*}_{0}(0)}\mQ_{c}^{+}\!\Big(\!X^{*}_{\tau^{+,*}_{0}(0)},j\!\Big)\!\Big)\!=\!e^{-cs}\!\!\!\sum_{k\in\mathbf{E}_{+}}\!\bE_{0,i}^{*}\bigg(\!e^{-c\tau^{+,*}_{0}(0)}\1_{\big\{\!X^{*}_{\tau_{0}^{+,*}}\!=k\!\big\}}\!\bigg)\mQ_{c}^{+}\!(k,j)\nonumber\\
\label{eq:JGPlusgjPlus} &=e^{-cs}\sum_{k\in\mathbf{E}_{+}}\mathsf{\Pi}^{+}_{c}(i,k)\,\mQ_{c}^{+}(k,j)=e^{-cs}\big(\mathsf{\Pi}^{+}_{c}\mathsf{Q}^{+}_{c}\big)(i,j).
\end{align}
By plugging \eqref{eq:GPlusgjPlus}$-$\eqref{eq:JGPlusgjPlus} into \eqref{eq:WHPlusgj}, we obtain
\begin{align*}
\mV^{-1}\bigg(\frac{\partial}{\partial s}+\mathsf{\Lambda}\bigg) \begin{pmatrix} \mI^{+} \\ \mathsf{\Pi}_{c}^{+} \end{pmatrix} e^{-cs}\,\mathbf{e}_{j}^{+} = \begin{pmatrix} \mI^{+} \\ \mathsf{\Pi}_{c}^{+} \end{pmatrix} \mQ_{c}^{+}\mathbf{e}_{j}^{+},
\end{align*}
where $\mathbf{e}_{j}^{+}$ is the $j$-th $m_{+}$-dimensional unit column vector. Finally, by evaluating the derivative and taking $s=0$ on the left-hand side above, we deduce that
\begin{align*}
\mV^{-1}\big(\mathsf{\Lambda}-c\,\mI\big) \begin{pmatrix} \mI^{+} \\ \mathsf{\Pi}_{c}^{+} \end{pmatrix} = \begin{pmatrix} \mI^{+} \\ \mathsf{\Pi}_{c}^{+} \end{pmatrix} \mQ_{c}^{+},
\end{align*}
which is the ``+" part of \eqref{eq:TimeHomoWHReform}.
\end{remark}

\begin{remark}\label{rem:Uniqueness}
From the discussion in Remark \ref{rem:RealityCheck}, for each $c>0$, solving the time-homogeneous Wiener-Hopf equation \eqref{eq:TimeHomoWH} for the matrices $(\mathsf{\Pi}_{c}^{\pm},\mQ_{c}^{\pm})$ is equivalent to solving the time-inhomogeneous Wiener-Hopf equation \eqref{eq:WH}, subject to the conditions $(a^{\pm})$ and $(b^{\pm})$, for the operators $(J^{\pm},G^{\pm})$ with $g^{\pm}\in C_{0}^{1}(\overline{\sX_{\pm}})$ of the form \eqref{eq:Funtgjpm}. Therefore, for each $c\in(0,\infty)$, the uniqueness of $(\mathsf{\Pi}_{c}^{\pm},\mQ_{c}^{\pm})$ as a solution to \eqref{eq:TimeHomoWH} corresponds to the uniqueness of $(J^{\pm},G^{\pm})$ as a solution to \eqref{eq:WH}, subject to $(a^{\pm})$ and $(b^{\pm})$, when $g^{\pm}$ is restricted to the subclasses of $C_{0}^{1}(\overline{\sX_{\pm}})$ of the form \eqref{eq:Funtgjpm}.

When $c=0$, the functions $g^{\pm}$ of the form \eqref{eq:Funtgjpm} do not belong to $C_{0}^{1}(\overline{\sX_{\pm}})$ anymore. Hence, our uniqueness result does not contradict the non-uniqueness of $(\mathsf{\Pi}_{0}^{\pm},\mQ_{0}^{\pm})$ that was shown in \cite{BarlowRogersWilliams1980}.
\end{remark}

\section{Proofs of the main results}\label{sec:Proofs}

In this section we prove Theorems \ref{thm:WHExistUniq} and \ref{thm:WHProbInterpr}. We will only give the proofs of the ``+" case in both theorems, as the ``$-$" case can be proved in an analogous way with $v$ replaced by $-v$.

\subsection{Auxiliary Markov families}\label{subsec:TimeHomogen}
In this subsection, we introduce an  auxiliary time-inhomogenous Markov family $\cM$ and an  auxiliary time-homogenous Markov family $\widetilde \cM$. We start by introducing some more notations of spaces and $\sigma$-fields. Let $\sY:=\mathbf{E}\times\bR$, and the Borel $\sigma$-field on $\sY$ is denoted by $\cB(\sY):=2^{\mathbf{E}}\otimes\cB(\bR)$. Accordingly, let $\overline{\sY}:=\sY\cup\{(\partial,\infty)\}$ be the one-point completion of $\sY$, and $\cB(\overline{\sY}):=\sigma(\cB(\sY)\cup\set{(\partial,\infty)})$. Moreover, we set $\sZ:=\bR_{+}\times\sY=\sX\times\bR$ and $\overline{\sZ}:=\sZ\cup\{(\infty,\partial,\infty)\}$.

Let $\Omega$ be the set of c\`{a}dl\`{a}g functions $\omega$ on $\bR_{+}$ taking values in $\sY$. We define $\omega(\infty):=(\partial,\infty)$ for every $\omega\in\Omega$. As shown in Appendix \ref{sec:AppendixA}, one can construct a {\it standard} canonical time-inhomogeneous Markov family (cf. \cite[Definition I.6.6]{GikhmanSkorokhod2004})
\begin{align*}
\cM:=\big\{\big(\Omega,\sF,\bF_{s},(X_{t},\varphi_{t})_{t\in[s,\infty]},\bP_{s,(i,a)}\big),\,(s,i,a)\in\overline{\sZ}\big\}
\end{align*}
with transition function $P$ given by
\begin{align}\label{eq:DefTranProbXvarphi}
P(s,(i,a),t,A):=\bP^{*}_{s,i}\bigg(\bigg(X^{*}_{t},\,a+\int_{s}^{t}v\big(X^{*}_{u}\big)\,du\bigg)\in A\bigg),
\end{align}
where $(s,i,a)\in\overline{\sZ}$, $t\in[s,\infty]$, and $A\in\cB(\overline{\sY})$. Furthermore, $\cM$ has the following properties:
\begin{itemize}
\item[(i)] for any $(s,i,a)\in\overline{\sZ}$,
    \begin{align}\label{eq:LawXXStar}
    \text{the law of }\,X\,\,\text{under }\,\bP_{s,(i,a)}\,\,=\,\,\text{the law of }\,X^{*}\,\,\text{under }\,\bP_{s,i}^{*};
    \end{align}
\item[(ii)] for any $(s,i,a)\in\sZ$,
    \begin{align}\label{eq:DistvarphiInt}
    \bP_{s,(i,a)}\bigg(\varphi_{t}=a+\int_{s}^{t}v(X_{u})\,du,\,\,\,\,\text{for all $t\in[s,\infty)$}\bigg)=1.
    \end{align}
\end{itemize}

Considering the standard Markov family $\cM$, for any $s,\ell\in\bR_{+}$, we define
\begin{align*}
\tau_{\ell}^{+}(s):=\inf\big\{t\in[s,\infty]:\,\varphi_{t}>\ell\big\},
\end{align*}
which is an $\bF_{s}$-stopping time in light of the continuity of $\varphi$ and the right-continuity of the filtration $\bF_{s}$. By similar arguments as in the proof of Lemma \ref{lem:RangeXTaupm}, for any $(s,i,a)\in\sZ$ and $\ell\in[a,\infty)$,
\begin{align}\label{eq:RangeXtauPlus}
\bP_{s,(i,a)}\Big(X_{\tau^{+}_{\ell}(s)}\in\mathbf{E}_{+}\cup\{\partial\}\Big)=1.
\end{align}
Moreover, it follows from \eqref{eq:DistvarphiInt} that, for any $(s,i,a)\in\sZ$,
\begin{align*}
\tau_{\ell}^{+}(s)&=\inf\bigg\{t\geq s:\,\,a+\!\int_{s}^{t}v(X_{u})\,du>\ell\bigg\},\quad\bP_{s,(i,a)}-\text{a.}\,\text{s.}\,.
\end{align*}
If no confusion arise, we will omit the $s$ in $\tau^{+}_{\ell}(s)$.

\begin{proposition}\label{prop:ExpPStarExpP}
For any $g^{+}\in L^{\infty}(\overline{\sX_{+}})$, $(s,i,a)\in\sZ$, and $\ell\in[a,\infty)$,
\begin{align*}
\bE_{s,(i,a)}\Big(g^{+}\Big(\tau_{\ell}^{+},X_{\tau_{\ell}^{+}}\Big)\Big)=\bE_{s,i}^{*}\Big(g^{+}\Big(\tau_{\ell-a}^{+,*},X_{\tau_{\ell-a}^{+,*}}^{*}\Big)\Big).
\end{align*}
\end{proposition}

\begin{proof}
By \eqref{eq:DistvarphiInt} and Lemma \ref{lem:XvarphiPathst}, we have
\begin{align*}
\bE_{s,(i,a)}\Big(g^{+}\Big(\tau_{\ell}^{+},X_{\tau_{\ell}^{+}}\Big)\!\Big)&=\bE_{s,(i,a)}\bigg(g^{+}\bigg(\!\inf\bigg\{u\!\geq\!s:a+\!\!\int_{s}^{u}\!v(X_{r})dr\!>\!\ell\bigg\},X_{\inf\big\{u\geq s:\,a+\int_{s}^{u}v(X_{r})dr>\ell\big\}}\bigg)\!\bigg)\\
&=\bE^{*}_{s,i}\bigg(g^{+}\bigg(\!\inf\bigg\{u\geq s:\!\int_{s}^{u}\!v\big(X_{r}^{*}\big)dr>\ell-a\bigg\},X_{\inf\big\{u\geq s:\int_{s}^{u}v(X_{r}^{*})dr>\ell-a\big\}}^{*}\bigg)\!\bigg)\\
&=\bE^{*}_{s,i}\bigg(g^{+}\bigg(\tau_{\ell-a}^{+,*},\,X_{\tau_{\ell-a}^{+,*}}^{*}\bigg)\bigg),
\end{align*}
which completes the proof.
\end{proof}

Proposition \ref{prop:ExpPStarExpP} provides a useful representation of the expectation $\bE_{s,i}^{*}\Big(g^{+}\Big(\tau_{\ell-a}^{+,*},X_{\tau_{\ell-a}^{+,*}}^{*}\Big)\Big)$. We will need still another representation of this expectation. Towards this end, we will first transform the time-inhomogeneous Markov family $\cM$ into a {\it time-homogeneous} Markov family
\begin{align*}
\wt{\cM}=\big\{\big(\wt{\Omega},\wt{\sF},\wt{\bF},({Z}_t)_{t\in\overline{\bR}_{+}},(\theta_{r})_{r\in\bR_{+}},\wt{\bP}_{z}\big),z\in\overline{\sZ}\big\}
\end{align*}
following the setup in \cite{Bottcher2014}.  The construction of $\wt{\cM}$ proceeds as follows.
\begin{itemize}
\item We let $\wt{\Omega}:=\overline{\bR}_{+}\times\Omega$ to be the new sample space, with elements $\wt{\omega}=(s,\omega)$, where $s\in\overline{\bR}_{+}$ and $\omega\in\Omega$. On $\wt{\Omega}$ we consider the $\sigma$-field
    \begin{align*}
    \wt{\sF}:=\Big\{\wt{A}\subset\wt{\Omega}:\,\wt{A}_{s}\in\sF_{\infty}^{s}\,\,\text{for any }s\in\overline{\bR}_{+}\Big\},
    \end{align*}
    where $\wt{A}_{s}:=\{\omega\in\Omega:\,(s,\omega)\in\wt{A}\}$ and $\sF_{\infty}^{s}$ is the last element in $\bF_{s}$ (the filtration in $\cM$).
\item We let $\overline{\sZ}=\sZ\cup\{(\infty,\partial,\infty)\}$ to be the new state space, where $\sZ=\bR_{+}\times\sY=\sX\times\bR$, with elements $z=(s,i,a)$. On $\sZ$ we consider the $\sigma$-field
	\begin{align*}
    \wt{\cB}(\sZ):=\left\{\wt{B}\subset\sZ:\,\wt{B}_{s}\in\cB(\sY)\,\,\,\text{for any }s\in\bR_{+}\right\},
	\end{align*}
    where $\wt{B}_{s}:=\big\{(i,a)\in\sY:\,(s,i,a)\in\wt{B}\big\}$. Let $\wt{\cB}(\overline{\sZ}):=\sigma(\wt{\cB}(\sZ)\cup\{(\infty,\partial,\infty)\})$.
\item We consider a family of probability measures $(\wt{\bP}_{z})_{z\in\overline{\sZ}}$, where, for $z=(s,i,a)\in\overline{\sZ}$,
	\begin{align}\label{eq:Probz}
    \wt{\bP}_{z}\big(\wt{A}\big)=\wt{\bP}_{s,i,a}\big(\wt{A}\big):=\bP_{s,(i,a)}\big(\wt{A}_{s}\big),\quad\wt{A}\in\wt{\sF}.
	\end{align}
\item We consider the process $Z:=(Z_{t})_{t\in\overline{\bR}_{+}}$ on $(\wt{\Omega},\wt{\sF})$, where, for $t\in\overline{\bR}_{+}$,
	\begin{align}\label{eq:ProcZ}
    Z_{t}(\wt{\omega}):=\big(s+t,X_{s+t}(\omega),\varphi_{s+t}(\omega)\big),\quad\wt{\omega}=(s,\omega)\in\wt{\Omega}.
	\end{align}
	Hereafter, we denote the three components of $Z$ by $Z^{1}$, $Z^{2}$, and $Z^{3}$, respectively.
\item On $(\wt{\Omega},\wt{\sF})$, we define  $\wt{\bF}:=(\wt{\sF}_{t})_{t\in\overline{\bR}_{+}}$, where $\wt{\sF}_{t}:=\wt{\sG}_{t+}$ (with the convention $\wt{\sG}_{\infty+}=\wt{\sG}_{\infty}$), and $(\wt{\sG}_{t})_{t\in\overline{\bR}_{+}}$ is the completion of the natural filtration generated by $(Z_{t})_{t\in\overline{\bR}_{+}}$ with respect to the set of probability measures $\{\wt{\bP}_{z},z\in\overline{\sZ}\}$ (cf. \cite[Chapter I]{GikhmanSkorokhod2004}).
\item Finally, for any $r\in\bR_{+}$, we consider the shift operator ${\theta}_{r}:\wt{\Omega}\rightarrow\wt{\Omega}$ defined by
    \begin{align*}
    \theta_{r}\,\wt{\omega}=(u+r,\omega_{\cdot+r}),\quad\wt{\omega}=(u,\omega)\in\wt{\Omega}.
    \end{align*}
    It follows that $Z_{t}\circ{\theta}_{r}=Z_{t+r}$, for any $t,r\in\bR_{+}$.
\end{itemize}

For $z=(s,i,a)\in\overline{\sZ}$, $t\in\overline{\bR}_{+}$, and $\wt{B}\in\wt{\cB}(\overline{\sZ})$, we define the transition function $\wt{P}$ by
\begin{align*}
\wt{P}\big(z,t,\wt{B}\big):=\wt{\bP}_{z}\big(Z_{t}\in\wt{B}\big).
\end{align*}
In view of \eqref{eq:Probz}, we have
\begin{align}\label{eq:TranProbTildeX}
\wt{P}\big(z,t,\wt{B}\big)=\bP_{s,(i,a)}\Big((X_{t+s},\varphi_{t+s})\in\wt{B}_{s+t}\Big)=P\big(s,(i,a),s+t,\wt{B}_{s+t}\big).
\end{align}
By Lemma \ref{lem:FellerTranProb}, the transition function $P$, defined in \eqref{eq:DefTranProbXvarphi}, is associated with a Feller semigroup, so that $P$ is a Feller transition function. This and \cite[Theorem 3.2]{Bottcher2014} imply  that $\wt{P}$ is also a Feller transition function. In light of the right continuity of the sample paths, and invoking \cite[Theorem I.4.7]{GikhmanSkorokhod2004}, we conclude that $\wt{\cM}$ is a {\it time-homogeneous strong} Markov family. 

For any $\ell\in\bR$, we define
\begin{align*}
\wt{\tau}_{\ell}^{+}:=\inf\big\{t\in\overline{\bR}_{+}:\,Z^{3}_{t}>\ell\big\}.
\end{align*}
Note that $\wt{\tau}_{\ell}^{+}$ is an $\wt{\bF}$-stopping time since $Z^{3}$ has continuous sample paths and $\wt{\bF}$ is right-continuous. In light of \eqref{eq:DistvarphiInt}, \eqref{eq:Probz}, and \eqref{eq:ProcZ}, for any $(s,i,a)\in\sZ$, we have
\begin{align}\label{eq:DistTildeZ3IntTildeZ2}
\wt{\bP}_{s,i,a}\Big(Z^{3}_{t}=a+\int_{0}^{t}v\big(Z^{2}_{u}\big)\,du,\,\,\text{for all }t\in\bR_{+}\Big)=1.
\end{align}
Consequently, for any $(s,i)\in\sX_{+}$ and $\ell\in\bR$,
\begin{align}\label{eq:TildeTauPlus0}
\wt{\bP}_{s,i,\ell}\big(\wt{\tau}_{\ell}^{+}=0\big)=1.
\end{align}
Moreover, by \eqref{eq:RangeXtauPlus} and \eqref{eq:DistTildeZ3IntTildeZ2}, for any $(s,i,a)\in\overline{\sZ}$ and $\ell\in[a,\infty)$, we have
\begin{align}\label{eq:RangeZ2TildeTauPlus}
\wt{\bP}_{s,i,a}\Big(Z^{2}_{\wt{\tau}^{+}_{\ell}}\in\mathbf{E}_{+}\cup\{\partial\}\Big)=1.
\end{align}
By Proposition \ref{prop:ExpPStarExpP}, \eqref{eq:Probz} and \eqref{eq:ProcZ}, for any $g^{+}\in L^{\infty}(\overline{\sX_{+}})$, $(s,i,a)\in\sZ$, and $\ell\in[a,\infty)$,
\begin{align}\label{eq:ExpTildePPStarPlus}
\bE_{s,i}^{*}\bigg(g^{+}\bigg(\tau_{\ell-a}^{+,*},X_{\tau_{\ell-a}^{+,*}}^{*}\bigg)\bigg)=\wt{\bE}_{s,i,a}\Big(g^{+}\Big(Z^{1}_{\wt{\tau}_{\ell}^{+}},Z^{2}_{\wt{\tau}_{\ell}^{+}}\Big)\Big),
\end{align}
which, in particular, implies that
\begin{align}\label{eq:ExpTildePShift}
\wt{\bE}_{s,i,a}\Big(g^{+}\Big(Z^{1}_{\wt{\tau}_{\ell}^{+}},Z^{2}_{\wt{\tau}_{\ell}^{+}}\Big)\Big)=\wt{\bE}_{s,i,0}\Big(g^{+}\Big(Z^{1}_{\wt{\tau}_{\ell-a}^{+}},Z^{2}_{\wt{\tau}_{\ell-a}^{+}}\Big)\Big).
\end{align}
Consequently, the operators $J^{+}$ and $\cP_{\ell}^{+}$, $\ell\in\bR_{+}$, defined by \eqref{eq:DefJPlus} and \eqref{eq:DefPPlus}, can be written as
\begin{align}\label{eq:DefJPlusTildeP}
\big(J^{+}g^{+}\big)(s,i)&=\wt{\bE}_{s,i,0}\Big(g^{+}\Big(Z^{1}_{\wt{\tau}_{0}^{+}},Z^{2}_{\wt{\tau}_{0}^{+}}\Big)\Big),\quad g^{+}\in L^{\infty}(\overline{\sX_{+}}),\quad (s,i)\in\sX_{-},\\
\label{eq:DefPellPlusTildeP} \big(\cP_{\ell}^{+}g^{+}\big)(s,i)&=\wt{\bE}_{s,i,0}\Big(g^{+}\Big(Z^{1}_{\wt{\tau}_{\ell}^{+}},Z^{2}_{\wt{\tau}_{\ell}^{+}}\Big)\Big),\quad g^{+}\in L^{\infty}(\overline{\sX_{+}}),\quad (s,i)\in\sX_{+}.
\end{align}

We conclude this section with the following key lemma, which will be crucial in the proofs of {the} main results.

\begin{lemma}\label{lem:StrongMarkov}
Let $\wt{\tau}$ be any $\wt{\bF}$-stopping time, and $g^{+}\in L^{\infty}(\overline{\sX_{+}})$. Then, for any $(s,i,a)\in\overline{\sZ}$ and $\ell\in[a,\infty)$, we have
\begin{align}\label{eq:StrongMarkovCondPlus}
\1_{\{\wt{\tau}\leq\wt{\tau}^{+}_{\ell}\}}\,\wt{\bE}_{s,i,a}\Big(g^{+}\Big(Z^{1}_{\wt{\tau}^{+}_{\ell}},Z^{2}_{\wt{\tau}^{+}_{\ell}}\Big)\,\Big|\,\wt{\sF}_{\wt{\tau}}\Big)=\1_{\{\wt{\tau}\leq\wt{\tau}^{+}_{\ell}\}}\,\wt{\bE}_{Z^{1}_{\wt{\tau}},Z^{2}_{\wt{\tau}},Z^{3}_{\wt{\tau}}}\Big(g^{+}\Big(Z^{1}_{\wt{\tau}^{+}_{\ell}},Z^{2}_{\wt{\tau}^{+}_{\ell}}\Big)\Big),\quad\wt{\bP}_{s,i,a}-\text{a.}\,\text{s.}.
\end{align}
\end{lemma}

\begin{proof}
Note that if $(s,i,a)=(\infty,\partial,\infty)$, then both sides of \eqref{eq:StrongMarkovCondPlus} are zero. Hence, without loss of generality, assume that $(s,i,a)\in\sZ$ and $\{\wt{\tau}\leq\wt{\tau}^{+}_{\ell}\}\neq\emptyset$. Note that for any $\ell\in\bR$ and $\wt{\omega}\in\{\wt{\tau}\leq\wt{\tau}^{+}_{\ell}\}$,
\begin{align*}
\wt{\tau}^{+}_{\ell}\big(\theta_{\wt{\tau}}(\wt{\omega})\big)&=\inf\big\{t\in\bR_{+}:Z^{3}_{t}\big(\theta_{\wt{\tau}}(\wt{\omega})\big)>\ell\big\}=\inf\big\{t\in\bR_{+}:Z^{3}_{t+\wt{\tau}(\wt{\omega})}(\wt{\omega})>\ell\big\}\\
&=\inf\big\{t\geq\wt{\tau}(\wt{\omega}):Z^{3}_{t}(\wt{\omega})>\ell\big\}-\wt{\tau}(\wt{\omega})=\wt{\tau}_{\ell}^{+}(\wt{\omega})-\wt{\tau}(\wt{\omega}),
\end{align*}
and thus
\begin{align*}
\Big(Z_{\wt{\tau}_{\ell}^{+}}\circ\theta_{\wt{\tau}}\Big)(\wt{\omega})=Z_{\wt{\tau}^{+}_{\ell}(\theta_{\wt{\tau}}(\wt{\omega}))}\big(\theta_{\wt{\tau}}(\wt{\omega})\big)=Z_{\wt{\tau}^{+}_{\ell}(\wt{\omega})-\wt{\tau}(\wt{\omega})}\big(\theta_{\wt{\tau}(\wt{\omega})}\wt{\omega}\big)=Z_{\wt{\tau}_{\ell}^{+}}(\wt{\omega}).
\end{align*}
Therefore, for any $(s,i,a)\in\sZ$ and $\ell\in[a,\infty)$,
\begin{align*}
\1_{\{\wt{\tau}\leq\wt{\tau}^{+}_{\ell}\}}\,\wt{\bE}_{s,i,a}\Big(g^{+}\Big(Z^{1}_{\wt{\tau}^{+}_{\ell}},Z^{2}_{\wt{\tau}^{+}_{\ell}}\Big)\Big|\wt{\sF}_{\wt{\tau}}\Big)&=\wt{\bE}_{s,i,a}\Big(\1_{\{\wt{\tau}\leq\wt{\tau}^{+}_{\ell}\}}\,g^{+}\Big(Z^{1}_{\wt{\tau}^{+}_{\ell}},Z^{2}_{\wt{\tau}^{+}_{\ell}}\Big)\Big|\wt{\sF}_{\wt{\tau}}\Big)\\
&=\wt{\bE}_{s,i,a}\Big(\1_{\{\wt{\tau}\leq\wt{\tau}^{+}_{\ell}\}}\,g^{+}\Big(Z^{1}_{\wt{\tau}^{+}_{\ell}}\circ\theta_{\wt{\tau}},Z^{2}_{\wt{\tau}^{+}_{\ell}}\circ\theta_{\wt{\tau}}\Big)\Big|\wt{\sF}_{\wt{\tau}}\Big)\\
&=\1_{\{\wt{\tau}\leq\wt{\tau}^{+}_{\ell}\}}\,\wt{\bE}_{s,i,a}\Big(g^{+}\Big(Z^{1}_{\wt{\tau}^{+}_{\ell}},Z^{2}_{\wt{\tau}^{+}_{\ell}}\Big)\circ\theta_{\wt{\tau}}\Big|\wt{\sF}_{\wt{\tau}}\Big)\\
&=\1_{\{\wt{\tau}\leq\wt{\tau}^{+}_{\ell}\}}\,\wt{\bE}_{Z^{1}_{\wt{\tau}},Z^{2}_{\wt{\tau}},Z^{3}_{\wt{\tau}}}\Big(g^{+}\Big(Z^{1}_{\wt{\tau}^{+}_{\ell}},Z^{2}_{\wt{\tau}^{+}_{\ell}}\Big)\Big),
\end{align*}
where we used the fact that $\{\wt{\tau}\leq\wt{\tau}^{+}_{\ell}\}\in\wt{\sF}_{\wt{\tau}}$ (cf. \cite[Lemma 1.2.16]{KaratzasShreve1998}) in the first and third equality, and the strong Markov property of $Z$ (cf. \cite[Theorem III.9.4]{RogersWilliams1994}) in the last equality.
\end{proof}

\begin{corollary}\label{cor:StrongMarkovExp}
Under the assumptions of Lemma \ref{lem:StrongMarkov},
\begin{align*}
\wt{\bE}_{s,i,a}\Big(\1_{\{\wt{\tau}\leq\wt{\tau}^{+}_{\ell}\}}\,g^{+}\Big(Z^{1}_{\wt{\tau}^{+}_{\ell}},Z^{2}_{\wt{\tau}^{+}_{\ell}}\Big)\Big)=\wt{\bE}_{s,i,a}\Big(\1_{\{\wt{\tau}\leq\wt{\tau}^{+}_{\ell}\}}\,\wt{\bE}_{Z^{1}_{\wt{\tau}},Z^{2}_{\wt{\tau}},Z^{3}_{\wt{\tau}}}\Big(g^{+}\Big(Z^{1}_{\wt{\tau}^{+}_{\ell}},Z^{2}_{\wt{\tau}^{+}_{\ell}}\Big)\Big)\Big).
\end{align*}
\end{corollary}

\begin{proof}
This is a direct consequence of \eqref{eq:StrongMarkovCondPlus} and the fact that $\{\wt{\tau}\leq\wt{\tau}^{+}_{\ell}\}\in\wt{\sF}_{\wt{\tau}}$.
\end{proof}

\begin{corollary}\label{cor:StrongMarkovExpTildeTau}
For any $g^{+}\in L^{\infty}(\overline{\sX_{+}})$, $(s,i,a)\in\overline{\sZ}$, $\ell\in[a,\infty)$, and $h\in(0,\infty)$,
\begin{align*}
\wt{\bE}_{s,i,a}\Big(g^{+}\Big(Z^{1}_{\wt{\tau}^{+}_{\ell+h}},Z^{2}_{\wt{\tau}^{+}_{\ell+h}}\Big)\Big)=\wt{\bE}_{s,i,a}\Big(\big(\cP_{h}^{+}g^{+}\big)\Big(Z^{1}_{\wt{\tau}_{\ell}^{+}},Z^{2}_{\wt{\tau}_{\ell}^{+}}\Big)\Big). 
\end{align*}
\end{corollary}

\begin{proof}
Since $g^{+}\in L^{\infty}(\overline{\sX_{+}})$ and $\wt{\tau}_{\ell+h}^{+}\geq\wt{\tau}_{\ell}^{+}$, $g^{+}(Z^{1}_{\wt{\tau}^{+}_{\ell+h}},Z^{2}_{\wt{\tau}^{+}_{\ell+h}})=g^{+}(\infty,\partial)=0$ on $\{\wt{\tau}_{\ell}^{+}=\infty\}$, so that $\wt{\bE}_{Z^{1}_{\wt{\tau}_{\ell}^{+}},Z^{2}_{\wt{\tau}_{\ell}^{+}},Z^{3}_{\wt{\tau}_{\ell}^{+}}}(g^{+}(Z^{1}_{\wt{\tau}^{+}_{\ell+h}},Z^{2}_{\wt{\tau}^{+}_{\ell+h}}))=0$ on $\{\wt{\tau}_{\ell}^{+}=\infty\}$. Moreover, $Z^{3}_{\wt{\tau}^{+}_{\ell}}=\ell$ on $\{\wt{\tau}_{\ell}^{+}<\infty\}$. Thus, using Corollary \ref{cor:StrongMarkovExp}, \eqref{eq:ExpTildePShift}, \eqref{eq:DefPellPlusTildeP}, and \eqref{eq:RangeXtauPlus}, we obtain that
\begin{align*}
\wt{\bE}_{s,i,a}\Big(g^{+}\Big(Z^{1}_{\wt{\tau}^{+}_{\ell+h}},Z^{2}_{\wt{\tau}^{+}_{\ell+h}}\Big)\Big)&=\wt{\bE}_{s,i,a}\bigg(\wt{\bE}_{Z^{1}_{\wt{\tau}^{+}_{\ell}},Z^{2}_{\wt{\tau}^{+}_{\ell}},Z^{3}_{\wt{\tau}^{+}_{\ell}}}\Big(g^{+}\Big(Z^{1}_{\wt{\tau}^{+}_{\ell+h}},Z^{2}_{\wt{\tau}^{+}_{\ell+h}}\Big)\Big)\bigg)\\
&=\wt{\bE}_{s,i,a}\bigg(\1_{\{\wt{\tau}^{+}_{\ell}<\infty\}}\,\wt{\bE}_{Z^{1}_{\wt{\tau}^{+}_{\ell}},Z^{2}_{\wt{\tau}^{+}_{\ell}},\ell}\Big(g^{+}\Big(Z^{1}_{\wt{\tau}^{+}_{\ell+h}},Z^{2}_{\wt{\tau}^{+}_{\ell+h}}\Big)\Big)\bigg)\\
&=\wt{\bE}_{s,i,a}\bigg(\1_{\{\wt{\tau}^{+}_{\ell}<\infty\}}\,\wt{\bE}_{Z^{1}_{\wt{\tau}^{+}_{\ell}},Z^{2}_{\wt{\tau}^{+}_{\ell}},0}\Big(g^{+}\Big(Z^{1}_{\wt{\tau}^{+}_{h}},Z^{2}_{\wt{\tau}^{+}_{h}}\Big)\Big)\bigg)\\
&=\wt{\bE}_{s,i,a}\Big(\1_{\{\wt{\tau}^{+}_{\ell}<\infty\}}\big(\cP_{h}^{+}g^{+}\big)\Big(Z^{1}_{\wt{\tau}^{+}_{\ell}},Z^{2}_{\wt{\tau}^{+}_{\ell}}\Big)\Big)\\
&=\wt{\bE}_{s,i,a}\Big(\big(\cP_{h}^{+}g^{+}\big)\Big(Z^{1}_{\wt{\tau}^{+}_{\ell}},Z^{2}_{\wt{\tau}^{+}_{\ell}}\Big)\Big),
\end{align*}
where the last equality is due to the fact that $(\cP_{h}^{+}g^{+})(\infty,\partial)=0$.
\end{proof}

\begin{remark}\label{rem:ProofExpgpsimJPMinus}
We now verify \eqref{eq:ExpgpsimJPPlus} using the strong Markov family $\wt{\cM}$. Indeed, by \eqref{eq:ExpTildePPStarPlus} and Corollary \ref{cor:StrongMarkovExpTildeTau}, for any $g^{+}\in L^{\infty}(\overline{\sX_{+}})$, $(s,i)\in\sX_{-}$, and $\ell\in(0,\infty)$,
\begin{align*}
\bE_{s,i}^{*}\Big(g^{+}\Big(\tau_{\ell}^{+,*},X_{\tau_{\ell}^{+,*}}^{*}\Big)\Big)=\wt{\bE}_{s,i,0}\Big(g^{+}\Big(Z^{1}_{\wt{\tau}_{\ell}^{+}},Z^{2}_{\wt{\tau}_{\ell}^{+}}\Big)\Big)=\wt{\bE}_{s,i,0}\Big(\big(\cP_{\ell}^{+}g^{+}\big)\Big(Z^{1}_{\wt{\tau}_{0}^{+}},Z^{2}_{\wt{\tau}_{0}^{+}}\Big)\Big)=\big(J^{+}\cP_{\ell}^{+}g^{+}\big)(s,i).
\end{align*}
\end{remark}

\subsection{A regularity lemma}\label{subsec:RegLemma}
Fix $g^{+}\in C_{0}(\overline{\sX_{+}})$, and define $f_{+}:\sX\times\bR_{+}\rightarrow\bR$ by
\begin{align}\label{eq:FuntfPlus}
f_{+}(s,i,\ell):=\wt{\bE}_{s,i,0}\Big(g^{+}\Big(Z_{\wt{\tau}_{\ell}^{+}}^{1},Z_{\wt{\tau}_{\ell}^{+}}^{2}\Big)\Big).
\end{align}
In particular, in view of \eqref{eq:TildeTauPlus0}, we have
\begin{align}\label{eq:Trivfplus}
f_{+}(s,i,0)=g^{+}(s,i),\quad (s,i)\in\overline{\sX_{+}}.
\end{align}
Moreover, by \eqref{eq:DefJPlusTildeP}, \eqref{eq:DefPellPlusTildeP}, and \eqref{eq:FuntfPlus},
\begin{align}\label{eq:JPlusfPlus}
J^{+}g^{+}(s,i)&=f_{+}(s,i,0),\quad (s,i)\in\sX_{-},\\
\label{eq:PellPlusfPlus} \cP^{+}_{\ell}g^{+}(s,i)&=f_{+}(s,i,\ell),\quad (s,i)\in\sX_{+},\quad\ell\in\bR_{+}.
\end{align}
The following lemma addresses the continuity of $f_{+}$ with respect to different variables. In particular, due to \eqref{eq:JPlusfPlus} and \eqref{eq:PellPlusfPlus}, for any $g^{+}\in C_{0}(\overline{\sX_{+}})$, the continuity of $J^{+}g^{+}(\cdot,i)$, and $\cP_{\cdot}^{+}g^{+}(\cdot,i)$, with respect to each individual variable, is established as special cases of $f_{+}$.

Recall that, by Assumption~\ref{assump:GenLambda}, $K$ is a constant such that $\sup_{s\in\bR_{+},i,j\in\mathbf{E}}|\mathsf{\Lambda}_{s}(i,j)|\leq K$. Additionally, recall that $\underline{v}=\min_{i\in\mathbf{E}}|v(i)|$ and $\overline{v}=\max_{i\in\mathbf{E}}|v(i)|$.

\begin{lemma}\label{lem:UnifContf}
For any $g^{+}\in C_{0}(\overline{\sX_{+}})$, $f_{+}(\cdot,i,\cdot)$ is uniformly continuous on $\bR_{+}^{2}$, uniformly for all $i\in\mathbf{E}$. That is, for any $\varepsilon>0$, there exists $\delta=\delta(\varepsilon,K,\|g^{+}\|_{\infty},\underline{v},\overline{v})>0$ such that
\begin{align*}
\sup_{i\in\mathbf{E}}\,\sup_{\substack{(s_{1},\ell_{1}),(s_{2},\ell_{2})\in\bR_{+}^{2}: \\ |s_{2}-s_{1}|+|\ell_{2}-\ell_{1}|<\delta}}\big|f_{+}(s_{2},i,\ell_{2})-f_{+}(s_{1},i,\ell_{1})\big|<\varepsilon.
\end{align*}
Moreover, for any $i\in\mathbf{E}$ and $\ell\in\bR_{+}$, $f(\cdot,i,\ell)\in C_{0}(\bR_{+})$. In particular, $J^{+}g^{+}\in C_{0}(\overline{\sX_{-}})$ and $\cP_{\ell}^{+}g^{+}\in C_{0}(\overline{\sX_{+}})$.
\end{lemma}

The proof of this lemma is deferred to Appendix \ref{sec:AppendixB}.

\subsection{Existence of the Wiener-Hopf factorization}\label{sec:ExistProof}

This section is devoted to the proof of the $``+"$ portion of Theorem \ref{thm:WHExistUniq}. We do this by demonstrating the existence of solution to \eqref{eq:WHPlus} subject to conditions ($a^{+}$) and ($b^{+}$). Recall that $J^{+}$ and $(\cP_{\ell}^{+})_{\ell\in\bR_{+}}$ are defined as in \eqref{eq:DefJPlus} and \eqref{eq:DefPPlus}, and have the respective representations \eqref{eq:DefJPlusTildeP} and \eqref{eq:DefPellPlusTildeP} in terms of the time-homogeneous Markov family $\wt{\cM}$; $G^{+}$ is defined as in \eqref{eq:DefGenGPlus} with respect to $(\cP_{\ell}^{+})_{\ell\in\bR_{+}}$. We will show that $(J^{+},G^{+})$ is a solution to \eqref{eq:WHPlus} (which is equivalent to \eqref{eq:WHPlus1}$-$\eqref{eq:WHPlus2}) subject to ($a^{+}$) and ($b^{+}$). The proof is divided into four steps.

\medskip
\noindent
\textbf{Step 1.}  In this step show that $J^{+}$ satisfies the condition $(a^{+})$(i).

Let $g^{+}\in C_{0}(\overline{\sX_{+}})$. By Lemma \ref{lem:UnifContf}, we have $J^{+}g^{+}\in C_{0}(\overline{\sX_{-}})$. Moreover, if $\supp g^{+}\subset[0,\eta_{g^{+}}]\times\mathbf{E}_{+}$ for some $\eta_{g^{+}}\in(0,\infty)$, we have $(J^{+}g^{+})(s,i)=\wt{\bE}_{s,i,0}(g^{+}(s+\wt{\tau}_{0}^{+},Z^{2}_{\wt{\tau}_{0}^{+}}))=0$, for any $(s,i)\in[\eta_{g^{+}},\infty)\times\mathbf{E}_{-}$, which completes the proof in Step 1.

\medskip
\noindent
\textbf{Step 2.} Here we will show that $(\cP^{+}_{\ell})_{\ell\in\bR_{+}}$ is a strongly continuous positive contraction semigroup on $C_{0}(\overline{\sX_{+}})$, and thus a Feller semigroup.

Let $g^{+}\in C_{0}(\overline{\sX_{+}})$ and $\ell\in\bR_{+}$. By Lemma \ref{lem:UnifContf}, we have $\cP^{+}_{\ell}g^{+}\in C_{0}(\overline{\sX_{+}})$. The positivity and contraction property of $\cP_{\ell}^{+}$ follow immediately from its definition. Hence, it remains to show that $(\cP^{+}_{\ell})_{\ell\in\bR_{+}}$ is a strongly continuous semigroup.

To this end, we fix any $(s,i)\in\sX_{+}$. By \eqref{eq:Trivfplus} and \eqref{eq:PellPlusfPlus}, we first have
\begin{align}\label{eq:PellPlusSemiGroup1}
\big(\cP_{0}^{+}g^{+}\big)(s,i)=f_{+}(s,i,0)=g^{+}(s,i).
\end{align}
Moreover, for any $\ell\in\bR_{+}$ and $h>0$, by \eqref{eq:DefPellPlusTildeP} and Corollary \ref{cor:StrongMarkovExpTildeTau}, we have
\begin{align}\label{eq:PellPlusSemiGroup2}
\big(\cP^{+}_{\ell+h}g^{+}\big)(s,i)&=\wt{\bE}_{s,i,0}\Big(g^{+}\Big(\!Z^{1}_{\wt{\tau}^{+}_{\ell+h}},Z^{2}_{\wt{\tau}^{+}_{\ell+h}}\Big)\!\Big)=\wt{\bE}_{s,i,0}\Big(\!\big(\cP^{+}_{h}g^{+}\big)\!\Big(Z^{1}_{\wt{\tau}^{+}_{\ell}},Z^{2}_{\wt{\tau}^{+}_{\ell}}\Big)\!\Big)=\big(\cP^{+}_{\ell}\cP^{+}_{h}g^{+}\big)(s,i),
\end{align}
Hence, $(\cP_{\ell}^{+})_{\ell\in\bR_{+}}$ is a semigroup on $C_{0}(\overline{\sX_{+}})$.

Finally, for any $\ell\in\bR_{+}$ and $g^{+}\in C_{0}(\overline{\sX_{+}})$, by \eqref{eq:PellPlusfPlus} and Lemma \ref{lem:UnifContf}, we have
\begin{align*}
\lim_{\ell\rightarrow 0+}\sup_{(s,i)\in\overline{\sX_{+}}}\big|\big(\cP^{+}_{\ell}g^{+}\big)(s,i)-g^{+}(s,i)\big|&=\lim_{\ell\rightarrow 0+}\sup_{(s,i)\in\sX_{+}}\big|\big(\cP^{+}_{\ell}g^{+}\big)(s,i)-g^{+}(s,i)\big|\\
&=\lim_{\ell\rightarrow 0+}\sup_{(s,i)\in\sX_{+}}\big|f_{+}(s,i,\ell)-f_{+}(s,i,0)\big|=0,
\end{align*}
{which} shows the strong continuity of $(\cP_{\ell}^{+})_{\ell\in\bR_{+}}$, and thus completes the proof in Step 2.

\medskip
\noindent
\textbf{Step 3.}  We will show here that $G^{+}$ is the strong generator of $(\cP^{+}_{\ell})_{\ell\in\bR_{+}}$ with domain $C_{0}^{1}(\overline{\sX_{+}})$, and that
    \begin{align}\label{eq:WHplusU}
    G^{+}g^{+}=\big(\mV^{+}\big)^{-1}\bigg(\frac{\partial}{\partial s}+\wt{\mA}+\wt{\mB}J^{+}\bigg)g^{+},\quad g^{+}\in C_{0}^{1}(\overline{\sX_{+}}).
    \end{align}
     The argument proceeds in two sub-steps: \textbf{(i)} and \textbf{(ii)}.

\medskip
\noindent
\textbf{(i)} We first show that, for any $g^{+}\in C_{0}(\overline{\sX_{+}})$, the pointwise limit in \eqref{eq:DefGenGPlus} exists for every $(s,i)\in\overline{\sX_{+}}$ if and only if $g(\cdot,i)$ is right-differentiable on $\bR_{+}$ for each $i\in\mathbf{E}$. Moreover, for such $g^{+}$, we have
\begin{align}
&\lim_{\ell\rightarrow 0+}\frac{1}{\ell}\big(\cP^{+}_{\ell}g^{+}(s,i)-g^{+}(s,i)\big)\nonumber\\
\label{eq:LimitGPlus} &\quad =\frac{1}{v(i)}\bigg(\frac{\partial_+ g^{+}}{\partial s}(s,i)+\sum_{j\in\mathbf{E}_{+}}\mathsf{\Lambda}_{s}(i,j)g^{+}(s,j)+\sum_{j\in\mathbf{E}_{-}}\mathsf{\Lambda}_{s}(i,j)\big(J^{+}g^{+}\big)(s,j)\bigg),\quad (s,i)\in\overline{\sX_{+}}.
\end{align}
When $(s,i)=(\infty,\partial)$, \eqref{eq:LimitGPlus} is trivial since both sides of the equality are equal to zero. In what follows, fix  $g^{+}\in C_{0}(\overline{\sX_{+}})$ and $(s,i)\in\sX_{+}$.

Let $\wt{\gamma}_{1}$ be the first jump time of $\wt{Z}^{2}$. For any $\ell\in(0,\infty)$, by \eqref{eq:DefPellPlusTildeP} and \eqref{eq:Gamma1overh}, we have
\begin{align}
&\frac{1}{\ell}\big(\cP^{+}_{\ell}g^{+}(s,i)-g^{+}(s,i)\big)=\frac{1}{\ell}\Big(\wt{\bE}_{s,i,0}\Big(g^{+}\Big(Z^{1}_{\wt{\tau}_{\ell}^{+}},Z^{2}_{\wt{\tau}_{\ell}^{+}}\Big)\Big)-g^{+}(s,i)\Big)\nonumber\\
&\quad =\frac{1}{\ell}\bigg(\wt{\bE}_{s,i,0}\Big(g^{+}\Big(Z^{1}_{\wt{\tau}_{\ell}^{+}},Z^{2}_{\wt{\tau}_{\ell}^{+}}\Big)\1_{\{\wt{\gamma}_{1}>\ell/v(i)\}}\Big)+\wt{\bE}_{s,i,0}\Big(g^{+}\Big(Z^{1}_{\wt{\tau}_{\ell}^{+}},Z^{2}_{\wt{\tau}_{\ell}^{+}}\Big)\1_{\{\wt{\gamma}_{1}\leq\ell/v(i)\}}\Big)-g^{+}(s,i)\bigg)\nonumber\\
&\quad =\frac{1}{\ell}\bigg(\wt{\bE}_{s,i,0}\Big(g^{+}\Big(Z^{1}_{\wt{\tau}_{\ell}^{+}},Z^{2}_{\wt{\tau}_{\ell}^{+}}\Big)\1_{\{\wt{\gamma}_{1}>\ell/v(i),\wt{\tau}^{+}_{\ell}=\ell/v(i)\}}\Big)+\wt{\bE}_{s,i,0}\Big(g^{+}\Big(Z^{1}_{\wt{\tau}_{\ell}^{+}},Z^{2}_{\wt{\tau}_{\ell}^{+}}\Big)\1_{\{\wt{\gamma}_{1}\leq\ell/v(i)\}}\Big)-g^{+}(s,i)\bigg)\nonumber\\
&\quad =\frac{1}{\ell}\bigg(\wt{\bE}_{s,i,0}\Big(g^{+}\Big(Z^{1}_{\ell/v(i)},Z^{2}_{\ell/v(i)}\Big)\1_{\{\wt{\gamma}_{1}>\ell/v(i)\}}\Big)+\wt{\bE}_{s,i,0}\Big(g^{+}\Big(Z^{1}_{\wt{\tau}_{\ell}^{+}},Z^{2}_{\wt{\tau}_{\ell}^{+}}\Big)\1_{\{\wt{\gamma}_{1}\leq\ell/v(i)\}}\Big)-g^{+}(s,i)\bigg)\nonumber\\
&\quad =\frac{1}{\ell}\,\wt{\bP}_{s,i,0}\bigg(\wt{\gamma}_{1}>\frac{\ell}{v(i)}\bigg)\bigg(g^{+}\bigg(s+\frac{\ell}{v(i)},i\bigg)-g^{+}(s,i)\bigg)-\frac{1}{\ell}\,\wt{\bP}_{s,i,0}\bigg(\wt{\gamma}_{1}\leq\frac{\ell}{v(i)}\bigg)g^{+}(s,i)\nonumber\\
&\qquad +\frac{1}{\ell}\,\wt{\bE}_{s,i,0}\Big(g^{+}\Big(Z^{1}_{\wt{\tau}_{\ell}^{+}},Z^{2}_{\wt{\tau}_{\ell}^{+}}\Big)\1_{\{\wt{\gamma}_{1}\leq\ell/v(i)\}}\Big)\nonumber\\
\label{eq:DecompGPlus} &\quad =:\cI_{1}(\ell)-\cI_{2}(\ell)+\cI_{3}(\ell).
\end{align}
Clearly,
\begin{align}\label{eq:LimitGPlus1}
\lim_{\ell\rightarrow 0+}\cI_{1}(\ell)=\frac{1}{v(i)}\frac{\partial g^{+}}{\partial s+}(s,i)
\end{align}
if and only if $g^{+}(\cdot,i)$ is right-differentiable at $s$. As for $\cI_{2}(\ell)$, \eqref{eq:TailDistJumpTildeZ2} implies that
\begin{align}\label{eq:LimitGPlus2}
\lim_{\ell\rightarrow 0+}\cI_{2}(\ell)=\lim_{\ell\rightarrow 0}\frac{1}{\ell}\bigg(1-\exp\bigg(\int_{s}^{s+\ell/v(i)}\!\Lambda_{u}(i,i)\,du\bigg)\bigg)g^{+}(s,i)=-\frac{\mathsf{\Lambda}_{s}(i,i)}{v(i)}g^{+}(s,i).
\end{align}

It remains to analyze the limit of $\cI_{3}(\ell)$, as $\ell\rightarrow 0+$. By \eqref{eq:hoverGamma1} and \eqref{eq:StrongMarkovCondPlus},
\begin{align*}
\cI_{3}(\ell)&=\frac{1}{\ell}\,\wt{\bE}_{s,i,0}\Big(g^{+}\Big(Z^{1}_{\wt{\tau}_{\ell}^{+}},Z^{2}_{\wt{\tau}_{\ell}^{+}}\Big)\1_{\{\wt{\gamma}_{1}\leq\ell/v(i),\,\wt{\gamma}_{1}\leq\wt{\tau}^{+}_{\ell}\}}\Big)\\
&=\frac{1}{\ell}\,\wt{\bE}_{s,i,0}\Big(\1_{\{\wt{\gamma}_{1}\leq\ell/v(i),\,\wt{\gamma}_{1}\leq\wt{\tau}^{+}_{\ell}\}}\,\wt{\bE}_{s,i,0}\Big(g^{+}\Big(Z^{1}_{\wt{\tau}_{\ell}^{+}},Z^{2}_{\wt{\tau}_{\ell}^{+}}\Big)\Big|\wt{\sF}_{\wt{\gamma}_{1}}\Big)\Big)\\
&=\frac{1}{\ell}\,\wt{\bE}_{s,i,0}\Big(\1_{\{\wt{\gamma}_{1}\leq\ell/v(i),\,\wt{\gamma}_{1}\leq\wt{\tau}_{\ell}^{+}\}}\,\wt{\bE}_{Z_{\wt{\gamma}_{1}}^{1},Z_{\wt{\gamma}_{1}}^{2},Z_{\wt{\gamma}_{1}}^{3}}\!\Big(g^{+}\Big(Z^{1}_{\wt{\tau}_{\ell}^{+}},Z^{2}_{\wt{\tau}_{\ell}^{+}}\Big)\Big)\Big).
\end{align*}
Since $Z^{3}_{\wt{\gamma}_{1}}\leq\ell$ on $\{\wt{\gamma}_{1}\leq\wt{\tau}_{\ell}^{+}\}$, by \eqref{eq:ExpTildePShift} and \eqref{eq:FuntfPlus}, we have
\begin{align*}
&\1_{\{\wt{\gamma}_{1}\leq\ell/v(i),\,\wt{\gamma}_{1}\leq\wt{\tau}_{\ell}^{+}\}}\,\wt{\bE}_{Z_{\wt{\gamma}_{1}}^{1},Z_{\wt{\gamma}_{1}}^{2},Z_{\wt{\gamma}_{1}}^{3}}\!\Big(g^{+}\Big(Z^{1}_{\wt{\tau}_{\ell}^{+}},Z^{2}_{\wt{\tau}_{\ell}^{+}}\Big)\Big)\\
&\quad =\1_{\{\wt{\gamma}_{1}\leq\ell/v(i),\,\wt{\gamma}_{1}\leq\wt{\tau}_{\ell}^{+}\}}\,\wt{\bE}_{t,j,a}\Big(g^{+}\Big(Z^{1}_{\wt{\tau}_{\ell}^{+}},Z^{2}_{\wt{\tau}_{\ell}^{+}}\Big)\Big)\Big|_{(t,j,a)=\big(Z_{\wt{\gamma}_{1}}^{1},Z_{\wt{\gamma}_{1}}^{2},Z_{\wt{\gamma}_{1}}^{3}\big)}\\
&\quad =\1_{\{\wt{\gamma}_{1}\leq\ell/v(i),\,\wt{\gamma}_{1}\leq\wt{\tau}_{\ell}^{+}\}}\,\wt{\bE}_{t,j,0}\Big(g^{+}\Big(Z^{1}_{\wt{\tau}_{\ell-a}^{+}},Z^{2}_{\wt{\tau}_{\ell-a}^{+}}\Big)\Big)\Big|_{(t,j,a)=\big(Z_{\wt{\gamma}_{1}}^{1},Z_{\wt{\gamma}_{1}}^{2},Z_{\wt{\gamma}_{1}}^{3}\big)}\\
&\quad =\1_{\{\wt{\gamma}_{1}\leq\ell/v(i),\,\wt{\gamma}_{1}\leq\wt{\tau}_{\ell}^{+}\}}\,f_{+}\big(Z_{\wt{\gamma}_{1}}^{1},Z_{\wt{\gamma}_{1}}^{2},\ell-Z_{\wt{\gamma}_{1}}^{3}\big).
\end{align*}
Thus, we can further decompose $\cI_{3}(\ell)$ as
\begin{align}
\cI_{3}(\ell)&=\frac{1}{\ell}\,\wt{\bE}_{s,i,0}\Big(\1_{\{\wt{\gamma}_{1}\leq\ell/v(i),\,\wt{\gamma}_{1}\leq\wt{\tau}_{\ell}^{+}\}}\,f_{+}\big(Z_{\wt{\gamma}_{1}}^{1},Z_{\wt{\gamma}_{1}}^{2},\ell-Z_{\wt{\gamma}_{1}}^{3}\big)\Big)\nonumber\\
&=\frac{1}{\ell}\,\wt{\bE}_{s,i,0}\Big(\1_{\{\wt{\gamma}_{1}\leq\ell/v(i),\,\wt{\gamma}_{1}\leq\wt{\tau}_{\ell}^{+}\}}\Big(f_{+}\big(Z^{1}_{\wt{\gamma}_{1}},Z^{2}_{\wt{\gamma}_{1}},\ell-Z^{3}_{\wt{\gamma}_{1}}\big)-f_{+}\big(s,Z^{2}_{\wt{\gamma}_{1}},\ell\big)\Big)\Big)\nonumber\\
&\quad +\frac{1}{\ell}\,\wt{\bE}_{s,i,0}\Big(\1_{\{\wt{\gamma}_{1}\leq\ell/v(i)\}}\,f_{+}\big(s,Z^{2}_{\wt{\gamma}_{1}},\ell\big)\Big)\nonumber\\
\label{eq:DecompGPlus3} &=:\cI_{31}(\ell)+\cI_{32}(\ell).
\end{align}
For $\cI_{31}(\ell)$, by \eqref{eq:ProcZ}, Lemma \ref{lem:UnifContf}, and \eqref{eq:DistFirstJumpTildeZ2}, we have
\begin{align}
\lim_{\ell\rightarrow 0+}\big|\cI_{31}(\ell)\big|&=\lim_{\ell\rightarrow 0+}\frac{1}{\ell}\Big|\wt{\bE}_{s,i,0}\Big(\1_{\{\wt{\gamma}_{1}\leq\ell/v(i),\,\wt{\gamma}_{1}\leq\wt{\tau}_{\ell}^{+}\}}\Big(f_{+}\big(s+\wt{\gamma}_{1},Z^{2}_{\wt{\gamma}_{1}},\ell-Z^{3}_{\wt{\gamma}_{1}}\big)-f_{+}\big(s,Z^{2}_{\wt{\gamma}_{1}},\ell\big)\Big)\Big)\Big|\nonumber\\
&\leq\lim_{\ell\rightarrow 0+}\frac{1}{\ell}\wt{\bE}_{s,i,0}\bigg(\1_{\{\wt{\gamma}_{1}\leq\ell/v(i),\,\wt{\gamma}_{1}\leq\wt{\tau}_{\ell}^{+}\}}\sup_{j\in\mathbf{E}}\sup_{\ell'\in[0,\ell],\,|s'-s|\in[0,\ell/\underline{v}]}\big|f_{+}(s',j,\ell')-f_{+}(s,j,\ell)\big|\bigg)\nonumber\\
&\leq\lim_{\ell\rightarrow 0+}\sup_{j\in\mathbf{E}}\sup_{\ell'\in[0,\ell],\,|s'-s|\in[0,\ell/\underline{v}]}\big|f_{+}(s',j,\ell')-f_{+}(s,j,\ell)\big|\cdot\frac{1}{\ell}\,\wt{\bP}_{s,i,0}\bigg(\wt{\gamma}_{1}\leq\frac{\ell}{v(i)}\bigg)\nonumber\\
\label{eq:LimitGPlus31} &\leq\frac{K}{\underline{v}}\lim_{\ell\rightarrow 0+}\sup_{j\in\mathbf{E}}\sup_{\ell'\in[0,\ell],\,|s'-s|\in[0,\ell/\underline{v}]}\big|f_{+}(s',j,\ell')-f_{+}(s,j,\ell)\big|=0,
\end{align}
where we recall that $\underline{v}=\min_{i\in\mathbf{E}}|v(i)|$. To study the limit of $\cI_{32}(\ell)$ as $\ell\rightarrow 0+$, we first rewrite $\cI_{32}(\ell)$ as
\begin{align}\label{eq:DecompGPlus32}
\cI_{32}(\ell)=\frac{1}{\ell}\sum_{j\in\mathbf{E}\setminus\{i\}}\wt{\bP}_{s,i,0}\bigg(\wt{\gamma}_{1}\leq\frac{\ell}{v(i)},\,Z_{\wt{\gamma}_{1}}^{2}=j\bigg)f_{+}(s,j,\ell).
\end{align}
Note that, for any $j\in\mathbf{E}\setminus\{i\}$, the probability in \eqref{eq:DecompGPlus32} can be further decomposed as
\begin{align}
&\wt{\bP}_{s,i,0}\bigg(\wt{\gamma}_{1}\leq\frac{\ell}{v(i)},\,Z_{\wt{\gamma}_{1}}^{2}=j\bigg)\nonumber\\
&\quad =\wt{\bP}_{s,i,0}\bigg(\wt{\gamma}_{1}\leq\frac{\ell}{v(i)},\,Z_{\wt{\gamma}_{1}}^{2}=j,\,Z_{\ell/v(i)}^{2}=j\bigg)+\wt{\bP}_{s,i,0}\bigg(\wt{\gamma}_{1}\leq\frac{\ell}{v(i)},\,Z_{\wt{\gamma}_{1}}^{2}=j,\,Z_{\ell/v(i)}^{2}\neq j\bigg)\nonumber\\
&\quad =\wt{\bP}_{s,i,0}\bigg(\wt{\gamma}_{1}\leq\frac{\ell}{v(i)},\,Z_{\ell/v(i)}^{2}=j\bigg)-\wt{\bP}_{s,i,0}\bigg(\wt{\gamma}_{1}\leq\frac{\ell}{v(i)},\,Z_{\wt{\gamma}_{1}}^{2}\neq j,\,Z_{\ell/v(i)}^{2}=j\bigg)\nonumber\\
\label{eq:DecompGPlusProb32} &\qquad +\wt{\bP}_{s,i,0}\bigg(\wt{\gamma}_{1}\leq\frac{\ell}{v(i)},\,Z_{\wt{\gamma}_{1}}^{2}=j,\,Z_{\ell/v(i)}^{2}\neq j\bigg).
\end{align}
By \eqref{eq:TranProbTildeX}, \eqref{eq:LawXXStar}, and \eqref{eq:DefGenXStar}, for $j\neq i$,
\begin{align*}
\lim_{\ell\rightarrow 0+}\frac{1}{\ell}\,\wt{\bP}_{s,i,0}\bigg(\wt{\gamma}_{1}\leq\frac{\ell}{v(i)},\,Z_{\ell/v(i)}^{2}=j\bigg)&=\lim_{\ell\rightarrow 0+}\frac{1}{\ell}\,\wt{\bP}_{s,i,0}\big(Z_{\ell/v(i)}^{2}=j\big)=\lim_{\ell\rightarrow 0+}\frac{1}{\ell}\,\bP_{s,(i,0)}\big(X_{s+\ell/v(i)}=j\big)\\
&=\lim_{\ell\rightarrow 0+}\frac{1}{\ell}\,\bP^{*}_{s,i}\big(X_{s+\ell/v(i)}^{*}=j\big)=\frac{\mathsf\Lambda_{s}(i,j)}{v(i)},
\end{align*}
which, together with Lemma \ref{lem:UnifContf}, gives
\begin{align}\label{eq:LimitGPlus321}
\lim_{\ell\rightarrow 0+}\frac{1}{\ell}\sum_{j\in\mathbf{E}\setminus\{i\}}\wt{\bP}_{s,i,0}\bigg(\wt{\gamma}_{1}\leq\frac{\ell}{v(i)},\,Z_{\ell/v(i)}^{2}=j\bigg)f_{+}(s,j,\ell)=\sum_{j\in\mathbf{E}\setminus\{i\}}\frac{\mathsf\Lambda_{s}(i,j)}{v(i)}f_{+}(s,j,0).
\end{align}
Moreover, denoting by $\wt{\gamma}_{2}$ the second jump time of $\wt{Z}^2$, then by \eqref{eq:FuntfPlus} and \eqref{eq:DistSecondJumpTildeZ2}, we have
\begin{align}
&\lim_{\ell\rightarrow 0+}\bigg|\frac{1}{\ell}\sum_{j\in\mathbf{E}\setminus\{i\}}\wt{\bP}_{s,i,0}\bigg(\wt{\gamma}_{1}\leq\frac{\ell}{v(i)},\,Z_{\wt{\gamma}_{1}}^{2}\neq j,\,Z_{\ell/v(i)}^{2}=j\bigg)f_{+}(s,j,\ell)\bigg|\nonumber\\
&\quad\leq\lim_{\ell\rightarrow 0+}\frac{1}{\ell}\sum_{j\in\mathbf{E}\setminus\{i\}}\wt{\bP}_{s,i,0}\bigg(\wt{\gamma}_{2}\leq\frac{\ell}{v(i)},\,Z_{\ell/v(i)}^{2}=j\bigg)\big\|g^{+}\big\|_{\infty}\nonumber\\
\label{eq:LimitGPlus322} &\quad\leq\lim_{\ell\rightarrow 0+}\frac{1}{\ell}\,\wt{\bP}_{s,i,0}\bigg(\wt{\gamma}_{2}\leq\frac{\ell}{v(i)}\bigg)\big\|g^{+}\big\|_{\infty}\leq\lim_{\ell\rightarrow 0+}\frac{2K^{2}\ell}{\underline{v}^{2}}\big\|g^{+}\big\|_{\infty}=0,
\end{align}
and similarly,
\begin{align}\label{eq:LimitGPlus323}
\lim_{\ell\rightarrow 0+}\frac{1}{\ell}\bigg|\sum_{j\in\mathbf{E}\setminus\{i\}}\wt{\bP}_{s,i,0}\bigg(\wt{\gamma}_{1}\leq\frac{\ell}{v(i)},\,Z_{\wt{\gamma}_{1}}^{2}=j,\,Z_{\ell/v(i)}^{2}\neq j\bigg)f_{+}(s,j,\ell)\bigg|=0.
\end{align}
Combining \eqref{eq:DecompGPlus32}$-$\eqref{eq:LimitGPlus323} leads to
\begin{align}\label{eq:LimitGPlus32}
\lim_{\ell\rightarrow 0+}\cI_{32}(\ell)=\!\!\sum_{j\in\mathbf{E}\setminus\{i\}}\!\!\frac{\mathsf\Lambda_{s}(i,j)}{v(i)}f_{+}(s,j,0)=\!\!\sum_{j\in\mathbf{E}_{+}\setminus\{i\}}\!\!\frac{\mathsf{\Lambda}_{s}(i,j)}{v(i)}\,g^{+}(s,j)+\!\sum_{j\in\mathbf{E}_{-}}\!\frac{\mathsf{\Lambda}_{s}(i,j)}{v(i)}\big(J^{+}g^{+}\big)(s,j),
\end{align}
where the last equality is due to \eqref{eq:Trivfplus} and \eqref{eq:JPlusfPlus}.

Therefore, from \eqref{eq:DecompGPlus3}, \eqref{eq:LimitGPlus31}, and \eqref{eq:LimitGPlus32}, we have
\begin{align}\label{eq:LimitGPlus3}
\lim_{\ell\rightarrow 0+}\cI_{3}(\ell)=\sum_{j\in\mathbf{E}_{+}\setminus\{i\}}\frac{\mathsf{\Lambda}_{s}(i,j)}{v(i)}\,g^{+}(s,j)+\sum_{j\in\mathbf{E}_{-}}\frac{\mathsf{\Lambda}_{s}(i,j)}{v(i)}\left(J^{+}g^{+}\right)(s,j).
\end{align}
Combining \eqref{eq:DecompGPlus}$-$\eqref{eq:LimitGPlus2} and \eqref{eq:LimitGPlus3}, we conclude that the limit in \eqref{eq:DefGenGPlus} exists for every $(s,i)\in\overline{\sX_{+}}$ if and only if $g(\cdot,i)$ is right-differentiable on $\bR_{+}$ for each $i$, and that for such $g^{+}\in C_{0}(\overline{\sX_{+}})$, \eqref{eq:LimitGPlus} holds true for any $(s,i)\in\overline{\sX_{+}}$.

\medskip
\noindent
\textbf{(ii)} We now show that $\sD(G^{+})=C_{0}^{1}(\overline{\sX_{+}})$. Toward this end we define
\begin{align*}
\sL(G^{+}):=\big\{g^{+}\in C_{0}(\overline{\sX_{+}}):\,\text{the limit in \eqref{eq:DefGenGPlus} exists for all }(s,i)\in\sX_{+}\,\,\,\text{and}\,\,\,G^{+}g^{+}\in C_{0}(\overline{\sX_{+}})\big\}.
\end{align*}
Since $(\cP^{+}_{\ell})_{\ell\in\bR_{+}}$ is a Feller semigroup on $C_{0}(\overline{\sX_{+}})$ (cf. Step 2), it follows from \cite[Theorem 1.33]{BottcherSchillingWang2013} that $G^{+}$ is the strong generator of $(\cP^{+}_{\ell})_{\ell\in\bR_{+}}$ with $\sD(G^{+})=\sL(G^{+})$. Hence, we only need to show that $\sL(G^{+})=C_{0}^{1}(\overline{\sX_{+}})$.

We first show that $\sL(G^{+})\subset C_{0}^{1}(\overline{\sX_{+}})$. For any $g^{+}\in\sL(G^{+})$, it was shown in Step~3 (i) that
\begin{align}\label{eq:GPlusStep3}
\big(G^{+}\!g^{+}\big)(s,i)\!=\!\frac{1}{v(i)}\bigg(\!\frac{\partial_{+}g^{+}}{\partial s}(s,i)\!+\!\!\sum_{j\in\mathbf{E}_{+}}\!\!\mathsf{\Lambda}_{s}(i,j)g^{+}\!(s,j)\!+\!\!\sum_{j\in\mathbf{E}_{-}}\!\!\mathsf{\Lambda}_{s}(i,j)\big(J^{+}\!g^{+}\big)(s,j)\!\!\bigg),\,\,(s,i)\!\in\!\overline{\sX_{+}},
\end{align}
where the right-hand side, as a function of $(s,i)$, belongs to $C_{0}(\overline{\sX_{+}})$. By Lemma \ref{lem:UnifContf}, we have $J^{+}g^{+}\in C_{0}(\overline{\sX_{-}})$. This, together with Assumption \ref{assump:GenLambda} (ii), ensures that
\begin{align}\label{eq:GPlusStep3part}
\sum_{j\in\mathbf{E}_{+}}\mathsf{\Lambda}_{s}(i,j)g^{+}(s,j)+\sum_{j\in\mathbf{E}_{-}}\mathsf{\Lambda}_{s}(i,j)\big(J^{+}g^{+}\big)(s,j),
\end{align}
as a function of $(s,i)\in\overline{\sX_{+}}$, belongs to $C_{0}(\overline{\sX_{+}})$. Thus, we must have $\partial_{+}g^{+}/\partial s$ exists at any $(s,i)\in\sX_{+}$, $\partial_{+}g^{+}(\infty,\partial)/\partial s=0$, and $\partial_{+}g^{+}/\partial s\in C_{0}(\overline{\sX_{+}})$. Therefore, $\partial g^{+}/\partial s$ exists and belongs to $C_{0}(\overline{\sX_{+}})$, i.e., $g^{+}\in C_{0}^{1}(\overline{\sX_{+}})$.

To show $C_{0}^{1}(\overline{\sX_{+}})\subset\sL(G^{+})$, we first note that for $g^{+}\in C_{0}^{1}(\overline{\sX_{+}})$, Step 3 (i) shows that the limit in \eqref{eq:DefGenGPlus} exists for every $(s,i)\in\overline{\sX_{+}}$, and that \eqref{eq:GPlusStep3} holds true. Since $g^{+}\in C_{0}(\overline{\sX_{+}})$, the same argument as above implies that \eqref{eq:GPlusStep3part}, as a function of $(s,i)$, belongs to $C_{0}(\overline{\sX_{+}})$. Hence, $G^{+}g^{+}\in C_{0}(\overline{\sX_{+}})$. The proof in Step 3 is now complete.

\medskip
\noindent
\textbf{Step 4.} In this step we will show that $J^{+}$ satisfies the condition $(a^{+})$(ii), that is for any $g^{+}\in C_{0}^{1}(\overline{\sX_{+}})$ we have $J^{+}g^{+}\in C_{0}^{1}(\overline{\sX_{-}})$. We will also show that
    \begin{align}\label{eq:WHplusD}
    \frac{\partial}{\partial s}\big(J^{+}g^{+}\big)=\big(-\wt{\mC}-\wt{\mD}J^{+}+\mV^{-}J^{+}G^{+}\big)\,g^{+},\quad g^{+}\in C_{0}^{1}(\overline{\sX_{+}}).
    \end{align}
We fix $g^{+}\in C_{0}^{1}(\overline{\sX_{+}})$ for the rest of the proof.

To begin with, we claim that in order to prove $J^{+}g^{+}\in C_{0}^{1}(\overline{\sX_{-}})$ and \eqref{eq:WHplusD}, it is sufficient to show that
\begin{align}\label{eq:WHplusDRightDev}
\frac{\partial_{+}}{\partial s}J^{+}g^{+}=\big(-\wt{\mC}-\wt{\mD}J^{+}+\mV^{-}J^{+}G^{+}\big)\,g^{+}.
\end{align}
In fact, since $g^{+}\in C_{0}^{1}(\overline{\sX_{+}})$, Step 3 shows that $G^{+}g^{+}\in C_{0}(\overline{\sX_{+}})$. Hence, by Lemma \ref{lem:UnifContf}, we have $J^{+}g^{+}\in C_{0}(\overline{\sX_{-}})$ and $J^{+}G^{+}g^{+}\in C_{0}(\overline{\sX_{-}})$. Given the definition of $\wt{\mC}$ and $\wt{\mD}$ and invoking  Assumption \ref{assump:GenLambda}, we conclude that $(-\wt{\mC}-\wt{\mD}J^{+}+\mV^{-}J^{+}G^{+})g^{+}\in C_{0}(\overline{\sX_{-}})$. Thus, indeed, \eqref{eq:WHplusDRightDev} implies that $J^{+}g^{+}\in C_{0}^{1}(\overline{\sX_{-}})$ and that \eqref{eq:WHplusD} holds.

To prove \eqref{eq:WHplusDRightDev}, it is sufficient to consider $(s,i)\in\sX_{-}$ only, since both sides of \eqref{eq:WHplusDRightDev} {are} equal to zero for $(s,i)=(\infty,\partial)$. In view of \eqref{eq:DefJPlusTildeP}, we will evaluate
\begin{align*}
\lim_{r\rightarrow 0+}\frac{1}{r}\bigg(\wt{\bE}_{s+r,i,0}\Big(g^{+}\Big(Z^{1}_{\wt{\tau}_{0}^{+}},Z^{2}_{\wt{\tau}_{0}^{+}}\Big)\Big)-\wt{\bE}_{s,i,0}\Big(g^{+}\Big(Z^{1}_{\wt{\tau}_{0}^{+}},Z^{2}_{\wt{\tau}_{0}^{+}}\Big)\Big)\bigg),\quad (s,i)\in\sX_{-}.
\end{align*}
For any $r>0$, by \eqref{eq:ProcZ} and \eqref{eq:FuntfPlus},
\begin{align}
&\wt{\bE}_{s+r,i,0}\Big(g^{+}\Big(Z^{1}_{\wt{\tau}_{0}^{+}},Z^{2}_{\wt{\tau}_{0}^{+}}\Big)\Big)-\wt{\bE}_{s,i,0}\Big(g^{+}\Big(Z^{1}_{\wt{\tau}_{0}^{+}},Z^{2}_{\wt{\tau}_{0}^{+}}\Big)\Big)\nonumber\\
&\quad =\wt{\bE}_{s,i,0}\bigg(\Big(\wt{\bE}_{s+r,i,0}\Big(g^{+}\Big(Z^{1}_{\wt{\tau}_{0}^{+}},Z^{2}_{\wt{\tau}_{0}^{+}}\Big)\Big)-f_{+}\big(Z^{1}_{r},Z^{2}_{r},0\big)\Big)\bigg)+\wt{\bE}_{s,i,0}\Big(f_{+}\big(Z^{1}_{r},Z^{2}_{r},0\big)-g^{+}\Big(Z^{1}_{\wt{\tau}_{0}^{+}},Z^{2}_{\wt{\tau}_{0}^{+}}\Big)\Big)\nonumber\\
&\quad =\wt{\bE}_{s,i,0}\Big(f_{+}(s+r,i,0\big)-f_{+}\big(s+r,Z^{2}_{r},0\big)\Big)+\wt{\bE}_{s,i,0}\Big(f_{+}\big(s+r,Z^{2}_{r},0\big)-g^{+}\Big(Z^{1}_{\wt{\tau}_{0}^{+}},Z^{2}_{\wt{\tau}_{0}^{+}}\Big)\Big)\nonumber\\
&\quad =\wt{\bE}_{s,i,0}\Big(f_{+}(s\!+\!r,i,0\big)-f_{+}\big(s\!+\!r,Z^{2}_{r},0\big)\Big)+\wt{\bE}_{s,i,0}\Big(\1_{\{\wt{\tau}_{0}^{+}\leq r\}}\Big(f_{+}\big(s\!+\!r,Z^{2}_{r},0\big)-g^{+}\Big(s\!+\!\wt{\tau}_{0}^{+},Z^{2}_{\wt{\tau}_{0}^{+}}\Big)\Big)\Big)\nonumber\\
\label{eq:PreDecompPartialJPlus} &\qquad +\wt{\bE}_{s,i,0}\Big(\1_{\{\wt{\tau}_{0}^{+}>r\}}\Big(f_{+}\big(s+r,Z^{2}_{r},0\big)-g^{+}\Big(Z^{1}_{\wt{\tau}_{0}^{+}},Z^{2}_{\wt{\tau}_{0}^{+}}\Big)\Big)\Big).
\end{align}
Clearly, by \eqref{eq:RangeXtauPlus} and \eqref{eq:Trivfplus}, $g^{+}(s+\wt{\tau}_{0}^{+},Z^{2}_{\wt{\tau}_{0}^{+}})=f_{+}(s+\wt{\tau}_{0}^{+},Z^{2}_{\wt{\tau}_{0}^{+}},0)$. Hence, the second term in \eqref{eq:PreDecompPartialJPlus} can be decomposed as
\begin{align}
&\wt{\bE}_{s,i,0}\Big(\1_{\{\wt{\tau}_{0}^{+}\leq r\}}\Big(f_{+}\big(s+r,Z^{2}_{r},0\big)-g^{+}\Big(s+\wt{\tau}_{0}^{+},Z^{2}_{\wt{\tau}_{0}^{+}}\Big)\Big)\Big)\nonumber\\
&\quad =\wt{\bE}_{s,i,0}\Big(\1_{\{\wt{\tau}_{0}^{+}\leq r\}}\Big(f_{+}\big(s+r,Z^{2}_{r},0\big)-f_{+}\Big(s+\wt{\tau}_{0}^{+},Z^{2}_{\wt{\tau}_{0}^{+}},0\Big)\Big)\Big)\nonumber\\
&\quad =\wt{\bE}_{s,i,0}\Big(\1_{\{\wt{\tau}_{0}^{+}\leq r\}}\Big(f_{+}\big(s+r,Z^{2}_{r},0\big)-f_{+}\Big(s+r,Z^{2}_{\wt{\tau}_{0}^{+}},0\Big)\Big)\Big)\nonumber\\
\label{eq:PreDecompPartialJPlus2} &\qquad +\wt{\bE}_{s,i,0}\Big(\1_{\{\wt{\tau}_{0}^{+}\leq r\}}\Big(f_{+}\Big(s+r,Z^{2}_{\wt{\tau}_{0}^{+}},0\Big)-f_{+}\Big(s+\wt{\tau}_{0}^{+},Z^{2}_{\wt{\tau}_{0}^{+}},0\Big)\Big)\Big).
\end{align}
Moreover, by \eqref{eq:StrongMarkovCondPlus}, \eqref{eq:ExpTildePShift}, \eqref{eq:FuntfPlus}, and \eqref{eq:ProcZ}, we can decompose the second term in \eqref{eq:PreDecompPartialJPlus} as
\begin{align}
&\wt{\bE}_{s,i,0}\Big(\1_{\{\wt{\tau}_{0}^{+}>r\}}\Big(f_{+}\big(s+r,Z^{2}_{r},0\big)-g^{+}\Big(Z^{1}_{\wt{\tau}_{0}^{+}},Z^{2}_{\wt{\tau}_{0}^{+}}\Big)\Big)\Big)\nonumber\\
&\quad =\wt{\bE}_{s,i,0}\bigg(\1_{\{\wt{\tau}_{0}^{+}>r\}}\Big(f_{+}\big(s+r,Z^{2}_{r},0\big)-\wt{\bE}_{s,i,0}\Big(g^{+}\Big(Z^{1}_{\wt{\tau}_{0}^{+}},Z^{2}_{\wt{\tau}_{0}^{+}}\Big)\Big|\sF_{r}\Big)\Big)\bigg)\nonumber\\
&\quad =\wt{\bE}_{s,i,0}\left(\1_{\{\wt{\tau}_{0}^{+}>r\}}\bigg(f_{+}\big(s+r,Z^{2}_{r},0\big)-\wt{\bE}_{t,j,a}\Big(g^{+}\Big(Z^{1}_{\wt{\tau}_{0}^{+}},Z^{2}_{\wt{\tau}_{0}^{+}}\Big)\Big)\Big|_{(t,j,a)=\big(Z^{1}_{r},Z^{2}_{r},Z^{3}_{r}\big)}\bigg)\right)\nonumber\\
&\quad =\wt{\bE}_{s,i,0}\left(\1_{\{\wt{\tau}_{0}^{+}>r\}}\bigg(f_{+}\big(s+r,Z^{2}_{r},0\big)-\wt{\bE}_{t,j,0}\Big(g^{+}\Big(Z^{1}_{\wt{\tau}_{-a}^{+}},Z^{2}_{\wt{\tau}_{-a}^{+}}\Big)\Big)\Big|_{(t,j,a)=\big(Z^{1}_{r},Z^{2}_{r},Z^{3}_{r}\big)}\bigg)\right)\nonumber\\
&\quad =\wt{\bE}_{s,i,0}\Big(\1_{\{\wt{\tau}_{0}^{+}>r\}}\big(f_{+}(s+r,Z^{2}_{r},0)-f_{+}\big(Z^{1}_{r},Z^{2}_{r},-Z^{3}_{r}\big)\big)\Big)\nonumber\\
\label{eq:PreDecompPartialJPlus3} &\quad =\wt{\bE}_{s,i,0}\Big(\1_{\{\wt{\tau}_{0}^{+}>r\}}\big(f_{+}(s+r,Z^{2}_{r},0)-f_{+}\big(s+r,Z^{2}_{r},-Z^{3}_{r}\big)\big)\Big).
\end{align}
Hence, by combining \eqref{eq:PreDecompPartialJPlus}$-$\eqref{eq:PreDecompPartialJPlus3}, we obtain that
\begin{align}
&\frac{1}{r}\bigg(\wt{\bE}_{s+r,i,0}\Big(g^{+}\Big(Z^{1}_{\wt{\tau}_{0}^{+}},Z^{2}_{\wt{\tau}_{0}^{+}}\Big)\Big)-\wt{\bE}_{s,i,0}\Big(g^{+}\Big(Z^{1}_{\wt{\tau}_{0}^{+}},Z^{2}_{\wt{\tau}_{0}^{+}}\Big)\Big)\bigg)\nonumber\\
&\quad =\frac{1}{r}\,\wt{\bE}_{s,i,0}\Big(f_{+}(s+r,i,0)-f_{+}\big(s+r,Z^{2}_{r},0\big)\Big)\nonumber\\
&\qquad +\frac{1}{r}\,\wt{\bE}_{s,i,0}\Big(\1_{\{\wt{\tau}_{0}^{+}\leq r\}}\big(f_{+}\big(s+r,Z^{2}_{r},0\big)-f_{+}\big(s+r,Z^{2}_{\wt{\tau}_{0}^{+}},0\big)\big)\Big)\nonumber\\
&\qquad +\frac{1}{r}\,\wt{\bE}_{s,i,0}\Big(\1_{\{\wt{\tau}_{0}^{+}\leq r\}}\Big(f_{+}\Big(s+r,Z^{2}_{\wt{\tau}_{0}^{+}},0\Big)-f_{+}\Big(s+\wt{\tau}_{0}^{+},Z^{2}_{\wt{\tau}_{0}^{+}},0\Big)\Big)\Big)\nonumber\\
&\qquad +\frac{1}{r}\,\wt{\bE}_{s,i,0}\Big(\1_{\{\wt{\tau}_{0}^{+}>r\}}\big(f_{+}\big(s+r,Z^{2}_{r},0)-f_{+}\big(s+r,Z^{2}_{r},-Z^{3}_{r}\big)\big)\Big)\nonumber\\
\label{eq:DecompPartialJPlus} &\quad =:\cJ_{1}(r)+\cJ_{2}(r)+\cJ_{3}(r)+\cJ_{4}(r).
\end{align}
Next, we will analyze the limit of $\cJ_{k}(r)$, $k=1,2,3,4$, as $r\rightarrow 0+$.

We begin with evaluating the limit of $\cJ_{1}(r)$ as $r\rightarrow 0+$. By \eqref{eq:Probz}, \eqref{eq:ProcZ}, and \eqref{eq:LawXXStar}, and using the evolution system $\mU^{*}=(\mU^{*}_{s,t})_{0\leq s\leq t<\infty}$ defined as in \eqref{eq:DefEvolSytXStar}, we have
\begin{align}
\cJ_{1}(r)&=\frac{1}{r}\,\bE_{s,(i,a)}\big(f_{+}(s+r,i,0)-f_{+}\big(s+r,X_{s+r},0\big)\big)\nonumber\\
&=\frac{1}{r}\,\bE^{*}_{s,i}\big(f_{+}(s+r,i,0)-f_{+}\big(s+r,X_{s+r}^{*},0\big)\big)=-\frac{1}{r}\big(\big(\mU_{s,s+r}^{*}-\mI\big)f_{+}(s+r,\cdot,0)\big)(i)\nonumber\\
\label{eq:DecompJ1r} &= -\frac{1}{r}\big(\big(\mU_{s,s+r}^{*}-\mI\big)f_{+}(s,\cdot,0)\big)(i)+\frac{1}{r}\big(\big(\mU_{s,s+r}^{*}-\mI\big)\big(f_{+}(s,\cdot,0)-f_{+}(s+r,\cdot,0)\big)\big)(i).
\end{align}
It follows immediately from \eqref{eq:DefGenXStar} that
\begin{align}\label{eq:LimitJ1r1}
\lim_{r\rightarrow 0+}\frac{1}{r}\big(\big(\mU_{s,s+r}^{*}-\mI\big)f_{+}(s,\cdot,0)\big)(i)=\sum_{j\in\mathbf{E}}\mathsf{\Lambda}_{s}(i,j)f_{+}(s,j,0).
\end{align}
Moreover, by \eqref{eq:EvolDE}, Assumption \ref{assump:GenLambda} (so that $\|\mathsf{\Lambda}^{*}_{t}\|_{\infty}\leq 2K$ for all $t\in\overline{\bR}_{+}$), Lemma \ref{lem:UnifContf}, and the fact that $\mU^{*}_{s,t}$ is a contraction map, we have
\begin{align}
&\lim_{r\rightarrow 0+}\frac{1}{r}\big|\big(\big(\mU_{s,s+r}^{*}-\mI\big)\big(f_{+}(s,\cdot,0)-f_{+}(s+r,\cdot,0)\big)\big)(i)\big|\nonumber\\
&\quad\leq\lim_{r\rightarrow 0+}\frac{1}{r}\big\|\mU_{s,s+r}^{*}-\mI\big\|_{\infty}\sup_{(s,j)\in\sX}\big|f_{+}(s,j,0)-f_{+}(s+r,j,0)\big|\nonumber\\
&\quad\leq\lim_{r\rightarrow 0+}\frac{1}{r}\int_{s}^{s+r}\big\|\mU_{s,t}^{*}\big\|_{\infty}\big\|\mathsf{\Lambda}^{*}_{t}\big\|_{\infty}dt\cdot\sup_{(s,j)\in\sX}\big|f_{+}(s,j,0)-f_{+}(s+r,j,0)\big|\nonumber\\
\label{eq:LimitJ1r2} &\quad\leq 2K\lim_{r\rightarrow 0+}\sup_{(s,j)\in\sX}\big|f_{+}(s,j,0)-f_{+}(s+r,j,0)\big|=0.
\end{align}
Combining \eqref{eq:DecompJ1r}$-$\eqref{eq:LimitJ1r2} leads to
\begin{align}\label{eq:LimitJ1r}
\lim_{r\rightarrow 0+}\cJ_{1}(r)=-\sum_{j\in\mathbf{E}}\mathsf{\Lambda}_{s}(i,j)\,f_{+}(s,j,0).
\end{align}

Next, we will study the limits of $\cJ_{2}(r)$ and $\cJ_{3}(r)$ as $r\rightarrow 0+$. Since $i\in\mathbf{E}_{-}$, $Z^{2}$ must have at least one jump to $\mathbf{E}_{+}$ before $Z^{3}$ (which coincides with $\int_{0}^{\cdot}v(Z^{2}_{u})du$ in view of \eqref{eq:DistTildeZ3IntTildeZ2}) can upcross the level $0$, i.e., $\wt{\bP}_{s,i,0}(\wt{\gamma}_{1}\leq\wt{\tau}_{0}^{+})=1$, where we recall that $\wt{\gamma}_{1}$ denotes the first jump time of $Z^{2}$. Hence, by \eqref{eq:DistFirstJumpTildeZ2} and Lemma \ref{lem:UnifContf},
\begin{align}
\lim_{r\rightarrow 0+}\big|\cJ_{3}(r)\big|&=\lim_{r\rightarrow 0+}\frac{1}{r}\Big|\wt{\bE}_{s,i,0}\Big(\1_{\{\wt{\gamma}_{1}\leq\wt{\tau}_{0}^{+}\leq r\}}\Big(f_{+}\Big(s+r,Z^{2}_{\wt{\tau}_{0}^{+}},0\Big)-f_{+}\Big(s+\wt{\tau}_{0}^{+},Z^{2}_{\wt{\tau}_{0}^{+}},0\Big)\Big)\Big)\Big|\nonumber\\
&\leq\lim_{r\rightarrow 0+}\frac{1}{r}\,\wt{\bP}_{s,i,0}\big(\wt{\gamma}_{1}\leq r\big)\,\sup_{(r',j)\in[0,r]\times\mathbf{E}}\big|f_{+}(s+r',j,0)-f_{+}(s,j,0)\big|\nonumber\\
\label{eq:LimitJ3r} &\leq K\,\lim_{r\rightarrow 0+}\sup_{(r',j)\in[0,r]\times\mathbf{E}}\big|f_{+}(s+r',j,0)-f_{+}(s,j,0)\big|=0.
\end{align}
Moreover, note that $\1_{\{\wt{\tau}_{0}^{+}\leq r\}}(f_{+}(s+r,Z^{2}_{r},0)-f_{+}(s+r,Z^{2}_{\wt{\tau}_{0}^{+}},0))$ does not vanish only if $Z^{2}_{r}\neq Z^{2}_{\wt{\tau}_{0}^{+}}$, so $Z^{2}$ must jump at least twice before time $r$. Hence, by \eqref{eq:DistSecondJumpTildeZ2} and \eqref{eq:FuntfPlus}, we have
\begin{align}
\lim_{r\rightarrow 0+}\big|\cJ_{2}(r)\big|&=\lim_{r\rightarrow 0+}\frac{1}{r}\Big|\wt{\bE}_{s,i,0}\Big(\1_{\{\wt{\tau}_{0}^{+}\leq r,\,\wt{\gamma}_{2}\leq r\}}\big(f_{+}\big(s+r,Z^{2}_{r},0\big)-f_{+}\big(s+r,Z^{2}_{\wt{\tau}_{0}^{+}},0\big)\big)\Big)\Big|\nonumber\\
\label{eq:LimitJ2r} &\leq 2\big\|g^{+}\big\|_{\infty}\cdot\lim_{r\rightarrow 0+}\frac{1}{r}\,\wt{\bP}_{s,i,0}\big(\wt{\gamma}_{2}\leq r\big)\leq 2K^{2}\big\|g^{+}\big\|_{\infty}\cdot\lim_{r\rightarrow 0+}r=0,
\end{align}
where we recall that $\wt{\gamma}_{2}$ denotes the second jump time of $Z^{2}$.

Finally, we study the limit of $\cJ_{4}(r)$, as $r\rightarrow 0+$, by further decomposing $\cJ_{4}(r)$ as
\begin{align}
\cJ_{4}(r)&=\frac{1}{r}\,\wt{\bE}_{s,i,0}\Big(\1_{\{\wt{\tau}_{0}^{+}>r\}}\big(f_{+}\big(s+r,Z^{2}_{r},0\big)-f_{+}\big(s+r,Z^{2}_{r},-v(i)r\big)\big)\Big)\nonumber\\
&\quad +\frac{1}{r}\,\wt{\bE}_{s,i,0}\Big(\1_{\{\wt{\tau}_{0}^{+}>r\}}\big(f_{+}\big(s+r,Z^{2}_{r},-v(i)r\big)-f_{+}\big(s+r,Z^{2}_{r},-Z^{3}_{r}\big)\big)\Big)\nonumber\\
&=\frac{1}{r}\,\wt{\bE}_{s,i,0}\big(f_{+}\big(s+r,Z^{2}_{r},0\big)-f_{+}\big(s+r,Z^{2}_{r},-v(i)r\big)\big)\nonumber\\
&\quad -\frac{1}{r}\,\wt{\bE}_{s,i,0}\Big(\1_{\{\wt{\tau}_{0}^{+}\leq r\}}\big(f_{+}\big(s+r,Z^{2}_{r},0\big)-f_{+}\big(s+r,Z^{2}_{r},-v(i)r\big)\big)\Big)\nonumber\\
&\quad +\frac{1}{r}\,\wt{\bE}_{s,i,0}\Big(\1_{\{\wt{\tau}_{0}^{+}>r\}}\big(f_{+}\big(s+r,Z^{2}_{r},-v(i)r\big)-f_{+}\big(s+r,Z^{2}_{r},-Z^{3}_{r}\big)\big)\Big)\nonumber\\
\label{eq:DecompJ4r} &=:\cJ_{41}(r)+\cJ_{42}(r)+\cJ_{43}(r).
\end{align}
For $\cJ_{41}(r)$, by \eqref{eq:Probz}, \eqref{eq:ProcZ}, \eqref{eq:LawXXStar}, and \eqref{eq:DefEvolSytXStar}, we have
\begin{align*}
\cJ_{41}(r)&=\frac{1}{r}\,\bE_{s,(i,0)}\big(f_{+}\big(s+r,X_{s+r},0\big)-f_{+}\big(s+r,X_{s+r},-v(i)r\big)\big)\\
&=\frac{1}{r}\,\bE_{s,i}^{*}\big(f_{+}\big(s+r,X_{s+r}^{*},0\big)-f_{+}\big(s+r,X_{s+r}^{*},-v(i)r\big)\big)\\
&=\frac{1}{r}\,\mU_{s,s+r}^{*}\big(f_{+}(s+r,\cdot,0)-f_{+}(s+r,\cdot,-v(i)r)\big)(i)\\
&=\frac{1}{r}\big(\mU_{s,s+r}^{*}\!-\mI\big)\big(f_{+}(s\!+\!r,\cdot,0)\!-\!f_{+}(s\!+\!r,\cdot,-v(i)r)\big)(i)\!+\!\frac{1}{r}\big(f_{+}(s\!+\!r,i,0)\!-\!f_{+}(s\!+\!r,i,-v(i)r)\big).
\end{align*}
By Assumption \ref{assump:GenLambda}, \eqref{eq:EvolDE}, and Lemma \ref{lem:UnifContf}, a similar argument leading to \eqref{eq:LimitJ1r2} shows that
\begin{align*}
\lim_{r\rightarrow 0+}\frac{1}{r}\big|\big(\mU_{s,s+r}^{*}-\mI\big)\big(f_{+}(s+r,\cdot,0)-f_{+}(s+r,\cdot,-v(i)r)\big)(i)\big|=0.
\end{align*}
Hence, noting that $g^{+}\in C_{0}^{1}(\overline{\sX_{+}})=\sD(G^{+})$, by \eqref{eq:FuntfPlus} and Corollary \ref{cor:StrongMarkovExpTildeTau}, we have
\begin{align}
\lim_{r\rightarrow 0+}\cJ_{41}(r)&= -\lim_{r\rightarrow 0+}\frac{1}{r}\big(f_{+}(s+r,i,-v(i)r)-f_{+}(s+r,i,0)\big)\nonumber\\
&= -\lim_{r\rightarrow 0+}\frac{1}{r}\bigg(\wt{\bE}_{s,i,0}\Big(g^{+}\Big(Z^{1}_{\wt{\tau}_{-v(i)r}^{+}},Z^{2}_{\wt{\tau}_{-v(i)r}^{+}}\Big)\Big)-\wt{\bE}_{s,i,0}\Big(g^{+}\Big(Z^{1}_{\wt{\tau}_{0}^{+}},Z^{2}_{\wt{\tau}_{0}^{+}}\Big)\Big)\bigg)\nonumber\\
&= -\lim_{r\rightarrow 0+}\frac{1}{r}\,\wt{\bE}_{s,i,0}\Big(\big(\cP_{-v(i)r}^{+}\,g^{+}\big)\Big(Z^{1}_{\wt{\tau}_{0}^{+}},Z^{2}_{\wt{\tau}_{0}^{+}}\Big)-g^{+}\Big(Z^{1}_{\wt{\tau}_{0}^{+}},Z^{2}_{\wt{\tau}_{0}^{+}}\Big)\Big)\nonumber\\
\label{eq:LimitJ41r} &=v(i)\,\wt{\bE}_{s,i,0}\Big(\big(G^{+}g^{+}\big)\Big(Z^{1}_{\wt{\tau}_{0}^{+}},Z^{2}_{\wt{\tau}_{0}^{+}}\Big)\Big).
\end{align}
Next, since $\wt{\bP}_{s,i,0}(\wt{\gamma}_{1}\leq\wt{\tau}_{0}^{+})=1$ for $i\in\mathbf{E}_{-}$, by Lemma \ref{lem:UnifContf} and \eqref{eq:DistFirstJumpTildeZ2},
\begin{align}
\lim_{r\rightarrow 0+}\big|\cJ_{42}(r)\big|&\leq\lim_{r\rightarrow 0+}\frac{1}{r}\,\wt{\bE}_{s,i,0}\Big(\1_{\{\wt{\gamma}_{1}\leq r\}}\big|f_{+}\big(s+r,Z^{2}_{r},0\big)-f_{+}\big(s+r,Z^{2}_{r},-v(i)r\big)\big|\Big)\nonumber\\
&\leq\lim_{r\rightarrow 0+}\frac{1}{r}\,\wt{\bP}_{s,i,0}\big(\wt{\gamma}_{1}\leq r\big)\,\,\sup_{(j,\ell)\in\mathbf{E}\times[0,-v(i)r]}\big|f_{+}(s+r,j,0)-f_{+}(s+r,j,\ell)\big|\nonumber\\
\label{eq:LimitJ42r} &\leq K\,\lim_{r\rightarrow 0+}\sup_{(j,\ell)\in\mathbf{E}\times[0,-v(i)r]}\big|f_{+}(s+r,j,0)-f_{+}(s+r,j,\ell)\big|=0.
\end{align}
As for $\cJ_{43}(r)$, since $Z^{2}_{u}=Z^{2}_{0}$ for all $u\in[0,r]$ on $\{\wt{\gamma}_{1}>r\}$, it follows from \eqref{eq:DistTildeZ3IntTildeZ2} that $Z^{3}_{r}=\int_{0}^{r}v(Z^{2}_{u})du=v(i)r$, $\wt{\bP}_{s,i,0}-$a.s. on $\{\wt{\gamma}_{1}>r\}$, and thus
\begin{align*}
\1_{\{\wt{\gamma}_{1}>r\}}\big(f_{+}\big(s+r,Z^{2}_{r},-v(i)r\big)-f_{+}\big(s+r,Z^{2}_{r},-Z^{3}_{r}\big)\big)=0,\quad\wt{\bP}_{s,i,a}-\text{a.}\,\text{s.}.
\end{align*}
Hence, by Lemma \ref{lem:UnifContf} and \eqref{eq:DistFirstJumpTildeZ2}, we have
\begin{align}
\lim_{r\rightarrow 0+}\big|\cJ_{43}(r)\big|&=\lim_{r\rightarrow 0+}\frac{1}{r}\Big|\wt{\bE}_{s,i,0}\Big(\1_{\{\wt{\tau}_{0}^{+}>r\geq\wt{\gamma}_{1}\}}\big(f_{+}\big(s+r,Z^{2}_{r},-v(i)r\big)-f_{+}\big(s+r,Z^{2}_{r},-Z^{3}_{r}\big)\big)\Big)\Big|\nonumber\\
&\leq\lim_{r\rightarrow 0+}\frac{1}{r}\,\wt{\bE}_{s,i,0}\Big(\1_{\{\wt{\tau}_{0}^{+}>r\geq\wt{\gamma}_{1}\}}\sup_{j\in\mathbf{E},\,\ell_{1},\ell_{2}\in[0,\overline{v}r]}\big|f_{+}(s+r,j,\ell_{1})-f_{+}(s+r,j,\ell_{2})\big|\Big)\nonumber\\
&\leq\lim_{r\rightarrow 0+}\frac{1}{r}\,\wt{\bP}_{s,i,0}\big(\wt{\gamma}_{1}\leq r\big)\sup_{j\in\mathbf{E},\,\ell_{1},\ell_{2}\in[0,\overline{v}r]}\big|f_{+}(s+r,j,\ell_{1})-f_{+}(s+r,j,\ell_{2})\big|\nonumber\\
\label{eq:LimitJ43r} &\leq K\,\lim_{r\rightarrow 0+}\sup_{j\in\mathbf{E},\,\ell_{1},\ell_{2}\in[0,\overline{v}r]}\big|f_{+}(s+r,j,\ell_{1})-f_{+}(s+r,j,\ell_{2})\big|=0.
\end{align}
Combining \eqref{eq:DecompJ4r}$-$\eqref{eq:LimitJ43r}, we obtain that
\begin{align}\label{eq:LimitJ4r}
\lim_{r\rightarrow 0+}\cJ_{4}(r)=v(i)\,\wt{\bE}_{s,i,0}\Big(\big(G^{+}g^{+}\big)\Big(Z^{1}_{\wt{\tau}_{0}^{+}},Z^{2}_{\wt{\tau}_{0}^{+}}\Big)\Big).
\end{align}

Finally, in view of \eqref{eq:DefJPlusTildeP}, \eqref{eq:DecompPartialJPlus}, \eqref{eq:LimitJ1r}, \eqref{eq:LimitJ3r}, \eqref{eq:LimitJ2r}, and \eqref{eq:LimitJ4r}, for any $g^{+}\in C_{0}^{1}(\overline{\sX_{+}})$ and $(s,i)\in\sX_{-}$, we get
\begin{align}
\frac{\partial_{+}}{\partial s}\big(J^{+}g^{+}\big)(s,i)&=\lim_{r\rightarrow 0+}\frac{1}{r}\bigg(\wt{\bE}_{s+r,i,0}\Big(g^{+}\Big(Z^{1}_{\wt{\tau}_{0}^{+}},Z^{2}_{\wt{\tau}_{0}^{+}}\Big)\Big)-\wt{\bE}_{s,i,0}\Big(g^{+}\Big(Z^{1}_{\wt{\tau}_{0}^{+}},Z^{2}_{\wt{\tau}_{0}^{+}}\Big)\Big)\bigg)\nonumber\\
\label{eq:WHPlusDPoint} &=v(i)\,\wt{\bE}_{s,i,0}\Big(\big(G^{+}g^{+}\big)\Big(Z^{1}_{\wt{\tau}_{0}^{+}},Z^{2}_{\wt{\tau}_{0}^{+}}\Big)\Big)-\sum_{j\in\mathbf{E}}\mathsf{\Lambda}_{s}(i,j)\,\wt{\bE}_{s,j,0}\Big(g^{+}\Big(Z^{1}_{\wt{\tau}_{0}^{+}},Z^{2}_{\wt{\tau}_{0}^{+}}\Big)\Big).
\end{align}
Moreover, by \eqref{eq:Trivfplus}, \eqref{eq:JPlusfPlus}, and the definitions of $\wt{\mC}$ and $\wt{\mD}$ (cf. the end of Section \ref{subsec:Notations})
\begin{align}\label{eq:WHPlusDPoint2}
&v(i)\,\wt{\bE}_{s,i,0}\Big(\big(G^{+}g^{+}\big)\Big(Z^{1}_{\wt{\tau}_{0}^{+}},Z^{2}_{\wt{\tau}_{0}^{+}}\Big)\Big)-\sum_{j\in\mathbf{E}}\mathsf{\Lambda}_{s}(i,j)\,\wt{\bE}_{s,j,0}\Big(g^{+}\Big(Z^{1}_{\wt{\tau}_{0}^{+}},Z^{2}_{\wt{\tau}_{0}^{+}}\Big)\Big)\nonumber\\
&\quad =v(i)\big(J^{+}G^{+}g^{+}\big)(s,i)-\sum_{j\in\mathbf{E}_{+}}\mathsf{\Lambda}_{s}(i,j)g^{+}(s,j)-\sum_{j\in\mathbf{E}_{-}}\mathsf{\Lambda}_{s}(i,j)\big(J^{+}g^{+}\big)(s,j)\nonumber\\
&\quad =v(i)\big(J^{+}G^{+}g^{+}\big)(s,i)-\big(\wt{\mC}\,g^{+}\big)(s,i)-\big(\wt{\mD}J^{+}g^{+}\big)(s,i),
\end{align}
Putting together \eqref{eq:WHPlusDPoint} and \eqref{eq:WHPlusDPoint2}, we deduce \eqref{eq:WHplusDRightDev}, which completes the proof.

\subsection{Uniqueness of the Wiener-Hopf factorization}\label{sec:UniqProof}
In this section we prove the ``+" part of Theorem \ref{thm:WHProbInterpr}. Specifically, we will show that, if $(S^{+},H^{+})$ solves \eqref{eq:WHPlus} subject to $(a^{+})$ and $(b^{+})$, then, for any $g^{+}\in C_{0}(\overline{\sX_{+}})$, $S^{+}g^{+}=J^{+}g^{+}$ and $\cQ_{\ell}^{+}g^{+}=\cP_{\ell}^{+}g^{+}$, $\ell\in\bR_{+}$. This also guarantees the uniqueness of $G^{+}$, since two strongly continuous contraction semigroup coincide if and only if their generators coincide (cf. \cite[Theorem 1.2]{Dynkin1965}). Throughout this subsection, we assume that $(S^{+},H^{+})$ satisfies \eqref{eq:WHPlus} (or equivalently, \eqref{eq:WHPlus1} and \eqref{eq:WHPlus2}) and the conditions $(a^{+})$ and $(b^{+})$.

To begin with, we will show a sufficient condition of what we would like to prove. For any $g^{+}\in C_{0}(\overline{\sX_{+}})$, $(s,i,a)\in\sZ$, and $\ell\in[a,\infty)$, we define
\begin{align}\label{eq:FhatPlus}
\wh{F}_{+}(s,i,a,\ell;g^{+})&:=\left(\begin{pmatrix} I^{+} \\ S^{+} \end{pmatrix}\cQ^{+}_{\ell-a}g^{+}\right)(s,i),\\
\label{eq:FPlus} F_{+}(s,i,a,\ell;g^{+})&:=\wt{\bE}_{s,i,0}\Big(g^{+}\Big(Z^{1}_{\wt{\tau}_{\ell-a}^{+}},Z^{2}_{\wt{\tau}_{\ell-a}^{+}}\Big)\Big).
\end{align}
When no confusion arises, we will omit $g^{+}$ in $\wh{F}_{+}(s,i,a,\ell;g^{+})$ and $F_{+}(s,i,a,\ell;g^{+})$.
\begin{proposition}\label{prop:SuffCondUniq}
Suppose that
\begin{align}\label{eq:FhatFPlus}
\wh{F}_{+}(s,i,0,\ell;g^{+})=F_{+}(s,i,0,\ell;g^{+}),\quad (s,i)\in\sX,\quad\ell\in\bR_{+},\quad g^{+}\in C_{c}^{1}(\overline{\sX_{+}}).
\end{align}
Then, for any $g^{+}\in C_{0}(\overline{\sX_{+}})$,
\begin{align*}
S^{+}g^{+}=J^{+}g^{+},\quad\cQ^{+}_{\ell}g^{+}=\cP^{+}_{\ell}g^{+},\quad\ell\in\bR_{+}.
\end{align*}
\end{proposition}

\begin{proof}
Let $g^{+}\in C^{1}_{c}(\overline{\sX_{+}})$. By Corollary \ref{cor:StrongMarkovExpTildeTau}, for any $(s,i)\in\sX$ and $\ell\in\bR_{+}$,
\begin{align*}
F_{+}(s,i,0,\ell;g^{+})&=\wt{\bE}_{s,i,0}\Big(\big(\cP^{+}_{\ell}g^{+}\big)\Big(Z^{1}_{\wt{\tau}^{+}_{0}},Z^{2}_{\wt{\tau}^{+}_{0}}\Big)\Big)=\left(\begin{pmatrix} I^{+} \\ J^{+} \end{pmatrix}\cP^{+}_{\ell}g^{+}\right)(s,i).
\end{align*}
This, together with \eqref{eq:FhatFPlus}, implies that, for any $\ell\in\bR_{+}$,
\begin{align}\label{eq:QPellPlusCc}
\cQ^{+}_{\ell}g^{+}=\cP^{+}_{\ell}g^{+}{,\quad S^{+}\cQ^{+}_{\ell}g^{+}=J^{+}\cP^{+}_{\ell}g^{+},}
\end{align}
and thus $S^{+}\cP^{+}_{\ell}g^{+}=J^{+}\cP^{+}_{\ell}g^{+}$.
Since $(\cP_{\ell}^{+})_{\ell\in\bR_{+}}$ is a strongly continuous semigroup, and $S^{+}$ and $J^{+}$ are bounded operators, we have
\begin{align}\label{eq:SJPlusCc}
{S^{+}g^{+}=\lim_{\ell\rightarrow 0+}S^{+}\cP^{+}_{\ell}g^{+}=\lim_{\ell\rightarrow 0+}J^{+}\cP^{+}_{\ell}g^{+}=J^{+}g^{+}}
\end{align}
so that $S^+g^+ = J^+g^+$.
Alternatively, this equality can be obtained by letting $\ell=0$ in \eqref{eq:FhatFPlus}.

Finally, since $C^{1}_{c}(\overline{\sX_{+}})$ is dense in $C_{0}(\overline{\sX_{+}})$, and $S^{+}$, $J^{+}$, and $\cP_{\ell}^{+}$ are bounded operators, both \eqref{eq:QPellPlusCc} and \eqref{eq:SJPlusCc} hold true for any $g^{+}\in C_{0}(\overline{\sX_{+}})$, which completes the proof of the proposition.
\end{proof}

Proposition \ref{prop:SuffCondUniq} states that if \eqref{eq:FhatFPlus} is satisfied then the  ``+" part of Theorem \ref{thm:WHProbInterpr} holds true.  Thus, to conclude the proof of the  ``+" part of Theorem \ref{thm:WHProbInterpr}, and therefore the proof of uniqueness of our Wiener-Hopf factorization, it remains to prove that \eqref{eq:FhatFPlus} holds. The rest of this section is devoted to this task.

We need the following {three} technical lemmas, whose proofs are deferred to Appendix~\ref{sec:ProofLemUniq}.

\begin{lemma}\label{lem:DiffQell}
For any $g^{+}\in C_{0}^{1}(\overline{\sX_{+}})$ and $(s,i)\in\sX_{-}$, $(S^{+}\cQ^{+}_{\cdot}g^{+})(s,i)$ is differentiable on $\bR_{+}$, and
\begin{align*}
\frac{\partial}{\partial\ell}\big(\big(S^{+}\cQ^{+}_{\ell}g^{+}\big)(s,i)\big)=\big(S^{+}H^{+}\cQ^{+}_{\ell}g^{+}\big)(s,i).
\end{align*}
\end{lemma}

Let $C_{0}(\overline{\sZ})$ be the space of real-valued $\cB(\overline{\sZ})$-measurable functions $h$ on $\sZ$ such that $h(\infty,\partial,\infty)=0$, and that $h(\cdot,i,\cdot)\in C_{0}(\bR_{+}\times\bR)$ for all $i\in\mathbf{E}$. Let $C_{0}^{1}(\overline{\sZ})$ be the space of functions $h\in C_{0}(\overline{\sZ})$ such that, for all $i\in\mathbf{E}$, $\partial h(\cdot,i,\cdot)/\partial s$ and $\partial h(\cdot,i,\cdot)/\partial a$ exist and belong to $C_{0}(\bR_{+}\times\bR)$.

\begin{lemma}\label{lem:GenTildeM}
Let $\cA$ be the strong generator of the Feller semigroup associated with the Markov family $\wt{\cM}$. Then, $C_{0}^{1}(\overline{\sZ})\subset\sD(\cA)$, and for any $h\in C_{0}^{1}(\overline{\sZ})$,
\begin{align*}
(\cA\,h)(s,i,a)=\frac{\partial h}{\partial s}(s,i,a) + \sum_{j\in\mathbf{E}}\mathsf{\Lambda}_{s}(i,j)h(s,j,a)+v(i)\frac{\partial h}{\partial a}(s,i,a),\quad (s,i,a)\in\sZ.
\end{align*}
\end{lemma}

\begin{lemma}\label{lem:QellCompgPlus}
For any $g^{+}\in C_{c}(\overline{\sX_{+}})$ with $\supp g^{+}\subset[0,\eta_{g^{+}}]\times\mathbf{E}_{+}$ for some $\eta_{g^{+}}\in(0,\infty)$, we have $\supp\cQ^{+}_{\ell}g^{+}\subset[0,\eta_{g^{+}}]\times\mathbf{E}_{+}$, for any $\ell\in\bR_{+}$.
\end{lemma}

For any $a\in\bR$, let $C_{0}(\sX\times (-\infty,a])$ be the space of real-valued $\cB(\sX)\otimes\cB((-\infty,a])$-measurable functions $h$ on $\sX\times (-\infty,a])$ such that $h(\cdot,i,\cdot)\in C_{0}(\bR_{+}\times (-\infty,a])$ for all $i\in\mathbf{E}$. Let $C_{0}^{1}(\sX\times (-\infty,a])$ be the space of functions $h\in C_{0}(\sX\times (-\infty,a])$ such that, for all $i\in\mathbf{E}$, $\partial h(\cdot,i,\cdot)/\partial s$ and $\partial h(\cdot,i,\cdot)/\partial a$ exist and belong to $C_{0}(\bR_{+}\times (-\infty,a])$.

\begin{lemma}\label{lem:FhatPlusC01}
For any $\ell\in\bR$ and $g^{+}\in C_{c}^{1}(\overline{\sX_{+}})$, $\hat{F}_{+}(\cdot,\cdot,\cdot,\ell;g^{+})\in C_{0}^{1}(\sX\times (-\infty,\ell])$.
\end{lemma}

We are now in the position of proving \eqref{eq:FhatFPlus}. In what follows, we fix $g^{+}\in C_{c}^{1}(\overline{\sX_{+}})$ with $\supp g^{+}\subset[0,\eta_{g^{+}}]\times\mathbf{E}_{+}$ for some $\eta_{g^{+}}\in(0,\infty)$.

We first show that, for any $(s,i)\in\sX$, $\ell\in\bR_{+}$, and $T\in(0,\infty)$,
\begin{align}\label{eq:FhatPlusMartingale}
\wt{\bE}_{s,i,0}\Big(\wh{F}_{+}\Big(Z^{1}_{\wt{\tau}^{+}_{\ell}\wedge T},Z^{2}_{\wt{\tau}^{+}_{\ell}\wedge T},Z^{3}_{\wt{\tau}^{+}_{\ell}\wedge T},\ell\Big)\Big)=\wh{F}_{+}(s,i,0,\ell).
\end{align}
Let $\phi\in C^{1}(\bR)$ with $\phi(a)=1$ for $a\in(-\infty,\ell]$ and $\lim_{a\rightarrow\infty}\phi(a)=0$. We extend $\wh{F}_+(\cdot,\cdot,\cdot,\ell)$ to be a function on $\sZ$ by defining
\begin{align*}
\wh{F}_{+}(s,i,a,\ell):=\phi(a)\big(2\wh{F}_{+}(s,i,\ell,\ell)-\wh{F}_{+}(s,i,2\ell-a,\ell)\big),\quad (s,i)\in\sX,a\in(\ell,\infty).
\end{align*}
By Lemma \ref{lem:FhatPlusC01}, we now have $\wh{F}_{+}(\cdot,\cdot,\cdot,\ell)\in C_{0}^{1}(\overline{\sZ})$ (with the convention that $\wh{F}_{+}(\infty,\partial,\infty,\ell)=0$). It follows from Lemma \ref{lem:GenTildeM}, \eqref{eq:FhatPlus}, \eqref{eq:DevQella}, and Lemma \ref{lem:DiffQell} that, for any $(s,i)\in\sX$ and $a\in(-\infty,\ell)$,
\begin{align*}
\big(\cA\wh{F}_{+}\big)(s,i,a,\ell)&=\frac{\partial\wh{F}_{+}}{\partial s}(s,i,a,\ell)+\sum_{j\in\mathbf{E}}\mathsf{\Lambda}_{s}(i,j)\wh{F}_{+}(s,j,a,\ell)+v(i)\frac{\partial\wh{F}_{+}}{\partial a}(s,i,a,\ell)\\
&=\Bigg(\!\frac{\partial}{\partial s}\!\begin{pmatrix} I^{+} \\ S^{+} \end{pmatrix}\!\cQ^{+}_{\ell-a}g^{+}\!\!\Bigg)(s,i)+\!\Bigg(\!\wt{\mathsf{\Lambda}} \begin{pmatrix} I^{+} \\ S^{+} \end{pmatrix}\!\cQ^{+}_{\ell-a}g^{+}\!\!\Bigg)(s,i)+\!\Bigg(\!\mV\frac{\partial}{\partial a}\!\begin{pmatrix} I^{+} \\ S^{+} \end{pmatrix}\!\cQ^{+}_{\ell-a}g^{+}\!\!\Bigg)(s,i)\\
&=\left(\Bigg(\bigg(\frac{\partial}{\partial s}+\wt{\mathsf{\Lambda}}\bigg) \begin{pmatrix} I^{+} \\ S^{+} \end{pmatrix} -\mV \begin{pmatrix} I^{+} \\ S^{+} \end{pmatrix} H^{+}\Bigg)\cQ^{+}_{\ell-a}g^{+}\right)(s,i),
\end{align*}
where we note that $\cQ^{+}_{\ell-a}g^{+}\in\sD(H^{+})$ since $g^{+}\in C_{0}^{1}(\overline{\sX_{+}})=\sD(H^{+})$. Hence, since $(S^{+},H^{+})$ solves \eqref{eq:WHPlus}, we have
\begin{align*}
\big(\cA\wh{F}_{+}\big)(s,i,a,\ell)=0,\quad (s,i)\in\sX,\quad a\in(-\infty,\ell).
\end{align*}
Therefore, by Dynkin's formula (cf. \cite[III.10]{RogersWilliams1994}), we obtain that
\begin{align*}
\wt{\bE}_{s,i,0}\Big(\wh{F}_{+}\Big(Z^{1}_{\wt{\tau}^{+}_{\ell}\wedge T},Z^{2}_{\wt{\tau}^{+}_{\ell}\wedge T},Z^{3}_{\wt{\tau}^{+}_{\ell}\wedge T},\ell\Big)\Big)-\wh{F}_{+}(s,i,0,\ell)=\wt{\bE}_{s,i,0}\bigg(\int_{0}^{\wt{\tau}^{+}_{\ell}\wedge T}\!\!\big(\cA\wh{F}_{+}\big)\big(Z^{1}_{t},Z^{2}_{t},Z^{3}_{t},\ell\big)dt\bigg)=0,
\end{align*}
which completes the proof of \eqref{eq:FhatPlusMartingale}.

By \eqref{eq:FhatPlusMartingale}, we have
\begin{align*}
\wh{F}_{+}(s,i,0,\ell)=\wt{\bE}_{s,i,0}\Big(\wh{F}_{+}\Big(Z^{1}_{\wt{\tau}^{+}_{\ell}},Z^{2}_{\wt{\tau}^{+}_{\ell}},Z^{3}_{\wt{\tau}^{+}_{\ell}},\ell\Big)\1_{\{\wt{\tau}^{+}_{\ell}<T\}}\Big)+\wt{\bE}_{s,i,0}\Big(\wh{F}_{+}\big(Z^{1}_{T},Z^{2}_{T},Z^{3}_{T},\ell\big)\1_{\{\wt{\tau}^{+}_{\ell}\geq T\}}\Big).
\end{align*}
From the definition of $\wt{\tau}^{+}_{\ell}$ and the right-continuity of the sample paths of $Z$, we have $Z^{3}_{\wt{\tau}^{+}_{\ell}}=\ell$ on $\{\wt{\tau}^{+}_{\ell}<T\}$. Moreover, it is clear from the construction of $\wt{\cM}$ that $Z_{t}^{2}\in\mathbf{E}$ for $t\in\bR_{+}$, and, in view of \eqref{eq:RangeZ2TildeTauPlus}, we deduce that $Z^{2}_{\wt{\tau}^{+}_{\ell}}\in\mathbf{E}_{+}$ on $\{\wt{\tau}^{+}_{\ell}<T\}$. Together with \eqref{eq:FhatPlus}, \eqref{eq:FPlus}, and \eqref{eq:ProcZ}, we obtain that
\begin{align}
\wh{F}_{+}(s,i,0,\ell)&=\wt{\bE}_{s,i,0}\Big(\wh{F}_{+}\Big(Z^{1}_{\wt{\tau}^{+}_{\ell}},Z^{2}_{\wt{\tau}^{+}_{\ell}},\ell,\ell\Big)\1_{\{\wt{\tau}^{+}_{\ell}<T\}}\Big)+\wt{\bE}_{s,i,0}\Big(\wh{F}_{+}\big(Z^{1}_{T},Z^{2}_{T},Z^{3}_{T},\ell\big)\1_{\{\wt{\tau}^{+}_{\ell}\geq T\}}\Big)\nonumber\\
&=\wt{\bE}_{s,i,0}\Big(g^{+}\Big(Z^{1}_{\wt{\tau}^{+}_{\ell}},Z^{2}_{\wt{\tau}^{+}_{\ell}}\Big)\1_{\{\wt{\tau}^{+}_{\ell}<T\}}\Big)+\wt{\bE}_{s,i,0}\Big(\wh{F}_{+}\big(Z^{1}_{T},Z^{2}_{T},Z^{3}_{T},\ell\big)\1_{\{\wt{\tau}^{+}_{\ell}\geq T\}}\Big)\nonumber\\
&=F_{+}(s,i,0,\ell)-\wt{\bE}_{s,i,0}\Big(g^{+}\Big(Z^{1}_{\wt{\tau}^{+}_{\ell}},Z^{2}_{\wt{\tau}^{+}_{\ell}}\Big)\1_{\{\wt{\tau}^{+}_{\ell}\geq T\}}\Big)+\wt{\bE}_{s,i,0}\Big(\wh{F}_{+}\big(Z^{1}_{T},Z^{2}_{T},Z^{3}_{T},\ell\big)\1_{\{\wt{\tau}^{+}_{\ell}\geq T\}}\Big)\nonumber\\
&=F_{+}(s,i,0,\ell)-\wt{\bE}_{s,i,0}\Big(g^{+}\Big(s+\wt{\tau}^{+}_{\ell},Z^{2}_{\wt{\tau}^{+}_{\ell}}\Big)\1_{\{\wt{\tau}^{+}_{\ell}\geq T\}}\Big)\nonumber\\
\label{eq:FhatFPlusDiff} &\quad +\wt{\bE}_{s,i,0}\Big(\wh{F}_{+}\big(s+T,Z^{2}_{T},Z^{3}_{T},\ell\big)\1_{\{\wt{\tau}^{+}_{\ell}\geq T\}}\Big).
\end{align}
Therefore, in order to prove \eqref{eq:FhatFPlus}, it remains to show that the last two terms in \eqref{eq:FhatFPlusDiff} vanish. Since $g^{+}\in C_{c}(\overline{\sX_{+}})$ (and so $g^{+}(Z^{1}_{\infty},Z^{2}_{\infty})=g^{+}(\infty,\partial)=0$) with $\supp g^{+}\subset[0,\eta_{g^{+}}]\times\mathbf{E}_{+}$, and using the fact that $Z^{2}_{\wt{\tau}^{+}_{\ell}}\in\mathbf{E}_{+}$ on $\{\wt{\tau}^{+}_{\ell}<\infty\}$, we have, for $T\in[\eta_{g^{+}}-s,\infty)$,
\begin{align*}
g^{+}\Big(s+\wt{\tau}^{+}_{\ell},Z^{2}_{\wt{\tau}^{+}_{\ell}}\Big)\1_{\{\wt{\tau}^{+}_{\ell}\geq T\}}=g^{+}\Big(s+\wt{\tau}^{+}_{\ell},Z^{2}_{\wt{\tau}^{+}_{\ell}}\Big)\1_{\{\wt{\tau}^{+}_{\ell}\in[T,\infty)\}}=0.
\end{align*}
Hence, the second term in \eqref{eq:FhatFPlusDiff} vanishes when $T\in[\eta_{g^{+}}-s,\infty)$. As for the last term in \eqref{eq:FhatFPlusDiff}, since $\supp g^{+}\subset[0,\eta_{g^{+}}]\times\mathbf{E}_{+}$, Lemma \ref{lem:QellCompgPlus} and the condition $(a^{+})(\text{i})$ ensure that $\supp\wh{F}_{+}(\cdot,\cdot,\cdot,\ell)\subset[0,\eta_{g}^{+}]\times\mathbf{E}\times[0,\ell]$. Hence, when $T\in[\eta_{g^{+}}-s,\infty)$, $\wh{F}_{+}(s+T,Z^{2}_{T},Z^{3}_{T},\ell)=0$ so that the last term in \eqref{eq:FhatFPlusDiff} vanishes. Therefore, by choosing $T\in[\eta_{g^{+}}-s,\infty)$, we obtain \eqref{eq:FhatFPlus} from \eqref{eq:FhatFPlusDiff}.

The proof of the ``+" part of Theorem \ref{thm:WHProbInterpr} is complete. As mentioned earlier, the proof of the ``-" part of Theorem \ref{thm:WHProbInterpr} proceeds in direct analogy to ``+" part given above.

\appendix

\section{Construction of $\cM$}\label{sec:AppendixA}

In this appendix we provide construction of the standard time-inhomogeneous Markov family $$\cM=\big\{\big(\Omega,\sF,\bF_{s},(X_{t},\varphi_{t})_{t\in[s,\infty]},\bP_{s,(i,a)}\big),\,(s,i,a)\in\overline{\sZ}\big\}$$ introduced in Section \ref{subsec:TimeHomogen}.

\subsection{Construction via the transition function}\label{subsec:TranFunt}

Recall the function $P:\overline{\sZ}\times\overline{\bR}_{+}\times\cB(\overline{\sY})$ given as in \eqref{eq:DefTranProbXvarphi}, where $\sY=\mathbf{E}\times\bR$, $\overline{\sY}=\sY\cup\{(\partial,\infty)\}$, $\sZ=\bR_{+}\times\sY$, and $\overline{\sZ}=\sZ\cup\{(\infty,\partial,\infty)\}$.

\begin{lemma}\label{lem:TranProbXvarphi}
$P$ is a transition function.
\end{lemma}

\begin{proof}
The proof is divided into the following two steps.

\medskip
\noindent
\textbf{Step 1.} We first show that for any $0\leq s\leq t\leq\infty$ and $A\in\cB(\overline{\sY})$,
\begin{align*}
P(s,\,\cdot\,,t,A):\,\,\big(\overline{\sY},\cB(\overline{\sY})\big)\rightarrow\big([0,1],\cB([0,1])\big)
\end{align*}
is measurable. Note that when $t=\infty$, for any $s\in[0,\infty]$, $(i,a)\in\overline{\sY}$, and $A\in\cB(\overline{\sY})$, \linebreak $P(s,(i,a),\infty,A)=\delta_{(\partial,\infty)}(A)$, where $\delta_{(\partial,\infty)}$ denotes the Dirac measure at $(\partial,\infty)$. Hence, $P(s,\,\cdot\,,\infty,A)$ is $\cB(\overline{\sY})$-measurable. When $t\in\bR_{+}$, $(X_{t}^{*},\phi_{t}^{*})$ takes values only in $\sY$, and thus $P(s,(i,a),t,\,\cdot\,)$ is supported {on} $\sY$. {Thus,} it is sufficient to discuss the measurability of $P(s,\,\cdot\,,t,A)$ when $0\leq s\leq t<\infty$ and $A\in\cB(\sY)=2^{\mathbf{E}}\otimes\cB(\bR)$.

Let $A=\{j\}\times (-\infty,b]$, for some $j\in\mathbf{E}$ and $b\in\bR$. By \eqref{eq:DefTranProbXvarphi},
\begin{align*}
P\big(s,(i,a),t,\{j\}\times (-\infty,b]\big)=\bP_{s,i}^{*}\bigg(X_{t}^{*}=j,\,\,\int_{s}^{t}v\big(X_{u}^{*}\big)\,du\in[0,b-a]\bigg).
\end{align*}
Hence, for any $i\in\mathbf{E}$, $P(s,(i,\cdot),t,\{j\}\times (-\infty,b])$ is left-continuous on $\bR$, and is thus $\cB(\bR)$-measurable. Therefore, $P(s,\,\cdot\,,t,\{j\}\times (-\infty,b])$ is $\cB(\sY)$-measurable since $\mathbf{E}$ is finite.

Next, let
\begin{align*}
\cH_{1}:=\big\{A\in\cB(\sY):\,P(s,\,\cdot\,,t,A)\text{ is }\,\cB(\sY)\text{-measurable, for any }0\leq s\leq t<\infty\big\}.
\end{align*}
The above arguments have shown that
\begin{align}\label{eq:DefH0}
\cH_{1}\supset\cH_{0}:=\{B\times (-\infty,b]:B\subset\mathbf{E},\,b\in\bR\},
\end{align}
since $\mathbf{E}$ is finite. Clearly, $\cH_{0}$ is a $\pi$-system on $\sY$. We will now show that $\cH_{1}$ is a $\lambda$-system on $\sY$. First, $P(s,(i,a),t,\sY)\equiv 1$ for all $(i,a)\in\sY$ and $0\leq s\leq t<\infty$, so that $P(s,\,\cdot\,,t,\sY)$ is $\cB(\sY)$-measurable for any $0\leq s\leq t<\infty$, which implies that $\sY\in\cH_{1}$. Moreover, if $A\in\cH_{1}$, then $P(s,\,\cdot\,,t,A)$ is $\cB(\sY)$-measurable for any $0\leq s\leq t<\infty$. Hence, $P(s,\,\cdot\,,t,A^{c})=1-P(s,\,\cdot\,,t,A)$ is $\cB(\sY)$-measurable for any $0\leq s\leq t<\infty$, which implies that $A^{c}\in\cH_{1}$. Finally, if $(A_{n})_{n\in\bN}$ is a sequence of disjoint subsets in $\cH_{1}$, then $P(s,\,\cdot\,,t,A_{n})$ is $\cB(\sY)$-measurable for any $0\leq s\leq t<\infty$ and $n\in\bN$. Since,
\begin{align*}
P\bigg(s,(i,a),t,\bigcup_{n=1}^{\infty}A_{n}\bigg)=\sum_{n=1}^{\infty}P\big(s,(i,a),t,A_{n}\big),\quad (i,a)\in\sY,
\end{align*}
$P(s,\,\cdot\,,t,\cup_{n\in\bN}A_{n})$ is also $\cB(\sY)$-measurable for any $0\leq s\leq t<\infty$, and thus $\cup_{n\in\bN}A_{n}\in\cH_{1}$. Hence, $\cH_{1}$ is a $\lambda$-system on $\sY$, and by the monotone class theorem, $\cH_{1}\supset\sigma(\cH_{0})=\cB(\sY)$. Therefore, $\cH_{1}=\cB(\sY)$, which completes the proof of Step 1.

\medskip
\noindent
\textbf{Step 2.} It is clear from \eqref{eq:DefTranProbXvarphi} that, for any $0\leq s\leq t\leq\infty$ and $(i,a)\in\overline{\sY}$, $P(s,(i,a),t,\,\cdot\,)$ is a probability measure on $(\overline{\sY},\cB(\overline{\sY}))$. In particular, $P(s,(i,a),s,\,\cdot\,)$ is the Dirac measure at $(i,a)$.

To show that $P$ is a transition function, it remains to show that $P$ satisfies the Chapman-Kolmogorov equation, namely, for any $0\leq s\leq r\leq t\leq\infty$, $(i,a)\in\overline{\sY}$, and $A\in\cB(\overline{\sY})$,
\begin{align}\label{eq:ChapmanKolmogorov}
P\big(s,(i,a),t,A\big)=\int_{\overline{\sY}}P\big(r,(j,b),t,A\big)\,dP\big(s,(i,a),r,(j,b)\big).
\end{align}
Note that when $t=\infty$, \eqref{eq:ChapmanKolmogorov} is satisfied since $P(s,(i,a),\infty,A)=P(r,(j,b),\infty,A)=\delta_{(\partial,\infty)}(A)$. Since $P(s,(i,a),t,\,\cdot\,)$ is supported on $\sY$, it is sufficient to show that
\begin{align}\label{eq:ChapmanKolmogorovFinite}
P\big(s,(i,a),t,A\big)=\int_{\sY}P\big(r,(j,b),t,A\big)\,dP\big(s,(i,a),r,(j,b)\big),
\end{align}
for any $0\leq s\leq r\leq t<\infty$, $(i,a)\in\sY$, and $A\in\cB(\sY)$.

Again, we start {with} the case when $A=\{k\}\times (-\infty,c]$, for some $k\in\mathbf{E}$ and $c\in\bR$. For each $n\in\bN$, and any $j\in\mathbf{E}$, let
\begin{align*}
P_{n}\big(r,(j,b),t,A\big)\!:=\!\left\{\begin{array}{ll} P\big(r,(j,-n),t,A\big), &b\in(-\infty,-n], \\ \\ \displaystyle{P\bigg(\!r,\!\bigg(\!j,-n\!+\!\frac{n(p\!-\!1)}{2^{n}}\!\bigg),t,A\!\bigg)}, &\!\!\displaystyle{b\!\in\!\bigg(\!\!-n\!+\!\frac{n(p\!-\!1)}{2^{n}},-n\!+\!\frac{np}{2^{n}}\bigg]},\,\,\,1\!\leq\!p\!\leq\!n2^{n+1}, \\ \\ P\big(r,(j,n),t,A\big), &b\in(n,\infty).\end{array}\right.
\end{align*}
The dominated convergence theorem implies that
\begin{align*}
\int_{\sY}P\big(r,(j,b),t,A\big)\,dP\big(s,(i,a),r,(j,b)\big)&=\int_{\sY}\Big(\lim_{n\rightarrow\infty}P_{n}\big(r,(j,b),t,A\big)\Big)\,dP\big(s,(i,a),r,(j,b)\big)\\
&=\lim_{n\rightarrow\infty}\int_{\sY}P_{n}\big(r,(j,b),t,A\big)\,dP\big(s,(i,a),r,(j,b)\big).
\end{align*}
Since $\mathbf{E}$ is a finite set, for any $0\leq s\leq t<\infty$ and $(i,a)\in\sY$, when $n\in\bN$ is large enough,
\begin{align*}
\bP^{*}_{s,i}\bigg(a+\int_{s}^{t}v\big(X^{*}_{u}\big)\,du\in (-n,n)\bigg)=1.
\end{align*}
Hence, for large $n\in\bN$, by the Markov property of $X^{*}$,
\begin{align*}
&\int_{\sY}P_{n}\big(r,(j,b),t,A\big)\,dP\big(s,(i,a),r,(j,b)\big)\\
&\quad =\sum_{j\in\mathbf{E}}\sum_{p=1}^{2^{n+1}}P\bigg(r,\bigg(j,-n+\frac{n(p-1)}{2^{n}}\bigg),t,A\bigg)P\,\bigg(s,(i,a),r,\{j\}\times\bigg(-n+\frac{n(p-1)}{2^{n}},-n+\frac{np}{2^{n}}\bigg]\bigg)\\
&\quad =\sum_{j\in\mathbf{E}}\sum_{p=1}^{2^{n+1}}\bP_{r,j}^{*}\bigg(-n+\frac{n(p-1)}{2^{n}}+\int_{r}^{t}v\big(X^{*}_{u}\big)\,du\in(-\infty,c],\,X^{*}_{t}=k\bigg)\\
&\qquad\qquad\quad\,\,\,\cdot\bP^{*}_{s,i}\bigg(a+\int_{s}^{r}v\big(X^{*}_{u}\big)\,du\in\bigg(-n+\frac{n(p-1)}{2^{n}},-n+\frac{np}{2^{n}}\bigg],\,X_{r}^{*}=j\bigg)\\
&\quad=\sum_{p=1}^{2^{n+1}}\bP^{*}_{s,i}\bigg(-n+\frac{n(p-1)}{2^{n}}+\int_{r}^{t}v\big(X^{*}_{u}\big)\,du\in(-\infty,c],\,X^{*}_{t}=k,\\
&\qquad\qquad\qquad\quad a+\int_{s}^{r}v\big(X^{*}_{u}\big)\,du\in\bigg(-n+\frac{n(p-1)}{2^{n}},-n+\frac{np}{2^{n}}\bigg]\bigg).
\end{align*}
Note that, for each $n\in\bN$,
\begin{align*}
&\bigg\{a+\int_{s}^{t}v\big(X^{*}_{u}\big)\,du\in(-\infty,c]\bigg\}\bigcap\bigg\{a+\int_{s}^{t}v\big(X^{*}_{u}\big)\,du\in (-n,n)\bigg\}\\
&\subset\bigcup_{p=1}^{2^{n+1}}\!\bigg(\!\bigg\{a\!+\!\!\int_{s}^{r}\!v\big(X^{*}_{u}\big)du\!\in\!\bigg(\!\!\!-\!n\!+\!\frac{n(p\!-\!1)}{2^{n}},-n\!+\!\frac{np}{2^{n}}\bigg]\bigg\}\bigcap\bigg\{\!\!-\!n\!+\!\frac{n(p\!-\!1)}{2^{n}}\!+\!\!\int_{r}^{t}\!v\big(X^{*}_{u}\big)du\!\in\!(-\infty,c]\!\bigg\}\!\bigg)\\
&\subset\bigg\{a+\int_{s}^{t}v\big(X^{*}_{u}\big)\,du\in\bigg(\!\!-\infty,c+\frac{1}{n}\bigg]\bigg\}.
\end{align*}
Hence, for $A=\{k\}\times (-\infty,c]$,
\begin{align*}
\int_{\sY}P\big(r,(j,b),t,A\big)\,dP\big(s,(i,a),r,(j,b)\big)&=\lim_{n\rightarrow\infty}\int_{\sY}P_{n}\big(r,(j,b),t,A\big)\,dP\big(s,(i,a),r,(j,b)\big)\\
&=\bP^{*}_{s,i}\bigg(a+\int_{s}^{t}v\big(X^{*}_{u}\big)\,du\in(-\infty,c],\,X^{*}_{t}=k\bigg)\\
&=P\big(s,(i,a),t,A\big).
\end{align*}

To prove the Chapman-Kolmogorov equation \eqref{eq:ChapmanKolmogorovFinite} for general $A\in\cB(\sY)$, we use again the monotone class theorem. Let
\begin{align*}
\cH_{2}:=\big\{A\in\cB(\sY):\,\text{the equation (\ref{eq:ChapmanKolmogorovFinite}) holds for any }\,0\leq s\leq r\leq t<\infty\,\,\,\text{and}\,\,\,(i,a)\in\sY\big\}.
\end{align*}
The above arguments have shown that $\cH_{2}$ contains the $\pi$-system $\cH_{0}$, defined as in \eqref{eq:DefH0}. Moreover, using arguments similar to those from the end of Step 1, we can show that $\sY\in\cH_{2}$, and that $\cH_{2}$ is closed under complements and countable disjoint unions. Hence, $\cH_{2}$ is a $\lambda$-system, and by the monotone class theorem, $\cH_{2}\supset\sigma(\cH_{0})=\cB(\sY)$. Therefore, $\cH_{2}=\cB(\sY)$. This completes the proof of Step 2, and thus concludes the proof of the lemma.
\end{proof}

Let $L^{\infty}(\sY)$ be the collection of all bounded, $\cB(\sY)$-measurable real-valued functions $f$ on $\sY$, and $C_{0}(\sY)$ be the collection of functions $f\in L^{\infty}(\sY)$ such that $f(i,\cdot)\in C_{0}(\bR)$ for all $i\in\mathbf{E}$. Let $\mT:=(\mT_{s,t})_{0\leq s\leq t<\infty}$ be the evolution system corresponding to the transition function $P$ defined by
\begin{align*}
\big(\mT_{s,t}f\big)(i,a):=\int_{\sY}f(j,b)\,dP\big(s,(i,a),(j,b)\big),\quad (i,a)\in\sY,
\end{align*}
where $f\in L^{\infty}(\sY)$. Hence, for any $0\leq s\leq r\leq t<\infty$ and $f\in L^{\infty}(\sY)$, we have
\begin{align}\label{eq:EvolProp}
\mT_{s,r}\circ\mT_{r,t}&=\mT_{s,t},\\
\label{eq:BoundTst} \big\|\mT_{s,t}f\big\|_{\infty}&\leq\|f\|_{\infty}.
\end{align}

\begin{lemma}\label{lem:FellerTranProb}
$(\mT_{s,t})_{0\leq s\leq t}$ is a Feller evolution system. That is,
\begin{itemize}
\item [(a)] $\mT_{s,t}(C_{0}(\sY))\subset C_{0}(\sY)$, for any $0\leq s\leq t<\infty$;
\item [(b)] for any $f\in C_{0}(\sY)$ and $s\in\bR_{+}$, $\mT_{s,t}f$ converges to $f$ uniformly on $\sY$, as $t\downarrow s$;
\item [(c)] for any $f\in C_{0}(\sY)$, the function $(\mT_{s,t}f)(i,a)$ is jointly continuous with respect to $(s,t,i,a)$ on $\{(s,t)\in\bR_{+}^{2}:s\leq t\}\times\sY$.
\end{itemize}
\end{lemma}

\begin{proof}
\textbf{(a)} For any $0\leq s\leq t<\infty$, $f\in C_{0}(\sY)$, and $i\in\mathbf{E}$, we first show that $(\mT_{s,t}f)(i,\,\cdot\,)$ is continuous on $\bR$. By \eqref{eq:DefTranProbXvarphi}, for any $a,a',b\in\bR$ and $j\in\mathbf{E}$,
\begin{align}\label{eq:PTransInvar}
P\big(s,(i,a'),t,\{j\}\times(-\infty,b]\big)=P\big(s,(i,a),t,\{j\}\times(-\infty,b+(a-a')]\big).
\end{align}
Hence, we have
\begin{align*}
\big(\mT_{s,t}f\big)(i,a')=\int_{\sY}f(j,b)\,dP\big(s,(i,a'),t,(j,b)\big)=\int_{\sY}f\big(j,b+(a'-a)\big)\,dP\big(s,(i,a),t,(j,b)\big),
\end{align*}
so that
\begin{align*}
\big|\big(\mT_{s,t}f\big)(i,a')-\big(\mT_{s,t}f\big)(i,a)\big|\leq\bigg|\int_{\sY}\left|f\big(j,b+(a'-a)\big)-f(j,b)\right|dP\big(s,(i,a),t,(j,b)\big)\bigg|.
\end{align*}
The continuity of $(\mT_{s,t}f)(i,\cdot)$ follows from the uniform continuity of $f(j,\cdot)$, since $f(j,\cdot)\in C_{0}(\bR)$, for any $j\in\mathbf{E}$.

It remains to show that $|(\mathbf{T}_{s,t}f)(i,a)|\rightarrow 0$, as $|a|\rightarrow\infty$. By \eqref{eq:DefTranProbXvarphi}, for any $a\in\bR$,
\begin{align}\label{eq:TranProbFiniteBound}
P\big(s,(i,a),t,\mathbf{E}\times\big[a-\overline{v}(t-s),a+\overline{v}(t+s)\big]\big)=1,
\end{align}
where we recall that $\overline{v}:=\max_{i\in\mathbf{E}}|v(i)|$. Hence,
\begin{align*}
\big|\big(\mT_{s,t}f\big)(i,a)\big|\leq\int_{\mathbf{E}\times[a-\overline{v}(t-s),a+\overline{v}(t+s)]}\big|f(j,b)\big|\,dP\big(s,(i,a),t,(j,b)\big)\rightarrow 0,\quad |a|\rightarrow\infty,
\end{align*}
since $f(j,\cdot)\in C_{0}(\bR)$ for any $j\in\mathbf{E}$. This completes the proof of part (a).

\medskip
\noindent
\textbf{(b)} For any $0\leq s\leq t<\infty$ and $f\in C_{0}(\sY)$, by \eqref{eq:TranProbFiniteBound} and \eqref{eq:DefTranProbXvarphi},
\begin{align}
\sup_{(i,a)\in\sY}\!\!\big|\big(\mT_{s,t}f\big)(i,a)-f(i,a)\big|&\leq\sup_{(i,a)\in\sY}\int_{\sY}\big|f(j,b)-f(i,a)\big|\,dP\big(s,(i,a),t,(j,b)\big)\nonumber\\
&\leq\sup_{(i,a)\in\sY}\int_{\sY}\big|f(j,b)-f(i,b)\big|\,dP\big(s,(i,a),t,(j,b)\big)\nonumber\\
&\quad +\sup_{(i,a)\in\sY}\int_{\sY}\big|f(i,b)-f(i,a)\big|\,dP\big(s,(i,a),t,(j,b)\big)\nonumber\\
&\leq 2\|f\|_{\infty}\sup_{(i,a)\in\sY}P\big(s,(i,a),t,(\mathbf{E}\setminus\{i\})\times\bR\big)\nonumber\\
&\quad +\sup_{(i,a)\in\sY}\int_{\mathbf{E}\times[a-\overline{v}h,a+\overline{v}h]}\big|f(i,b)-f(i,a)\big|\,dP\big(s,(i,a),t,(j,b)\big)\nonumber\\
\label{eq:DecompTstf} &\leq 2\|f\|_{\infty}\sup_{i\in\mathbf{E}}\bP_{s,i}^{*}\big(X_{t}^{*}\neq i\big)+\!\sup_{i\in\mathbf{E},|b-a|\in[0,\overline{v}(t-s)]}\!\big|f(i,b)-f(i,a)\big|.
\end{align}
By the right-continuity of the sample paths of $X^{*}$ and by the dominated convergence theorem,
\begin{align}\label{eq:LimitRContXStar}
\lim_{t\downarrow s}\bP_{s,i}^{*}\big(X_{t}^{*}\neq i\big)=\bP_{s,i}^{*}\big(X_{s}^{*}\neq i\big)=0.
\end{align}
Together with the uniform continuity of $f(i,\cdot)$ on $\bR$, we obtain that
\begin{align*}
\lim_{t\downarrow s}\sup_{(i,a)\in\sY}\left|\big(\mT_{s,t}f\big)(i,a)-f(i,a)\right|=0,
\end{align*}
which completes the proof of part (b).

\medskip
\noindent
\textbf{(c)} For any $f\in C_{0}(\sY)$ and $i\in\mathbf{E}$, since $\mathbf{E}$ is finite, it is sufficient to establish the joint continuity of $(\mT_{\,\cdot\,,\,\cdot\,}f)(i,\,\cdot\,)$ at any $(s,t,a)\in\{(s,t)\in\bR^{2}_{+}:s\leq t\}\times\bR$. For any $(s',t',b)\in\{(s,t)\in\bR^{2}_{+}:s\leq t\}\times\bR$ (without loss of generality, assume that $s\leq s'\leq t\leq t'$, as the other cases can be proved similarly),
\begin{align}\label{eq:DecompDiffTst}
\big|\big(\mT_{s',t'}f\big)(i,b)-\big(\mT_{s,t}f\big)(i,a)\big|&\leq\big|\big(\mT_{s',t'}f\big)(i,b)-\big(\mT_{s',t}f\big)(i,b)\big|+\big|\big(\mT_{s',t}f\big)(i,b)-\big(\mT_{s,t}f\big)(i,b)\big|\nonumber\\
&\quad +\big|\big(\mT_{s,t}f\big)(i,b)-\big(\mT_{s,t}f\big)(i,a)\big|.
\end{align}
For the first term in \eqref{eq:DecompDiffTst}, by \eqref{eq:EvolProp}, \eqref{eq:BoundTst}, and part (b), for any $\varepsilon\in(0,\infty)$, there exists $\delta_{1}=\delta_{1}(\varepsilon,t)\in(0,\infty)$ such that, whenever $t'-t\in[0,\delta_{1})$,
\begin{align}\label{eq:EstDiffTst1}
\sup_{(i,b)\in\sY}\big|\big(\mT_{s',t'}f\big)(i,b)-\big(\mT_{s',t}f\big)(i,b)\big|=\big\|\mT_{s',t}\big(\mT_{t,t'}f-f\big)\big\|_{\infty}\leq\big\|\mT_{t,t'}f-f\big\|_{\infty}\leq\frac{\varepsilon}{3}.
\end{align}
As for the last term in \eqref{eq:DecompDiffTst}, it follows from part (a) that $(\mT_{s,t}f)(i,\cdot)$ is uniformly continuous on $\bR$. Hence, there exists $\delta_{2}=\delta_{2}(\varepsilon,s,t)\in(0,\infty)$ such that, whenever $|b-a|\in[0,\delta_{2})$,
\begin{align}\label{eq:EstDiffTst2}
\big|\big(\mT_{s,t}f\big)(i,b)-\big(\mT_{s,t}f\big)(i,a)\big|\leq\frac{\varepsilon}{3}.
\end{align}
It remains to analyze the second term in \eqref{eq:DecompDiffTst}. Using a similar argument leading to \eqref{eq:DecompTstf} (but with $f$ replaced with $\mT_{s',t}f$), we have
\begin{align*}
&\big|\big(\mT_{s',t}f\big)(i,b)-\big(\mT_{s,t}f\big)(i,b)\big|=\big|\big(\mT_{s',t}f\big)(i,b)-\big(\mT_{s,s'}\big(\mT_{s',t}f\big)\big)(i,b)\big|\\
&\quad\leq 2\|\mT_{s',t}f\|_{\infty}\sup_{i\in\mathbf{E}}\bP_{s,i}^{*}\big(X_{s'}^{*}\neq i\big)+\sup_{i\in\mathbf{E}}\sup_{|b-c|\in[0,\overline{v}(s'-s)]}\big|\big(\mT_{s',t}f\big)(i,b)-\big(\mT_{s',t}f\big)(i,c)\big|.
\end{align*}
By \eqref{eq:BoundTst} and \eqref{eq:LimitRContXStar}, there exists $\delta_{31}=\delta_{31}(\varepsilon,s)\in(0,\infty)$ such that, whenever $s'-s\in[0,\delta_{31})$,
\begin{align*}
\big\|\mT_{s',t}f\big\|_{\infty}\sup_{i\in\mathbf{E}}\bP_{s,i}^{*}\big(X_{s'}^{*}\neq i\big)\leq\|f\|_{\infty}\sup_{i\in\mathbf{E}}\bP_{s,i}^{*}\big(X_{s'}^{*}\neq i\big)\leq\frac{\varepsilon}{12}.
\end{align*}
Moreover, by \eqref{eq:PTransInvar} and the uniform continuity of $f(i,\cdot)$ on $\bR$, there exists $\delta_{32}=\delta_{32}(\varepsilon)\in(0,\infty)$ such that, whenever $s'-s\in[0,\delta_{32})$,
\begin{align*}
&\sup_{i\in\mathbf{E}}\sup_{|b-c|\in[0,\overline{v}(s'-s)]}\big|\big(\mT_{s',t}f\big)(i,b)-\big(\mT_{s',t}f\big)(i,c)\big|\\
&\quad =\sup_{i\in\mathbf{E}}\sup_{|b-c|\in[0,\overline{v}(s'-s)]}\bigg|\int_{\sY}f(j,d)\,dP\big(s',(i,b),t,(j,d)\big)-\int_{\sY}f(j,d')\,dP\big(s',(i,c),t,(j,d')\big)\bigg|\\
&\quad =\sup_{i\in\mathbf{E}}\sup_{|b-c|\in[0,\overline{v}(s'-s)]}\bigg|\int_{\sY}f(j,d)\,dP\big(s',(i,b),t,(j,d)\big)-\int_{\sY}f(j,d+c-b)\,dP\big(s',(i,b),t,(j,d)\big)\bigg|\\
&\quad\leq\sup_{i\in\mathbf{E}}\sup_{|b-c|\in[0,\overline{v}(s'-s)]}\int_{\sY}\big|f(j,d)-f(j,d+c-b)\big|\,dP\big(s',(i,b),t,(j,d)\big)\\
&\quad\leq\sup_{j\in\mathbf{E}}\sup_{|d-d'|\in[0,\overline{v}(s'-s)]}\big|f(j,d)-f(j,d')\big|\leq\frac{\varepsilon}{6}.
\end{align*}
Therefore, if $s'-s\in[0,\delta_{3})$, where $\delta_{3}:=\min(\delta_{31},\delta_{32})$, we have
\begin{align}\label{eq:EstDiffTst3}
\sup_{(i,b)\in\sY}\big|\big(\mT_{s',t}f\big)(i,b)-\big(\mT_{s,t}f\big)(i,b)\big|\leq\frac{\varepsilon}{3}.
\end{align}
Combining \eqref{eq:DecompDiffTst} - \eqref{eq:EstDiffTst3}, and letting $\delta=\delta(\varepsilon,s,t):=\min(\delta_{1},\delta_{2},\delta_{3})$, for any $(s',t',b)\in\{(s,t)\in\bR^{2}_{+}:s\leq t\}\times\bR$ such that $|s'-s|+|t'-t|+|b-a|\in[0,\delta]$, we obtain that
\begin{align*}
\big|\big(\mT_{s',t'}f\big)(i,b)-\big(\mT_{s,t}f\big)(i,a)\big|\leq\varepsilon,
\end{align*}
which completes the proof of part (c), and thus concludes the proof of the lemma.
\end{proof}

Let $\Omega$ be the collection of all c\`{a}dl\`{a}g functions $\omega$ on $\bR_{+}$ taking values in $\sY$, with extended value $\omega(\infty)=(\partial,\infty)$, on which we let $(X,\varphi):=(X_{t},\varphi_{t})_{t\in\overline{\bR}_{+}}$ be the coordinate mapping process. By Lemma \ref{lem:TranProbXvarphi}, Lemma \ref{lem:FellerTranProb}, and~\cite[Theorem I.6.3]{GikhmanSkorokhod2004}, there exists a standard time-inhomogeneous Markov family $\cM:=\{(\Omega,\sF,\bF_{s},(X_{t},\varphi_{t})_{t\in[s,\infty]},\bP_{s,(i,a)}),\,(s,i,a)\in\overline{\sZ}\}$, such that
\begin{align}\label{eq:MarginDistXvarphi}
\bP_{s,(i,a)}\big(\big(X_{t},\varphi_{t}\big)\in A\big)=P\big(s,(i,a),t,A\big)=\bP^{*}_{s,i}\bigg(\bigg(X^{*}_{t},\,a+\int_{s}^{t}v\big(X^{*}_{u}\big)\,du\bigg)\in A\bigg).
\end{align}

\subsection{Proof of properties (i) and (ii) in Section \ref{subsec:TimeHomogen}}

The property (i) follows immediately from \eqref{eq:MarginDistXvarphi}. So, it remains to prove property (ii). Towards this end, we first extend \eqref{eq:MarginDistXvarphi} to measurable sets in $\Omega$.

\begin{lemma}\label{lem:XvarphiPathst}
For any $0\leq s\leq t<\infty$, let $\Omega_{s,t}$ be the collection of all c\`{a}dl\`{a}g functions on $[s,t]$ taking values in $\sY$, i.e., $\Omega_{s,t}=\Omega|_{[s,t]}$. Let $\cG_{s,t}$ be the cylindrical $\sigma$-field on $\Omega_{s,t}$ generated by $(X_{u},\varphi_{u})_{u\in[s,t]}$. Then, for any $a\in\bR$ and $C\in\sG_{s,t}$,
\begin{align}\label{eq:XvarphiPathst}
\bP_{s,(i,a)}\big(\big(X_{r},\varphi_{r}\big)_{r\in[s,t]}\in C\big)=\bP^{*}_{s,i}\bigg(\bigg(X^{*}_{r},\,a+\int_{s}^{r}v\big(X^{*}_{u}\big)\,du\bigg)_{r\in[s,t]}\in C\bigg).
\end{align}
\end{lemma}

\begin{proof}
We first shows that, for any $n\in\bN$, $s\leq t_{1}\leq\cdots\leq t_{n}\leq t$, $j_{1},\ldots,j_{n}\in\mathbf{E}$, and $b_{1},\ldots,b_{n}\in\bR$,
\begin{align}
&\bP_{s,(i,a)}\big(X_{t_{j}}\!=x_{j},\,\varphi_{t_{j}}\!\in(-\infty,b_{j}],\,1\leq j\leq n\big)\nonumber\\
\label{eq:Xvarphicylst} &\quad =\bP^{*}_{s,i}\Big(X^{*}_{t_{j}}\!=x_{j},\,a+\!\int_{s}^{t_{j}}v\big(X^{*}_{u}\big)\,du\!\in(-\infty,b_{j}],\,1\leq j\leq n\Big).
\end{align}
The proof will proceed by induction in $n$. For $n=1$, \eqref{eq:Xvarphicylst} is just a special case of \eqref{eq:MarginDistXvarphi}. Assume that \eqref{eq:Xvarphicylst} holds for $n=p\in\bN$. For any $s\leq t_{1}\leq\cdots\leq t_{p+1}\leq t$, $j_{1},\ldots,j_{p+1}\in\mathbf{E}$, and $b_{1},\ldots,b_{p+1}\in\bR$, by \eqref{eq:MarginDistXvarphi} and the Markov property of $(X,\varphi)$,
\begin{align}
&\bP_{s,(i,a)}\big(X_{t_{j}}=x_{j},\,\varphi_{t_{j}}\in(-\infty,b_{j}],\,j=1,\ldots,p+1\big)\nonumber\\
&\quad =\bE_{s,(i,a)}\Big(\1_{\{X_{t_{j}}=x_{j},\,\varphi_{t_{j}}\in(-\infty,b_{j}],\,j=1,\ldots,p\}}\bP_{t_{p},(X_{t_{p}},\varphi_{t_{p}})}\big(X_{t_{p+1}}=x_{p+1},\,\varphi_{t_{p+1}}\in(-\infty,b_{p+1}]\big)\Big)\nonumber\\
&\quad =\bE_{s,(i,a)}\bigg(\1_{\{X_{t_{j}}=x_{j},\,\varphi_{t_{j}}\in(-\infty,b_{j}],\,j=1,\ldots,p\}}\nonumber\\
\label{eq:Xvarphim+1} &\qquad\qquad\quad\,\,\,\,\,\cdot\bP_{t_{p},k}^{*}\bigg(X_{t_{p+1}}^{*}\!\!=\!x_{p+1},\,c+\!\int_{t_{p}}^{t_{p+1}}v\big(X^{*}_{u}\big)\,du\in(-\infty,b_{p+1}]\bigg)\bigg|_{(k,c)=(X_{t_{p}},\varphi_{t_{p}})}\bigg).
\end{align}
By the induction hypothesis for $n=p$, the joint distribution of $(X_{t_{j}},\varphi_{t_{j}},j=1,\ldots,p)$ under $\bP_{s,(i,a)}$ coincides with the joint distribution of $(X_{t_{j}}^{*},a+\int_{s}^{t_{j}}v(X^{*}_{u})du,j=1,\ldots,p)$ under $\bP_{s,i}^{*}$. Applying the standard procedure of approximation by simple functions we conclude that for any bounded measurable function $f:(\mathbf{E}^{n}\times\bR^{p},2^{\mathbf{E}^{p}}\otimes\cB(\bR^{p}))\rightarrow(\bR,\cB(\bR))$,
\begin{align*}
\bE_{s,(i,a)}\Big(\!f\big(X_{t_{1}},\ldots,X_{t_{p}},\varphi_{t_{1}},\ldots,\varphi_{t_{p}}\big)\!\Big)\!=\!\bE_{s,i}^{*}\bigg(\!f\bigg(\!X_{t_{1}}^{*},\ldots,X_{t_{p}}^{*},a\!+\!\!\int_{s}^{t_{1}}\!\!v\big(X^{*}_{u}\big)du,\ldots,a\!+\!\!\int_{s}^{t_{p}}\!\!v\big(X^{*}_{u}\big)du\!\bigg)\!\bigg).
\end{align*}
Together with \eqref{eq:Xvarphim+1} and the Markov property of $X^{*}$, we obtain that
\begin{align*}
&\bP_{s,(i,a)}\big(X_{t_{j}}=x_{j},\,\varphi_{t_{j}}\in(-\infty,b_{j}],\,j=1,\ldots,p+1\big)\\
&\quad =\bE_{s,i}^{*}\bigg(\1_{\big\{X_{t_{j}}^{*}=x_{j},\,a+\int_{s}^{t_{j}}v(X_{u}^{*})du\in(-\infty,b_{j}],\,j=1,\ldots,p\big\}}\\
&\qquad\qquad\,\,\,\cdot\bP_{t_{p},k}^{*}\bigg(X_{t_{p+1}}^{*}=x_{p+1},\,c+\int_{t_{p}}^{t_{p+1}}v\big(X^{*}_{u}\big)\,du\in(-\infty,b_{p+1}]\bigg)\bigg|_{(k,c)=(X^{*}_{t_{p}},\int_{s}^{t_{p}}v(X^{*}_{u})du)}\bigg)\\
&\quad=\bP_{s,i}^{*}\bigg(X_{t_{j}}^{*}=x_{j},\,a+\int_{s}^{t_{j}}v\big(X_{u}^{*}\big)\,du\in(-\infty,b_{j}],\,j=1,\ldots,p+1\bigg),
\end{align*}
By induction, we complete the proof of \eqref{eq:Xvarphicylst}.

Next, let
\begin{align*}
\cH:=\big\{C\in\cG_{s,t}:\,\text{the equality (\ref{eq:XvarphiPathst}) holds}\big\}.
\end{align*}
The above arguments show that $\cH\supset\cH_{c}$,  where $\cH_{c}$ denotes the collection of all cylinder sets on $\Omega_{s,t}$ of the form
\begin{align*}
\big\{\omega_{s,t}=(\omega_{s,t}^{1},\omega_{s,t}^{2})\in\Omega_{s,t}:\,\omega_{s,t}^{1}(t_{j})=x_{j},\,\omega_{s,t}^{2}(t_{j})\in(-\infty,b_{j}],\,j=1,\ldots,n\big\},
\end{align*}
for some $x_{1},\ldots,x_{n}\in\mathbf{E}$, $b_{1},\ldots,b_{n}\in\bR$, $s\leq t_{1}\leq\cdots\leq t_{n}\leq t$, and $n\in\bN$. Clearly, $\cH_{c}$ is a $\pi$-system on $\Omega_{s,t}$, and one can check that $\cH$ is closed under complements and countable disjoint unions with $\Omega_{s,t}\in\cH$. Therefore, $\cH$ is a $\lambda$-system on $\Omega_{s,t}$, and by the monotone class theorem, $\cH\supset\sigma(\cH_{c})=\cG_{s,t}$, and thus $\cH=\cG_{s,t}$, which completes the proof of the lemma.
\end{proof}

\begin{proof}[Proof of property (ii), i.e. of \eqref{eq:DistvarphiInt}.]
For any $(s,i,a)\in\sZ$ and $t\in[s,\infty)$, let
\begin{align*}
C_{s,t}:=\bigg\{\omega_{s,t}=(\omega_{s,t}^{1},\omega_{s,t}^{2})\in\Omega_{s,t}:\,\omega_{s,t}^{2}(t)=a+\int_{s}^{t}v\big(\omega_{s,t}^{1}(u)\big)\,du\bigg\}\in\cG_{s,t}.
\end{align*}
By Lemma \ref{lem:XvarphiPathst},
\begin{align*}
\bP_{s,(i,a)}\bigg(\varphi_{t}=a+\int_{s}^{t}v\big(X_{u}\big)\,du\bigg)&=\bP_{s,(i,a)}\big(\big(X_{r},\varphi_{r}\big)_{r\in[s,t]}\in C_{s,t}\big)\\
&=\bP_{s,i}^{*}\bigg(\bigg(X_{r}^{*},\,a+\int_{s}^{r}v\big(X_{u}^{*}\big)\,du\bigg)\in C_{s,t}\bigg)=1.
\end{align*}
Since $\varphi$ has c\`{a}dl\`{a}g sample paths, and $a+\int_{s}^{t}v(X_{u})du$, $t\in[s,\infty)$, has continuous sample paths, \cite[Problem 1.1.5]{KaratzasShreve1998} implies that these two processes are indistinguishable, which implies~\eqref{eq:DistvarphiInt}.
\end{proof}

\section{Proof of Lemma \ref{lem:UnifContf}}\label{sec:AppendixB}

For any $\varepsilon>0$, $i\in\mathbf{E}$, and $(s_{1},\ell_{1}),(s_{2},\ell_{2})\in\bR_{+}^{2}$, without loss of generality, assume that $s_{2}\geq s_{1}$ and $\ell_{2}\geq\ell_{1}$. Then,
\begin{align}\label{eq:DecompfPlussell12}
\big|f_{+}(s_{2},i,\ell_{2})-f_{+}(s_{1},i,\ell_{1})\big|\leq\big|f_{+}(s_{2},i,\ell_{2})-f_{+}(s_{2},i,\ell_{1})\big|+\big|f_{+}(s_{2},i,\ell_{1})-f_{+}(s_{1},i,\ell_{1})\big|
\end{align}
The proof will be divided into three steps.

\medskip
\noindent
\textbf{Step 1.} We {begin by investigating}  the first term in \eqref{eq:DecompfPlussell12}. Noting that $(\cP_{\ell}^{+}g^{+})(\infty,\partial)=g^{+}(\infty,\partial)=0$, by \eqref{eq:FuntfPlus}, Corollary~\ref{cor:StrongMarkovExpTildeTau}, and \eqref{eq:DefPellPlusTildeP},
\begin{align}
&\big|f_{+}(s_{2},i,\ell_{2})-f_{+}(s_{2},i,\ell_{1})\big|=\bigg|\wt{\bE}_{s_{2},i,0}\Big(g^{+}\Big(Z_{\wt{\tau}_{\ell_{2}}^{+}}^{1},Z_{\wt{\tau}_{\ell_{2}}^{+}}^{2}\Big)\Big)-\wt{\bE}_{s_{2},i,0}\Big(g^{+}\Big(Z_{\wt{\tau}_{\ell_{1}}^{+}}^{1},Z_{\wt{\tau}_{\ell_{1}}^{+}}^{2}\Big)\Big)\bigg|\nonumber\\
&=\bigg|\wt{\bE}_{s_{2},i,0}\Big(\big(\cP_{\ell_{1}}^{+}g^{+}\big)\Big(Z^{1}_{\wt{\tau}_{\ell_{2}-\ell_{1}}^{+}},Z^{2}_{\wt{\tau}_{\ell_{2}-\ell_{1}}^{+}}\Big)\Big)-\wt{\bE}_{s_{2},i,0}\Big(g^{+}\Big(Z_{\wt{\tau}_{\ell_{1}}^{+}}^{1},Z_{\wt{\tau}_{\ell_{1}}^{+}}^{2}\Big)\Big)\bigg|\nonumber\\
&=\left|\wt{\bE}_{s_{2},i,0}\!\left(\1_{\{\wt{\tau}^{+}_{\ell_{1}}<\infty\}}\bigg(\wt{\bE}_{Z_{\wt{\tau}_{\ell_{1}}^{+}}^{1},Z_{\wt{\tau}_{\ell_{1}}^{+}}^{2},0}\Big(g^{+}\Big(Z_{\wt{\tau}_{\ell_{2}-\ell_{1}}^{+}}^{1},Z_{\wt{\tau}_{\ell_{2}-\ell_{1}}^{+}}^{2}\Big)\Big)-g^{+}\Big(Z_{\wt{\tau}_{\ell_{1}}^{+}}^{1},Z_{\wt{\tau}_{\ell_{1}}^{+}}^{2}\Big)\bigg)\right)\right|\nonumber\\
\label{eq:IniEstfPlusell121} &\leq\sup_{(t,j)\in\sX_{+}}\bigg|\wt{\bE}_{t,j,0}\Big(g^{+}\Big(Z_{\wt{\tau}_{\ell_{2}-\ell_{1}}^{+}}^{1},Z_{\wt{\tau}_{\ell_{2}-\ell_{1}}^{+}}^{2}\Big)\Big)-g^{+}(t,j)\bigg|.
\end{align}
Recall that $\wt{\gamma}_{1}$ is the first jump time of $Z^{2}$. For any $(t,j)\in\sX_{+}$, on the event $\{Z^{3}_{u}=\int_{0}^{u}v(Z^{2}_{r})dr,\forall u\geq 0\}$ (which has probability 1 under $\wt{\bP}_{t,j,0}$ in view of \eqref{eq:DistTildeZ3IntTildeZ2}), we have
\begin{align*}
\wt{\gamma}_{1}(\wt{\omega})>\frac{\ell_{2}-\ell_{1}}{v(j)},\,Z_{0}(\wt{\omega})=(t,j,0)&\iff Z^{3}_{\wt{\gamma}_{1}}(\wt{\omega})>\ell_{2}-\ell_{1},\,Z_{0}(\wt{\omega})=(t,j,0)\\
&\,\implies\wt{\tau}_{\ell_{2}-\ell_{1}}^{+}(\wt{\omega})=\frac{\ell_{2}-\ell_{1}}{v(j)},\,Z^{2}_{\wt{\tau}_{\ell_{2}-\ell_{1}}^{+}}(\wt{\omega})=j.
\end{align*}
Hence, by \eqref{eq:ProcZ} and \eqref{eq:DistFirstJumpTildeZ2}, for any $(t,j)\in\sX_{+}$,
\begin{align}
&\bigg|\wt{\bE}_{t,j,0}\Big(g^{+}\Big(Z_{\wt{\tau}_{\ell_{2}-\ell_{1}}^{+}}^{1},Z_{\wt{\tau}_{\ell_{2}-\ell_{1}}^{+}}^{2}\Big)\Big)-g^{+}(t,j)\bigg|\nonumber\\
&\quad\leq\wt{\bE}_{t,j,0}\bigg(\Big|g^{+}\Big(Z_{\wt{\tau}_{\ell_{2}-\ell_{1}}^{+}}^{1},Z_{\wt{\tau}_{\ell_{2}-\ell_{1}}^{+}}^{2}\Big)\Big|\1_{\{\wt{\gamma}_{1}\leq(\ell_{2}-\ell_{1})/v(j)\}}\bigg)\nonumber\\
&\qquad +\bigg|\wt{\bE}_{t,j,0}\Big(g^{+}\Big(Z_{\wt{\tau}_{\ell_{2}-\ell_{1}}^{+}}^{1},Z_{\wt{\tau}_{\ell_{2}-\ell_{1}}^{+}}^{2}\Big)\1_{\{\wt{\gamma}_{1}>(\ell_{2}-\ell_{1})/v(j)\}}\Big)-g^{+}(t,j)\bigg|\nonumber\\
&\quad\leq\big\|g^{+}\big\|_{\infty}\wt{\bP}_{t,j,0}\bigg(\wt{\gamma}_{1}\leq\frac{\ell_{2}-\ell_{1}}{v(j)}\bigg)+\bigg|\wt{\bE}_{t,j,0}\Big(g^{+}\Big(Z_{(\ell_{2}-\ell_{1})/v(j)}^{1},j\Big)\1_{\{\wt{\gamma}_{1}>(\ell_{2}-\ell_{1})/v(j)\}}\Big)-g^{+}(t,j)\bigg|\nonumber\\
&\quad\leq\frac{K\big\|g^{+}\big\|_{\infty}}{\underline{v}}(\ell_{2}-\ell_{1})+\left|g^{+}\bigg(t+\frac{\ell_{2}-\ell_{1}}{v(j)},j\bigg)\wt{\bP}_{t,j,0}\bigg(\wt{\gamma}_{1}>\frac{\ell_{2}-\ell_{1}}{v(j)}\bigg)-g^{+}(t,j)\right|\nonumber\\
&\quad\leq\frac{K\big\|g^{+}\big\|_{\infty}}{\underline{v}}(\ell_{2}\!-\!\ell_{1})+\left|g^{+}\!\bigg(t\!+\!\frac{\ell_{2}\!-\!\ell_{1}}{v(j)},j\bigg)\wt{\bP}_{t,j,0}\bigg(\wt{\gamma}_{1}\!\leq\!\frac{\ell_{2}-\ell_{1}}{v(j)}\bigg)\right|+\left|g^{+}\bigg(t\!+\!\frac{\ell_{2}\!-\!\ell_{1}}{v(j)},j\bigg)-g^{+}(t,j)\right|\nonumber\\ &\quad\leq\frac{2K\big\|g^{+}\big\|_{\infty}}{\underline{v}}(\ell_{2}-\ell_{1})+\left|g^{+}\bigg(t+\frac{\ell_{2}-\ell_{1}}{v(j)},j\bigg)-g^{+}(t,j)\right|\nonumber\\
\label{eq:EstgPlusell12} &\quad\leq\frac{2K\big\|g^{+}\big\|_{\infty}}{\underline{v}}(\ell_{2}-\ell_{1})+w_{g^{+}}\bigg(\frac{\ell_{2}-\ell_{1}}{\underline{v}}\bigg),
\end{align}
where we recall {that} $\underline{v}=\min_{i\in\mathbf{E}}|v(i)|$, and
\begin{align*}
w_{g^{+}}(\delta):=\sup_{j\in\mathbf{E}_{+}}\sup_{r,u\in\bR_{+}:\,|r-u|\in[0,\delta]}\big|g^{+}(r,j)-g^{+}(u,j)\big|
\end{align*}
is the modulus of continuity of $g^{+}$. Combining \eqref{eq:IniEstfPlusell121} and \eqref{eq:EstgPlusell12} leads to
\begin{align}\label{eq:EstfPlussell121}
\big|f_{+}(s_{2},i,\ell_{2})-f_{+}(s_{2},i,\ell_{1})\big|\leq\frac{2K\big\|g^{+}\big\|_{\infty}}{\underline{v}}(\ell_{2}-\ell_{1})+w_{g^{+}}\bigg(\frac{\ell_{2}-\ell_{1}}{\underline{v}}\bigg),
\end{align}
which completes the proof in Step 1.

\medskip
\noindent
\textbf{Step 2.} Next, we analyze the second term in \eqref{eq:DecompfPlussell12} by decomposing it as
\begin{align}
\big|f_{+}(s_{1},i,\ell_{1})-f_{+}(s_{2},i,\ell_{1})\big|&=\Big|\wt{\bE}_{s_{1},i,0}\Big(g^{+}\Big(Z_{\wt{\tau}_{\ell_{1}}^{+}}^{1},Z_{\wt{\tau}_{\ell_{1}}^{+}}^{2}\Big)\Big)-f_{+}(s_{2},i,\ell_{1})\Big|\nonumber\\
&\leq\wt{\bE}_{s_{1},i,0}\bigg(\Big|g^{+}\Big(Z_{\wt{\tau}_{\ell_{1}}^{+}}^{1},Z_{\wt{\tau}_{\ell_{1}}^{+}}^{2}\Big)-f_{+}(s_{2},i,\ell_{1})\Big|\1_{\{\wt{\tau}_{\ell_{1}}^{+}\leq s_{2}-s_{1}\}}\bigg)\nonumber\\
&\quad +\Big|\wt{\bE}_{s_{1},i,0}\Big(\Big(g^{+}\Big(Z_{\wt{\tau}_{\ell_{1}}^{+}}^{1},Z_{\wt{\tau}_{\ell_{1}}^{+}}^{2}\Big)-f_{+}(s_{2},i,\ell_{1})\Big)\1_{\{\wt{\tau}_{\ell}^{+}>s_{2}-s_{1}\}}\Big)\Big|\nonumber\\
\label{eq:DecompfPlussell122} &=:\cI_{1}+\cI_{2}.
\end{align}

To estimate $\cI_{1}$, we first note that when $i\in\mathbf{E}_{-}$ (so that $v(i)<0$), it follows from \eqref{eq:DistTildeZ3IntTildeZ2} that
\begin{align*}
\wt{\bP}_{s_{1},i,0}\big(\wt{\gamma}_{1}\!>\!s_{2}\!-\!s_{1},\wt{\tau}_{\ell_{1}}^{+}\!\leq\!s_{2}\!-\!s_{1}\big)&=\wt{\bP}_{s_{1},i,0}\big(\wt{\gamma}_{1}>s_{2}-s_{1},\,Z_{u}^{3}>\ell_{1}\,\,\text{for some }u\in[0,s_{2}-s_{1}]\big)\\
&=\wt{\bP}_{s_{1},i,0}\bigg(\wt{\gamma}_{1}>s_{2}\!-\!s_{1},\,\int_{0}^{u}v\big(Z^{2}_{r}\big)\,dr>\ell_{1}\,\,\text{for some }u\in[0,s_{2}\!-\!s_{1}]\bigg)\\
&=\wt{\bP}_{s_{1},i,0}\big(\wt{\gamma}_{1}>s_{2}-s_{1},\,v(i)u>\ell_{1}\,\,\text{for some }u\in[0,s_{2}-s_{1}]\big)=0.
\end{align*}
Hence, when $i\in\mathbf{E}_{-}$, by \eqref{eq:FuntfPlus} and \eqref{eq:DistFirstJumpTildeZ2} we have
\begin{align}
\cI_{1}&=\wt{\bE}_{s_{1},i,0}\bigg(\Big|g^{+}\Big(Z_{\wt{\tau}_{\ell_{1}}^{+}}^{1},Z_{\wt{\tau}_{\ell_{1}}^{+}}^{2}\Big)-\wt{\bE}_{s_{2},i,0}\Big(g^{+}\Big(Z_{\wt{\tau}_{\ell_{1}}^{+}}^{1},Z_{\wt{\tau}_{\ell_{1}}^{+}}^{2}\Big)\Big)\Big|\1_{\{\wt{\tau}_{\ell_{1}}^{+}\leq s_{2}-s_{1},\wt{\gamma}_{1}\leq s_{2}-s_{1}\}}\bigg)\nonumber\\
\label{eq:EstfPlussell1221Minus} &\leq 2\big\|g^{+}\big\|_{\infty}\wt{\bP}_{s_{1},i,0}\big(\wt{\gamma}_{1}\leq s_{2}-s_{1}\big)\leq 2K\,\big\|g^{+}\big\|_{\infty}(s_{2}-s_{1}).
\end{align}
In what follows, assume that $i\in\mathbf{E}_{+}$. We further decompose $\cI_{1}$ as
\begin{align}
\cI_{1}&\leq\wt{\bE}_{s_{1},i,0}\bigg(\Big|g^{+}\Big(Z_{\wt{\tau}_{\ell_{1}}^{+}}^{1},Z_{\wt{\tau}_{\ell_{1}}^{+}}^{2}\Big)-g^{+}\Big(s_{2},Z_{\wt{\tau}_{\ell_{1}}^{+}}^{2}\Big)\Big|\1_{\{\wt{\tau}_{\ell_{1}}^{+}\leq s_{2}-s_{1}\}}\bigg)\nonumber\\
&\quad +\wt{\bE}_{s_{1},i,0}\bigg(\Big|g^{+}\Big(s_{2},Z_{\wt{\tau}_{\ell_{1}}^{+}}^{2}\Big)-g^{+}(s_{2},i)\Big|\1_{\{\wt{\tau}_{\ell_{1}}^{+}\leq s_{2}-s_{1}\}}\bigg)\nonumber\\
&\quad +\wt{\bE}_{s_{1},i,0}\Big(\big|g^{+}(s_{2},i)-f_{+}(s_{2},i,\ell_{1})\big|\1_{\{\wt{\tau}_{\ell_{1}}^{+}\leq s_{2}-s_{1}\}}\Big)\nonumber\\
\label{eq:DecompfPlussell1221Plus} &=:\cI_{11}+\cI_{12}+\cI_{13}.
\end{align}
For $\cI_{11}$, by \eqref{eq:ProcZ}, we have
\begin{align}\label{eq:EstfPlussell1221Plus1}  \cI_{11}=\wt{\bE}_{s_{1},i,0}\Big(\Big|g^{+}\Big(s_{1}+\wt{\tau}_{\ell_{1}}^{+},Z_{\wt{\tau}_{\ell_{1}}^{+}}^{2}\Big)-g^{+}\Big(s_{2},Z_{\wt{\tau}_{\ell_{1}}^{+}}^{2}\Big)\Big|\1_{\{\wt{\tau}_{\ell_{1}}^{+}\leq s_{2}-s_{1}\}}\Big)\leq w_{g^{+}}(s_{2}-s_{1}).
\end{align}
As for $\cI_{12}$, note that $Z_{\wt{\tau}_{\ell_{1}}^{+}}^{2}=Z_{0}^{2}$ on $\{\wt{\gamma}_{1}>s_{2}-s_{1},\wt{\tau}_{\ell_{1}}^{+}\leq s_{2}-s_{1}\}$, and thus
\begin{align*}
\wt{\bE}_{s_{1},i,0}\bigg(\Big|g^{+}\Big(s_{2},Z_{\wt{\tau}_{\ell_{1}}^{+}}^{2}\Big)-g^{+}(s_{2},i)\Big|\1_{\{\wt{\gamma}_{1}>s_{2}-s_{1},\wt{\tau}_{\ell_{1}}^{+}\leq s_{2}-s_{1}\}}\bigg)=0.
\end{align*}
Hence, by \eqref{eq:DistFirstJumpTildeZ2}, we get
\begin{align}
\cI_{12}&=\wt{\bE}_{s_{1},i,0}\bigg(\Big|g^{+}\Big(s_{2},Z_{\wt{\tau}_{\ell_{1}}^{+}}^{2}\Big)-g^{+}(s_{2},i)\Big|\1_{\{\wt{\tau}_{\ell_{1}}^{+}\leq s_{2}-s_{1},\wt{\gamma}_{1}\leq s_{2}-s_{1}\}}\bigg)\nonumber\\
\label{eq:EstfPlussell1221Plus2} &\leq 2\,\big\|g^{+}\big\|_{\infty}\,\wt{\bP}_{s_{1},i,0}\big(\wt{\gamma}_{1}\leq s_{2}-s_{1}\big)\leq 2K\,\big\|g^{+}\big\|_{\infty}(s_{2}-s_{1}).
\end{align}
It remains to analyze $\cI_{13}$. Note that when $\ell_{1}>\overline{v}(t-s)$, by \eqref{eq:DistTildeZ3IntTildeZ2},
\begin{align*}
\wt{\bP}_{s_{1},i,0}\big(\wt{\tau}_{\ell_{1}}^{+}\leq s_{2}-s_{1}\big)&=\wt{\bP}_{s_{1},i,0}\Big(Z^{3}_{u}>\ell_{1}\,\,\,\text{for some } u\in[0,s_{2}-s_{1}]\Big)\\
&=\wt{\bP}_{s_{1},i,0}\bigg(\int_{0}^{u}v\big(Z^{2}_{r}\big)\,dr>\ell_{1}>\overline{v}(s_{2}-s_{1})\,\,\,\text{for some } u\in[0,s_{2}-s_{1}]\bigg)=0.
\end{align*}
It follows from \eqref{eq:Trivfplus} and \eqref{eq:EstfPlussell121} that
\begin{align}
\cI_{13}&=\big|g^{+}(s_{2},i)-f_{+}(s_{2},i,\ell_{1})\big|\,\wt{\bP}_{s_{1},i,0}\big(\wt{\tau}_{\ell_{1}}^{+}\leq s_{2}-s_{1}\big)\leq\big|f_{+}(s_{2},i,0)-f_{+}(s_{2},i,\ell_{1})\big|\1_{\{\ell_{1}\leq\overline{v}(s_{2}-s_{1})\}}\nonumber\\
\label{eq:EstfPlussell1221Plus3} &\leq\!\sup_{\substack{r_{1},r_{2}\in\bR_{+} \\ |r_{1}-r_{2}|\leq\overline{v}(s_{2}-s_{1})}}\!\!\!\big|f_{+}(s_{2},i,r_{1})\!-\!f_{+}(s_{2},i,r_{2})\big|\leq\frac{2\overline{v}}{\underline{v}}K\big\|g^{+}\big\|_{\infty}\!(s_{2}\!-\!s_{1})+w_{g^{+}}\!\bigg(\frac{\overline{v}}{\underline{v}}(s_{2}-s_{1})\bigg),
\end{align}
where we recall $\overline{v}=\max_{i\in\mathbf{E}}|v(i)|$. Combining \eqref{eq:DecompfPlussell1221Plus}$-$\eqref{eq:EstfPlussell1221Plus3}, we obtain that, for any $i\in\mathbf{E}_{+}$,
\begin{align}\label{eq:EstfPlussell1221Plus}
\cI_{1}\leq 2\bigg(1+\frac{\overline{v}}{\underline{v}}\bigg)K\,\big\|g^{+}\big\|_{\infty}(s_{2}-s_{1})+2w_{g^{+}}\bigg(\frac{\overline{v}}{\underline{v}}(s_{2}-s_{1})\bigg).
\end{align}
Comparing \eqref{eq:EstfPlussell1221Minus} with \eqref{eq:EstfPlussell1221Plus}, we see that \eqref{eq:EstfPlussell1221Plus} holds for any $i\in\mathbf{E}$, completing the study of $\cI_{1}$.

Next, we will investigate $\cI_{2}$. Note that $Z^{3}_{s_{2}-s_{1}}\leq\ell_{1}$ on the event $\{\wt{\tau}_{\ell_{1}}^{+}>s_{2}-s_{1}\}$. Hence, by Corollary \ref{cor:StrongMarkovExp} and \eqref{eq:FuntfPlus}, we further decompose $\cI_{2}$ as
\begin{align}
\cI_{2}&=\bigg|\wt{\bE}_{s_{1},i,0}\bigg(\Big(\wt{\bE}_{s_{2},j,a}\Big(g^{+}\Big(Z_{\wt{\tau}_{\ell_{1}}^{+}}^{1},Z_{\wt{\tau}_{\ell_{1}}^{+}}^{2}\Big)\Big)\Big|_{(j,a)=\big(Z_{s_{2}-s_{1}}^{2},Z_{s_{2}-s_{1}}^{3}\big)}-f_{+}(s_{2},i,\ell_{1})\Big)\1_{\{\wt{\tau}_{\ell_{1}}^{+}>s_{2}-s_{1}\}}\bigg)\bigg|\nonumber\\
&=\bigg|\wt{\bE}_{s_{1},i,0}\bigg(\Big(\wt{\bE}_{s_{2},j,0}\Big(g^{+}\Big(Z_{\wt{\tau}_{\ell_{1}-a}^{+}}^{1},Z_{\wt{\tau}_{\ell_{1}-a}^{+}}^{2}\Big)\Big)\Big|_{(j,a)=\big(Z_{s_{2}-s_{1}}^{2},Z_{s_{2}-s_{1}}^{3}\big)}-f_{+}(s_{2},i,\ell_{1})\Big)\1_{\{\wt{\tau}_{\ell_{1}}^{+}>s_{2}-s_{1}\}}\bigg)\bigg|\nonumber\\
&=\Big|\wt{\bE}_{s_{1},i,0}\Big(\Big(f_{+}\Big(s_{2},Z_{s_{2}-s_{1}}^{2},\ell_{1}-Z_{s_{2}-s_{1}}^{3}\Big)-f_{+}(s_{2},i,\ell_{1})\Big)\1_{\{\wt{\tau}_{\ell_{1}}^{+}>s_{2}-s_{1}\}}\Big)\Big|\nonumber\\
&\leq\wt{\bE}_{s_{1},i,0}\bigg(\Big|f_{+}\Big(s_{2},Z_{s_{2}-s_{1}}^{2},\ell_{1}-Z_{s_{2}-s_{1}}^{3}\Big)-f_{+}\Big(s_{2},i,\ell_{1}-Z_{s_{2}-s_{1}}^{3}\Big)\Big|\1_{\{\wt{\tau}_{\ell_{1}}^{+}>s_{2}-s_{1}\}}\bigg)\nonumber\\
&\quad +\wt{\bE}_{s_{1},i,0}\bigg(\Big|f_{+}\Big(s_{2},i,\ell_{1}-Z_{s_{2}-s_{1}}^{3}\Big)-f_{+}(s_{2},i,\ell_{1})\Big|\1_{\{\wt{\tau}_{\ell_{1}}^{+}>s_{2}-s_{1}\}}\bigg)\nonumber\\
\label{eq:DecompfPlussell1222} &=:\cI_{21}+\cI_{22}.
\end{align}
By \eqref{eq:FuntfPlus} and \eqref{eq:DistFirstJumpTildeZ2}, an argument similar to those leading to \eqref{eq:EstfPlussell1221Plus2} implies that
\begin{align}
\cI_{21}&=\wt{\bE}_{s_{1},i,0}\bigg(\Big|f_{+}\Big(s_{2},Z_{s_{2}-s_{1}}^{2},\ell_{1}-Z_{s_{2}-s_{1}}^{3}\Big)-f_{+}\Big(s_{2},i,\ell_{1}-Z_{s_{2}-s_{1}}^{3}\Big)\Big|\1_{\{\wt{\tau}_{\ell_{1}}^{+}>s_{2}-s_{1},\wt{\gamma}_{1}\leq s_{2}-s_{1}\}}\bigg)\nonumber\\
\label{eq:EstfPlussell12221} &\leq 2\,\big\|g^{+}\big\|_{\infty}\,\wt{\bP}_{s_{1},i,0}\big(\wt{\gamma}_{1}\leq s_{2}-s_{1}\big)\leq 2K\,\big\|g^{+}\big\|_{\infty}(s_{2}-s_{1}).
\end{align}
To estimate $\cI_{22}$, note that by \eqref{eq:DistTildeZ3IntTildeZ2},
\begin{align*}
\wt{\bP}_{s_{1},i,0}\Big(\big|Z_{s_{2}-s_{1}}^{3}\big|\leq\overline{v}(s_{2}-s_{1})\Big)=\wt{\bP}_{s_{1},i,0}\!\left(\bigg|\int_{0}^{s_{2}-s_{1}}v\big(Z^{2}_{u}\big)\,du\big|\leq\overline{v}(s_{2}-s_{1})\right)=1.
\end{align*}
Together with \eqref{eq:EstfPlussell121}, we have
\begin{align}
\cI_{22}&=\wt{\bE}_{s_{1},i,0}\bigg(\Big|f_{+}\Big(s_{2},i,\ell_{1}-Z_{s_{2}-s_{1}}^{3}\Big)-f_{+}(s_{2},i,\ell_{1})\Big|\1_{\{\wt{\tau}_{\ell_{1}}^{+}>s_{2}-s_{1},|Z_{s_{2}-s_{1}}^{3}|\leq\overline{v}(s_{2}-s_{1})\}}\bigg)\nonumber\\
\label{eq:EstfPlussell12222} &\leq\sup_{\substack{r_{1},r_{2}\in\bR_{+} \\ |r_{1}-r_{2}|\leq\overline{v}(s_{2}-s_{1})}}\!\!\!\big|f_{+}(s_{2},i,r_{1})\!-\!f_{+}(s_{2},i,r_{2})\big|\leq\frac{2\overline{v}}{\underline{v}}K\,\big\|g^{+}\big\|_{\infty}(s_{2}\!-\!s_{1})+w_{g^{+}}\!\bigg(\frac{\overline{v}}{\underline{v}}(s_{2}\!-\!s_{1})\!\bigg).
\end{align}
Combining \eqref{eq:DecompfPlussell1222}$-$\eqref{eq:EstfPlussell12222} leads to
\begin{align}\label{eq:EstfPlussell1222}
\cI_{2}\leq 2\bigg(1+\frac{\overline{v}}{\underline{v}}\bigg)K\,\big\|g^{+}\big\|_{\infty}(s_{2}-s_{1})+w_{g^{+}}\bigg(\frac{\overline{v}}{\underline{v}}(s_{2}-s_{1})\bigg).
\end{align}
Therefore, by \eqref{eq:DecompfPlussell122}, \eqref{eq:EstfPlussell1221Plus}, and \eqref{eq:EstfPlussell1222}, we obtain that
\begin{align}\label{eq:EstfPlussell122}
\big|f_{+}(s_{1},i,\ell_{1})-f_{+}(s_{2},i,\ell_{1})\big|\leq 4\,\bigg(1+\frac{\overline{v}}{\underline{v}}\bigg)K\,\big\|g^{+}\big\|_{\infty}(s_{2}-s_{1})+3\,w_{g^{+}}\bigg(\frac{\overline{v}}{\underline{v}}(s_{2}-s_{1})\bigg),
\end{align}
which completes the analysis in Step 2.

\medskip
\noindent
\textbf{Step 3.} By \eqref{eq:DecompfPlussell12}, \eqref{eq:EstfPlussell121}, and \eqref{eq:EstfPlussell122}, we have
\begin{align*}
\big|f_{+}(s_{1},i,\ell_{2})-f_{+}(s_{2},i,\ell_{1})\big|&\leq\frac{2}{\underline{v}}\,K\,\big\|g^{+}\big\|_{\infty}(\ell_{2}-\ell_{1})+4\,\bigg(1+\frac{\overline{v}}{\underline{v}}\bigg)K\,\big\|g^{+}\big\|_{\infty}(s_{2}-s_{1})\\
&\quad +w_{g^{+}}\bigg(\frac{\ell_{2}-\ell_{1}}{\underline{v}}\bigg)+3\,w_{g^{+}}\bigg(\frac{\overline{v}}{\underline{v}}(s_{2}-s_{1})\bigg).
\end{align*}
Therefore, the uniformly continuity of $f_{+}(\cdot,i,\cdot)$ on $\bR^{2}_{+}$ follows from the uniform the uniform continuity of $g^{+}(\cdot,i)$ on $\bR_{+}$, uniformly for all $i\in\mathbf{E}$.

It remains to show that for any $i\in\mathbf{E}$ and $\ell\in\bR_{+}$, $f(s,i,\ell)$ vanishes as $s\rightarrow\infty$. Since $g^{+}\in C_{0}(\sX_{+})$, by \eqref{eq:ProcZ} and \eqref{eq:FuntfPlus}, we have
\begin{align*}
\lim_{s\rightarrow\infty}\big|f(s,i,\ell)\big|=\lim_{s\rightarrow\infty}\Big|\wt{\bE}_{s,i,0}\Big(g^{+}\Big(s+\wt{\tau}_{\ell}^{+},Z^{2}_{\wt{\tau}_{\ell}^{+}}\Big)\Big)\Big|\leq\lim_{s\rightarrow\infty}\sup_{(t,j)\in[s,\infty)\times\mathbf{E}_{+}}\big|g^{+}(t,j)\big|=0.
\end{align*}
The last statement of Lemma \ref{lem:UnifContf} follows directly from \eqref{eq:JPlusfPlus} and \eqref{eq:PellPlusfPlus}.

\medskip
\noindent The proof of Lemma \ref{lem:UnifContf} is complete.

\section{Proofs of lemmas from Section \ref{sec:UniqProof}}\label{sec:ProofLemUniq}

\begin{proof}[Proof of Lemma \ref{lem:DiffQell}.]
This is a direct consequence of \cite[1.2.B \& 1.3.C]{Dynkin1965}.
\end{proof}

\begin{proof}[Proof of Lemma \ref{lem:GenTildeM}.]
Let $h\in C_{0}^{1}(\overline{\sZ})$. For any $(s,i,a)\in\sZ$ and $t\in\bR_{+}$, by \eqref{eq:Probz}, we have
\begin{align}
\frac{1}{t}\Big(\wt{\bE}_{s,i,a}\big(h\big(Z^{1}_{t},Z^{2}_{t},Z^{3}_{t}\big)\big)-h(s,i,a)\Big)&=\frac{1}{t}\,\wt{\bE}_{s,i,a}\Big(h\big(s+t,Z^{2}_{t},Z^{3}_{t}\big)-h\big(s+t,Z^{2}_{t},a\big)\Big)\nonumber\\
&\quad +\frac{1}{t}\Big(\wt{\bE}_{s,i,a}\big(h\big(s+t,Z^{2}_{t},a\big)\big)-h(s+t,i,a)\Big)\nonumber\\
\label{eq:DecompExphZ} &\quad +\frac{1}{t}\big(h(s+t,i,a)-h(s,i,a)\big).
\end{align}
For the first term in \eqref{eq:DecompExphZ}, by \eqref{eq:Probz} and \eqref{eq:DistTildeZ3IntTildeZ2},
\begin{align*}
\frac{1}{t}\,\wt{\bE}_{s,i,a}\Big(h\big(s\!+\!t,Z^{2}_{t},Z^{3}_{t}\big)\!-\!h\big(s\!+\!t,Z^{2}_{t},a\big)\Big)=\frac{1}{t}\,\wt{\bE}_{s,i,a}\bigg(h\bigg(\!s\!+\!t,Z^{2}_{t},a+\!\!\int_{0}^{t}\!v\big(Z^{2}_{r}\big)dr\!\bigg)\!-\!h\big(s\!+\!t,Z^{2}_{t},a\big)\!\bigg).
\end{align*}
Since
\begin{align*}
\frac{1}{t}\bigg|h\bigg(\!s\!+\!t,Z^{2}_{t},a\!+\!\!\!\int_{0}^{t}\!v\big(Z^{2}_{r}\big)dr\!\bigg)\!-\!h\big(s\!+\!t,Z^{2}_{t},a\big)\bigg|\!\leq\!\sup_{\substack{j\in\mathbf{E} \\ b\in[a-\overline{v}t,a+\overline{v}t]}}\!\!\frac{1}{t}\big|h(s\!+\!t,j,b)\!-\!h(s\!+\!t,j,a)\big|\!\leq\!\overline{v}\,\bigg\|\frac{\partial}{\partial a}h\bigg\|_{\infty},
\end{align*}
 then by the dominated convergence theorem we have
\begin{align}
&\lim_{t\rightarrow 0+}\frac{1}{t}\,\wt{\bE}_{s,i,a}\bigg(h\bigg(s+t,Z^{2}_{t},a+\!\int_{0}^{t}v\big(Z^{2}_{r}\big)\,dr\bigg)-h\big(s+t,Z^{2}_{t},a\big)\bigg)\nonumber\\
\label{eq:LimitExphZ1} &\quad =\bE_{s,i,a}\!\left(\lim_{t\rightarrow 0+}\frac{1}{t}\bigg(h\bigg(s+t,Z^{2}_{t},a+\!\int_{0}^{t}v\big(Z^{2}_{r}\big)\,dr\bigg)-h\big(s+t,Z^{2}_{t},a\big)\bigg)\right)=v(i)\frac{\partial h}{\partial a}(s,i,a).
\end{align}
Next, for the second term in \eqref{eq:DecompExphZ}, by \eqref{eq:DistTildeZ3IntTildeZ2}, \eqref{eq:LawXXStar}, \eqref{eq:DefEvolSytXStar}, and \eqref{eq:DefGenXStar}, we deduce that
\begin{align}
&\lim_{t\rightarrow 0+}\frac{1}{t}\Big(\wt{\bE}_{s,i,a}\big(h\big(s+t,Z^{2}_{t},a\big)\big)-h(s+t,i,a)\Big)=\lim_{t\rightarrow 0+}\frac{1}{t}\Big(\bE_{s,(i,a)}\big(h\big(s+t,X_{s+t},a\big)\big)-h(s+t,i,a)\Big)\nonumber\\
&\quad =\frac{1}{t}\Big(\bE^{*}_{s,i}\big(h\big(s+t,X^{*}_{s+t},a\big)\big)-h(s+t,i,a)\Big)=\lim_{t\rightarrow 0+}\frac{1}{t}\big(\big(\mU^{*}_{s,s+t}-\mI\big)h(s+t,\cdot,a)\big)(i)\nonumber\\
&\quad =\lim_{t\rightarrow 0+}\frac{1}{t}\big(\big(\mU^{*}_{s,s+t}-\mI\big)h(s,\cdot,a)\big)(i)-\lim_{t\rightarrow 0+}\big(\big(\mU^{*}_{s,s+t}-\mI\big)\big(h(s+t,\cdot,a)-h(s,\cdot,a)\big)\big)(i)\nonumber\\
\label{eq:LimitExphZ2} &\quad =\big(\mathsf{\Lambda}_{s}\,h(s,\cdot,a)\big)(i).
\end{align}
Combining \eqref{eq:DecompExphZ}$-$\eqref{eq:LimitExphZ2} leads to
\begin{align*}
&\lim_{t\rightarrow 0+}\frac{1}{t}\Big(\wt{\bE}_{s,i,a}\big(h\big(Z^{1}_{t},Z^{2}_{t},Z^{3}_{t}\big)\big)-h(s,i,a)\Big)=v(i)\frac{\partial}{\partial a}h(s,i,a)+\big(\mathsf{\Lambda}_{s}h(s,\cdot,a)\big)(i)+\frac{\partial h}{\partial s}(s,i,a).
\end{align*}
Since the semigroup induced by $\wt{\cM}$ is Feller, by \cite[Theorem 1.33]{BottcherSchillingWang2013}, the above pointwise limit is uniform for all $(s,i,a)\in\sZ$ and $h\in\sD(\cA)$, which completes the proof of the lemma.
\end{proof}

To proceed with the proof of Lemma \ref{lem:QellCompgPlus}, we first prove the following auxiliary result.

\begin{lemma}\label{lem:ExistUniPhiIntEq}
For any $\lambda\in\bR_{+}$ and $h^{+}\in C_{0}(\overline{\sX_{+}})$ with $\supp h^{+}\subset[0,\eta_{h^{+}}]\times\mathbf{E}_{+}$, for some $\eta_{h^{+}}\in(0,\infty)$, there exists a unique solution $\Phi\in C_{0}(\overline{\sX_{+}})$ to
\begin{align}\label{eq:PhiIntEq}
\left\{\begin{array}{l}
\displaystyle{\Phi(s,i)=\int_{s}^{\eta_{h^{+}}}\Big(\big(\big(\wt{\mA}-\lambda\mV^{+}+\wt{\mB}S^{+}\big)\Phi\big)(t,i)+\big(\mV^{+}h^{+}\big)(t,i)\Big)dt,\quad (s,i)\in[0,\eta_{h^{+}})\times\mathbf{E}_{+}}, \\
\Phi(s,i)=0,\quad (s,i)\in\big([\eta_{h^{+}},\infty)\times\mathbf{E}_{+}\big)\cup\{(\infty,\partial)\},
\end{array}
\right.
\end{align}
and furthermore, $\Phi\in C_{c}^{1}(\overline{\sX_{+}})$.

Moreover, for any $\lambda\in\bR_{+}$, $\Phi\in C_{0}^{1}(\overline{\sX_{+}})$ is a solution to \eqref{eq:PhiIntEq} if and only if $\Phi\in C_{0}^{1}(\overline{\sX_{+}})$ solves
\begin{align}\label{eq:LambdaHPlusEq}
\big(\lambda-H^{+}\big)\Phi=h^{+}
\end{align}
subject to $\Phi=0$ on $([\eta_{h^{+}},\infty)\times\mathbf{E}_{+})\cup\{(\infty,\partial)\}$. Consequently, there exists a unique solution $\Phi\in C_{0}^{1}(\overline{\sX_{+}})$ to \eqref{eq:LambdaHPlusEq} subject to $\Phi=0$ on $([\eta_{h^{+}},\infty)\times\mathbf{E}_{+})\cup\{(\infty,\partial)\}$.
\end{lemma}

\begin{proof}
In view of \eqref{eq:WHPlus1}, we have that \eqref{eq:LambdaHPlusEq} is equivalent to the following equation
\begin{align}\label{eq:PhiDiffEq}
\bigg(\lambda-(\mV^{+})^{-1}\bigg(\frac{\partial}{\partial s}+\wt{\mA}+\wt{\mB}S^{+}\bigg)\bigg)\Phi=h^{+}.
\end{align}
We first show that $\Phi\in C_{0}^{1}(\overline{\sX_{+}})$ is a solution to \eqref{eq:PhiIntEq} if and only if $\Phi$ solves \eqref{eq:PhiDiffEq} with $\Phi=0$ on $([\eta_{h^{+}},\infty)\times\mathbf{E}_{+})\cup\{(\infty,\partial)\}$. On the one hand, if $\Phi\in C_{0}^{1}(\overline{\sX_{+}})$ is a solution to \eqref{eq:PhiIntEq}, by differentiating the first equality in \eqref{eq:PhiIntEq} and rearranging terms, we obtain \eqref{eq:PhiDiffEq} on $[0,\eta_{h^{+}})\times\mathbf{E}_{+}$, while \eqref{eq:PhiDiffEq} holds trivially on $([\eta_{h^{+}},\infty)\times\mathbf{E}_{+})\cup\{(\infty,\partial)\}$ since both sides are equal to 0. On the other hand, if $\Phi\in C_{0}^{1}(\overline{\sX_{+}})$ solves \eqref{eq:PhiDiffEq} with $\Phi=0$ on $([\eta_{h^{+}},\infty)\times\mathbf{E}_{+})\cup\{(\infty,\partial)\}$, the first equality in \eqref{eq:PhiIntEq} follows by rearranging terms in \eqref{eq:PhiDiffEq} and then  integrating  both sides. Therefore, we only need to show that \eqref{eq:PhiIntEq} has a unique solution $\Phi\in C_{0}(\overline{\sX_{+}})$ which, in particular, belongs to $C_{c}^{1}(\overline{\sX_{+}})$. The proof will be done in the following three steps.

\medskip
\noindent
\textbf{Step 1.} For any $T>0$, let $\sX_{T}^{\pm}:=[T,\infty)\times\mathbf{E}_{\pm}$ and $\overline{\sX_{T}^{\pm}}:=\sX_{T}^{\pm}\cup\{(\infty,\partial)\}$. We define the spaces of functions $C_{0}(\overline{\sX_{T}^{\pm}})$ (respectively, $C_{c}(\overline{\sX_{T}^{\pm}})$) in analogy to $C_{0}(\overline{\sX_{\pm}})$ (respectively, $C_{c}(\overline{\sX_{\pm}})$), with the domain of functions restricted to $\overline{\sX_{T}^{\pm}}$. Clearly, any function in $C_{0}(\overline{\sX_{T}^{\pm}})$ can be regarded as the restriction of some function in $C_{0}(\overline{\sX_{\pm}})$ on $\overline{\sX_{T}^{\pm}}$. The goal of this step is to seek for an operator $S_{T}^{+}:C_{0}(\overline{\sX_{T}^{+}})\rightarrow C_{0}(\overline{\sX_{T}^{-}})$, for any $T\in(0,\infty)$, such that for each $g^{+}\in C_{0}(\overline{\sX_{+}})$, $S^{+}g^{+}=S^{+}_{T}(g^{+}|_{\overline{\sX_{T}^{+}}})$ on $\overline{\sX_{T}^{-}}$.

Note that, for any $(s,i)\in\overline{\sX_{-}}$, we define the linear functional $\Theta_{s,i}:C_{0}(\overline{\sX_{+}})\rightarrow\bR$ by
\begin{align*}
\Theta_{s,i}\,g^{+}:=\big(S^{+}g^{+}\big)(s,i),\quad g^{+}\in C_{0}(\overline{\sX_{+}}).
\end{align*}
Since $S^{+}:C_{0}(\overline{\sX_{+}})\rightarrow C_{0}(\overline{\sX_{-}})$ is bounded, $\Theta_{s,i}$ is a bounded linear functional on $C_{0}(\overline{\sX_{+}})$. Hence, by Riesz representation theorem (cf. \cite[Theorem 6.19]{Rudin1987}), there is a unique signed measure $\mu_{s,i}$ on $(\overline{\sX_{+}},\cB(\overline{\sX_{+}}))$, such that
\begin{align}\label{eq:RieszSPlus}
\big(S^{+}g^{+}\big)(s,i)=\Theta_{s,i}\,g^{+}=\int_{\overline{\sX_{+}}}g^{+}(t,j)\,d\mu_{s,i}(t,j),\quad g^{+}\in C_{0}(\overline{\sX_{+}}).
\end{align}
We claim that, for any fixed $(s,i)\in\overline{\sX_{-}}$,
\begin{align}\label{eq:musiVanish}
|\mu_{s,i}|([0,s)\times\mathbf{E}_{+})=0.
\end{align}
Otherwise, there exists $[a,b]\subset [0,s)$ and $j_{0}\in\mathbf{E}_{+}$, such that $|\mu_{s,i}|([a,b]\times\{j_{0}\})>0$. Let $\mu_{s,i}^{+}$ and $\mu_{s,i}^{-}$ be the positive and {the} negative variations of $\mu_{s,i}$, and let $\overline{\sX_{+}}=D_{s,i}\cup D_{s,i}^{c}$ (where $D_{s,i}\in\cB(\overline{\sX_{+}})$) be the Hahn decomposition (cf. \cite[Theorem 6.14]{Rudin1987}) with respect to $\mu_{s,i}$, so that
\begin{align*}
\mu_{s,i}^{+}(A)=\mu_{s,i}\big(A\cap D_{s,i}\big),\quad\mu_{s,i}^{-}(A)=\mu_{s,i}\big(A\cap D_{s,i}^{c}\big),\quad\text{for any }\,A\in\cB(\overline{\sX_{+}}).
\end{align*}
Without loss of generality, we can assume that there exists $[a',b']\subset[a,b]$ such that $[a',b']\times\{j_{0}\}\subset([a,b]\times\{j_{0}\})\cap D_{s,i}$ and that $\mu_{s,i}^{+}([a',b']\times\{j_{0}\})>0$. Now we can construct $\rho\in C_{0}(\bR_{+})$ such that $\supp\rho=[a',b']$, and let $\tilde{g}^{+}(t,k)=\rho(t)\1_{\{j_{0}\}}(k)$, $(s,i)\in\overline{\sX_{+}}$. Clearly, $\tilde{g}^{+}\in C_{0}(\overline{\sX_{+}})$ with $\supp\tilde{g}^{+}\subset[a',b']\times\mathbf{E}_{+}\subset[0,s)\times\mathbf{E}_{+}$. Hence, by the condition $(a^{+})(\text{i})$, $(S^{+}\tilde{g}^{+})(s,i)=0$. However,
\begin{align*}
\int_{\overline{\sX_{+}}}g^{+}(t,j)\,d\mu_{s,i}(t,j)=\int_{[a',b']\times\{j_{0}\}}\rho(t)\,d\mu_{s,i}(t,j)>0,
\end{align*}
which contradicts \eqref{eq:RieszSPlus}. Therefore, \eqref{eq:musiVanish} holds true.

Next, for any $T\in(0,\infty)$, we define $S^{+}_{T}$ on $C_{0}(\overline{\sX_{T}^{+}})$ by
\begin{align*}
\big(S^{+}_{T}g^{+}_{T}\big)(s,i):=\int_{\overline{\sX_{T}^{+}}}\,g^{+}_{T}(t,j)\,d\mu_{s,i}(t,j),\quad (s,i)\in\overline{\sX_{T}^{-}},\quad g^{+}_{T}\in C_{0}(\overline{\sX_{T}^{+}}).
\end{align*}
By \eqref{eq:musiVanish}, for any $g^{+}_{T}\in C_{0}(\overline{\sX_{T}^{+}})$ and $g^{+}\in C_{0}(\overline{\sX_{+}})$ such that $g^{+}_{T}=g^{+}|_{\overline{\sX_{T}^{+}}}$,
\begin{align}
\big(S^{+}_{T}g^{+}_{T}\big)(s,i)&=\int_{\overline{\sX_{T}^{+}}}\,g^{+}_{T}(t,j)\,d\mu_{s,i}(t,j)=\int_{\overline{\sX_{T}^{+}}}\,g^{+}_{T}(t,j)\,d\mu_{s,i}(t,j)+\int_{[0,T)\times\mathbf{E}_{+}}\,g^{+}(t,j)\,d\mu_{s,i}(t,j)\nonumber\\
\label{eq:STPlusSPlus} &=\int_{\overline{\sX_{+}}}\,g^{+}(t,j)\,d\mu_{s,i}(t,j)=\big(S^{+}g^{+}\big)(s,i),\quad (s,i)\in\overline{\sX_{T}^{-}}.
\end{align}
In particular, since $S^{+}$ is a bounded operator on $C_{0}(\overline{\sX_{+}})$, we have
\begin{align}\label{eq:ContSTPlus}
S^{+}_{T}g^{+}_{T}=\big(S^{+}g^{+}\big)\big|_{\overline{\sX_{T}^{-}}}\in C_{0}\big(\overline{\sX_{T}^{-}}\big),
\end{align}
and
\begin{align}\label{eq:STPlusSPlusNorm}
\big\|S^{+}_{T}\big\|_{\infty}\!=\!\!\sup_{\|g^{+}_{T}\|_{\infty}=1}\!\big\|S^{+}_{T}g^{+}_{T}\big\|_{\infty}\!\leq\!\sup_{\|g^{+}\|_{\infty}=1}\!\sup_{(s,i)\in\overline{\sX_{T}^{-}}}\!\big|\big(S^{+}g^{+}\big)(s,i)\big|\!\leq\!\sup_{\|g^{+}\|_{\infty}=1}\!\big\|S^{+}g^{+}\big\|_{\infty}\!\!=\!\|S^{+}\|_{\infty}.
\end{align}
Moreover, for any $0<T_{1}\leq T_{2}<\infty$, $g^{+}\in C_{0}(\overline{\sX_{T}^{+}})$, $g^{+}_{T_{1}}\in C_{0}(\overline{\sX_{T_{1}}^{+}})$, and $g^{+}_{T_{2}}\in C_{0}(\overline{\sX_{T_{2}}^{+}})$, such that $g^{+}=g^{+}_{T_{1}}=g^{+}_{T_{2}}$ on $\overline{\sX_{T_{2}}^{+}}$, a similar argument using \eqref{eq:musiVanish} shows that
\begin{align}\label{eq:SPlusT1T2}
\big(S^{+}_{T_{2}}g^{+}_{T_{2}}\big)(u,i)=\big(S^{+}_{T_{1}}g^{+}_{T_{1}}\big)(u,i)=\big(S^{+}g^{+}\big)(u,i),\quad (u,i)\in\overline{\sX^{-}_{T_{2}}}.
\end{align}

\noindent
\textbf{Step 2.} We now verify the existence and uniqueness of the solution to \eqref{eq:PhiIntEq} by proceeding backwards starting from $\eta_{h^{+}}$. Fix any $\lambda\in\bR_{+}$ and $h^{+}\in C_{0}(\overline{\sX_{+}})$, and pick $\delta_{0}\in(0,\eta_{h^{+}})$ small enough so that
\begin{align}\label{eq:ContractionConst}
\big(\big\|\wt{\mA}\big\|_{\infty}+\lambda\|\mV^{+}\|_{\infty}+\big\|\wt{\mB}\big\|_{\infty}\|S^{+}\|_{\infty}\big)\delta_{0}<1.
\end{align}
Using \eqref{eq:ContSTPlus}, we define $\Gamma_{1}:C_{0}(\overline{\sX^{+}_{\eta_{1}}})\rightarrow C_{0}(\overline{\sX^{+}_{\eta_{1}}})$, where $\eta_{1}:=\eta_{h^{+}}-\delta_{0}$, by
\begin{align*}
(\Gamma_{1}\phi)(s,i):=\int_{s}^{\eta_{h^{+}}}\!\!\Big(\big(\big(\wt{\mA}-\lambda\mV^{+}+\wt{\mB}S^{+}_{\eta_{1}}\big)\phi\big)(t,i)+(\mV^{+}h^{+})(t,i)\Big)\,dt,\quad (s,i)\in\big[\eta_{1},\eta_{h^{+}}\big)\times\mathbf{E}_{+},
\end{align*}
and $(\Gamma_{1}\phi)(s,i)=0$ for $(s,i)\in\overline{\sX^{+}_{\eta_{h^{+}}}}$, where $\phi\in C_{0}(\overline{\sX^{+}_{\eta_{1}}})$. By Assumption \ref{assump:GenLambda} and \eqref{eq:ContSTPlus}, the integral on the right-hand side above is well defined since the integrand is continuous in $t$ and bounded. In this step, we will show that $\Gamma_{1}$ has a unique fixed point $\Phi^{(1)}\in C_{0}(\overline{\sX^{+}_{\eta_{1}}})$, so that
\begin{align}\label{eq:Phi1eta1}
\Phi^{(1)}(s,i)=\!\int_{s}^{\eta_{h^{+}}}\!\!\Big(\Big(\big(\wt{\mA}-\lambda\mV^{+}+\wt{\mB}S^{+}_{\eta_{1}}\big)\Phi^{(1)}\Big)(t,i)+(\mV^{+}h^{+})(t,i)\Big)\,dt,\,\,\,(s,i)\!\in\!\big[\eta_{1},\eta_{h^{+}}\big)\!\times\!\mathbf{E}_{+},
\end{align}
and
\begin{align}\label{eq:Phi1eta1BoundVal}
\Phi^{(1)}(s,i)=0,\quad (s,i)\in\overline{\sX_{\eta_{h^{+}}}^{+}}.
\end{align}
Moreover, we will show that $\Phi^{(1)}\in C_{0}^{1}(\overline{\sX^{+}_{\eta_{1}}})$. It is then clear from \eqref{eq:Phi1eta1BoundVal} that $\supp\Phi^{(1)}\subset[\eta_{1},\eta_{h^{+}}]\times\mathbf{E}_{+}$ so that $\Phi^{(1)}\in C_{c}^{1}(\overline{\sX^{+}_{\eta_{1}}})$. Therefore, the result in this step implies that \eqref{eq:PhiIntEq} has a unique solution $\Phi^{(1)}$ on the restricted domain $C_{c}(\overline{\sX^{+}_{\eta_{1}}})$, and $\Phi^{(1)}\in C_{c}^{1}(\overline{\sX^{+}_{\eta_{1}}})$.

We first show that $\Gamma_{1}$ has a unique fixed point in $C_{0}(\overline{\sX^{+}_{\eta_{1}}})$. For any $\phi_{1},\phi_{2}\in C_{0}(\overline{\sX^{+}_{\eta_{1}}})$, by Assumption \ref{assump:GenLambda} and \eqref{eq:STPlusSPlusNorm},
\begin{align*}
\big\|\Gamma_{1}\phi_{1}-\Gamma_{1}\phi_{2}\big\|_{\infty}&=\sup_{(s,i)\in[\eta_{1},\eta_{h^{+}})\times\mathbf{E}_{+}}\bigg|\int_{s}^{\eta_{h^{+}}}\Big(\big(\wt{\mA}-\lambda\mV^{+}+\wt{\mB}S^{+}_{\eta_{1}}\big)(\phi_{1}-\phi_{2})\Big)(t,i)\,dt\bigg|\\
&\leq\big(\big\|\wt{\mA}\big\|_{\infty}+\lambda\|\mV^{+}\|_{\infty}+\big\|\wt{\mB}\big\|_{\infty}\|S^{+}\|_{\infty}\big)\delta_{0}\,\|\phi_{1}-\phi_{2}\|_{\infty}.
\end{align*}
Hence, by \eqref{eq:ContractionConst}, $\Gamma_{1}$ is a contraction mapping on $C_{0}(\overline{\sX^{+}_{\eta_{1}}})$, and thus $\Gamma_{1}$ has a unique fixed point $\Phi^{(1)}\in C_{0}(\overline{\sX^{+}_{\eta_{1}}})$. In particular, \eqref{eq:Phi1eta1} and \eqref{eq:Phi1eta1BoundVal} hold true.

Next, we show that $\Phi^{(1)}\in C_{0}^{1}(\overline{\sX^{+}_{\eta_{1}}})$. By \eqref{eq:Phi1eta1BoundVal}, it is clear that $\Phi^{(1)}|_{\overline{\sX^{+}_{\eta_{h^{+}}}}}\in C_{0}^{1}(\overline{\sX^{+}_{\eta_{h^{+}}}})$. Also, since the integrand on the right-hand side of \eqref{eq:Phi1eta1} is continuous in $t$ and bounded, we obtain from \eqref{eq:Phi1eta1} that $\Phi^{(1)}(\cdot,i)$ is continuous on $[\eta_{1},\eta_{h^{+}})$, and that $\lim_{s\rightarrow\eta_{h^{+}}-}\Phi^{(1)}(s,i)=0=\Phi^{(1)}(\eta_{h^{+}},i)$, for any $i\in\mathbf{E}_{+}$. It remains to show that $\Phi^{(1)}(\cdot,i)$ is continuously differentiable at $s=\eta_{h^{+}}$ for any $i\in\mathbf{E}_{+}$, and, due to \eqref{eq:Phi1eta1BoundVal}, we only need to prove that $\partial_{-}\Phi^{(1)}(\eta_{h^{+}},i)/\partial s$ exists and equals to $0$, and that $\partial\Phi^{(1)}(\cdot,i)/\partial s$ is continuous at $s=\eta_{h^{+}}$.

We fix an $i\in\mathbf{E}_{+}$. Since $\supp\Phi^{(1)}\subset[\eta_{1},\eta_{h^{+}}]\times\mathbf{E}_{+}$, by \eqref{eq:SPlusT1T2} and the condition $(a^{+})(\text{i})$, for any $g^{+}\in C_{0}(\overline{\sX_{+}})$ such that $g^{+}|_{\overline{\sX^{+}_{\eta_{1}}}}=\Phi^{(1)}$, $(S^{+}_{\eta_{1}}\Phi^{(1)})(s,i)=(S^{+}g^{+})(s,j)=0$ for all $(s,j)\in(\eta_{h^{+}},\infty)\times\mathbf{E}_{-}$. It follows from the continuity of $S^{+}_{\eta_{1}}\Phi^{(1)}$ that
\begin{align*}
\big(S^{+}_{\eta_{1}}\Phi^{(1)}\big)(\eta_{h^{+}},j)=\lim_{s\rightarrow\eta_{h^{+}}-}\big(S^{+}_{\eta_{1}}\Phi^{(1)}\big)(s,j)=0,\quad j\in\mathbf{E}_{-}.
\end{align*}
This, together with \eqref{eq:Phi1eta1BoundVal}, implies that
\begin{align*}
&\Big(\big(\wt{\mA}-\lambda\mV^{+}+\wt{\mB}S^{+}_{\eta_{1}}\big)\Phi^{(1)}+\mV^{+}h^{+}\Big)(\eta_{h^{+}},i)\\
&\quad =\!\sum_{j\in\mathbf{E}_{+}}\!\!\big(A_{\eta_{h^{+}}}\!(i,j)-\lambda v(j)\big)\Phi^{(1)}(\eta_{h^{+}},j)+\!\!\sum_{j\in\mathbf{E}_{-}}\!\!B_{\eta_{h^{+}}}\!(i,j)\big(S^{+}_{\eta_{1}}\Phi^{(1)}\big)(\eta_{h^{+}},j)+\!\!\sum_{j\in\mathbf{E}_{+}}\!\!v(j)h^{+}(\eta_{h^{+}},j)=0.
\end{align*}
Therefore, by \eqref{eq:Phi1eta1}, \eqref{eq:Phi1eta1BoundVal}, and the fact that the integrand in \eqref{eq:Phi1eta1} is continuous in $t$, we have
\begin{align*}
\frac{\partial_{-}\Phi^{(1)}}{\partial s}(\eta_{h^{+}},i)&=\lim_{s\rightarrow\eta_{h^{+}}-}\frac{\Phi^{(1)}(s,i)-\Phi^{(1)}(\eta_{h^{+}},i)}{s-\eta_{h^{+}}}\\
&=\lim_{s\rightarrow\eta_{h^{+}}-}\frac{1}{s-\eta_{h^{+}}}\int_{s}^{\eta_{h^{+}}}\Big(\Big(\big(\wt{\mA}-\lambda\mV^{+}+\wt{\mB}S^{+}_{\eta_{1}}\big)\Phi^{(1)}\Big)(t,i)+(\mV^{+}h^{+})(t,i)\Big)dt\\
&=\Big(\big(\wt{\mA}-\lambda\mV^{+}+\wt{\mB}S^{+}_{\eta_{1}}\big)\Phi^{(1)}+\mV^{+}h^{+}\Big)(\eta_{h^{+}},i)=0=\frac{\partial_{+}\Phi^{(1)}}{\partial s}(\eta_{h^{+}},i).
\end{align*}
and
\begin{align*}
\lim_{s\rightarrow\eta_{h^{+}}-}\frac{\partial\Phi^{(1)}}{\partial s}(s,i)&=\lim_{s\rightarrow\eta_{h^{+}}-}\Big(\big(\wt{\mA}-\lambda\mV^{+}+\wt{\mB}S^{+}_{\eta_{1}}\big)\Phi^{(1)}+\mV^{+}h^{+}\Big)(s,i)\\
&=\Big(\big(\wt{\mA}-\lambda\mV^{+}+\wt{\mB}S^{+}_{\eta_{1}}\big)\Phi^{(1)}+\mV^{+}h^{+}\Big)(\eta_{h^{+}},i)=0=\lim_{s\rightarrow\eta_{h^{+}}+}\frac{\partial\Phi^{(1)}}{\partial s}(s,i),
\end{align*}
which completes the proof in step 2.

\medskip
\noindent
\textbf{Step 3.}  
In this step, we will extend the result in Step 2 and construct the unique solution to \eqref{eq:PhiIntEq} restricted to $\overline{\sX^{+}_{\eta_{2}}}$, where $\eta_{2}:=(\eta_{h^{+}}-2\delta_{0})\vee 0$. Without loss of generality, we take  $\delta_{0}\in(0,\eta_{h^{+}}/2)$.

In view of \eqref{eq:ContSTPlus}, we define $\Gamma_{2}:C_{0}(\overline{\sX^{+}_{\eta_{2}}})\rightarrow C_{0}(\overline{\sX^{+}_{\eta_{2}}})$ by
\begin{align*}
\big(\Gamma_{2}\phi)(s,i):=\Phi^{(1)}(\eta_{1},i)+\!\int_{s}^{\eta_{1}}\!\Big(\!\big(\big(\wt{\mA}\!-\!\lambda\mV^{+}\!+\!\wt{\mB}S^{+}_{\eta_{2}}\big)\phi\big)(t,i)\!+\!(\mV^{+}h^{+}\big)(t,i)\Big)dt,\,\,\,(s,i)\in [\eta_{2},\eta_{1})\times\mathbf{E}_{+},
\end{align*}
and $(\Gamma_{2}\phi)(s,i)=\Phi^{(1)}(s,i)$ for $(s,i)\in\overline{\sX_{\eta_{1}}^{+}}$. By Assumption \ref{assump:GenLambda} and \eqref{eq:ContSTPlus}, the integral on the right-hand side above is well defined since the integrand is continuous in $t$ and bounded. For any $\phi_{1},\phi_{2}\in C_{0}(\overline{\sX^{+}_{\eta_{2}}})$, by Assumption \ref{assump:GenLambda} and \eqref{eq:STPlusSPlusNorm},
\begin{align*}
\big\|\Gamma_{2}\phi_{1}-\Gamma_{2}\phi_{2}\big\|_{\infty}&=\sup_{(s,i)\in[\eta_{2},\eta_{1})\times\mathbf{E}_{+}}\bigg|\int_{s}^{\eta_{1}}\Big(\big(\wt{\mA}-\lambda\mV^{+}+\wt{\mB}S^{+}_{\eta_{2}}\big)(\phi_{1}-\phi_{2})\Big)(t,i)\,dt\bigg|\\
&\leq\big(\big\|\wt{\mA}\big\|_{\infty}+\lambda\|\mV^{+}\|_{\infty}+\big\|\wt{\mB}\big\|_{\infty}\|S^{+}\|_{\infty}\big)\delta_{0}\,\|\phi_{1}-\phi_{2}\|_{\infty}.
\end{align*}
Hence, by \eqref{eq:ContractionConst}, $\Gamma_{2}$ is contraction mapping on $C_{0}(\overline{\sX^{+}_{\eta_{2}}})$, and hence
$\Gamma_{2}$ has a unique fixed point $\Phi^{(2)}\in C_{0}(\overline{\sX^{+}_{\eta_{2}}})$. Therefore, for any $(s,i)\in [\eta_{2},\eta_{1})\times\mathbf{E}_{+}$,
\begin{align}\label{eq:Phi2eta2}
\Phi^{(2)}(s,i)=\Phi^{(1)}(\eta_{1},i)-\int_{s}^{\eta_{1}}\Big(\big(\big(\wt{\mA}-\lambda\mV^{+}+\wt{\mB}S^{+}_{\eta_{2}}\big)\Phi^{(2)}\big)(t,i)+(\mV^{+}h^{+})(t,i)\Big)\,dt,
\end{align}
and
\begin{align}\label{eq:Phi2eta2BoundVal}
\Phi^{(2)}(s,i)=\Phi^{(1)}(s,i),\quad (s,i)\in\overline{\sX_{\eta_{1}}^{+}}.
\end{align}

Moreover, we will show that $\Phi^{(2)}\in C_{0}^{1}(\overline{\sX^{+}_{\eta_{2}}})$. By Step 2 and \eqref{eq:Phi2eta2BoundVal}, it is clear that $\Phi^{(2)}|_{\overline{\sX_{\eta_{1}}^{+}}}\in C_{0}^{1}(\overline{\sX^{+}_{\eta_{1}}})$. Also, since the integrand on the right-hand side of \eqref{eq:Phi2eta2} is continuous in $t$ and bounded, we obtain from \eqref{eq:Phi2eta2} that $\Phi^{(2)}(\cdot,i)$ is continuous on $[\eta_{2},\eta_{1})$, and that $\lim_{s\rightarrow\eta_{1}-}\Phi^{(2)}(s,i)=\Phi^{(1)}(\eta_{1},i)=\Phi^{(2)}(\eta_{1},i)$, for any $i\in\mathbf{E}_{+}$. It remains to show that $\Phi^{(2)}(\cdot,i)$ is continuously differentiable at $s=\eta_{1}$ for any $i\in\mathbf{E}_{+}$, and, in light of \eqref{eq:Phi2eta2BoundVal}, we only need to show that $\partial_{-}\Phi^{(2)}(\eta_{1},i)/\partial s$ exists and equals to $\partial_{+}\Phi^{(1)}(\eta_{1},i)/\partial s$, and that $\partial\Phi^{(2)}(\cdot,i)/\partial s$ is continuous at $s=\eta_{1}$.

Let $i\in\mathbf{E}_{+}$. By \eqref{eq:SPlusT1T2} and \eqref{eq:Phi1eta1},
\begin{align*}
\Big(\big(\wt{\mA}-\lambda\mV^{+}\!+\!\wt{\mB}S^{+}_{\eta_{2}}\big)\Phi^{(2)}\!+\!\mV^{+}h^{+}\Big)(\eta_{1},i)=\Big(\big(\wt{\mA}-\lambda\mV^{+}\!+\!\wt{\mB}S^{+}_{\eta_{1}}\big)\Phi^{(1)}\!+\!\mV^{+}h^{+}\Big)(\eta_{1},i)=\frac{\partial_{+}\Phi^{(1)}}{\partial s}(\eta_{1},i).
\end{align*}
Therefore, by \eqref{eq:Phi2eta2}, \eqref{eq:Phi2eta2BoundVal}, and the fact that the integrand in \eqref{eq:Phi2eta2} is continuous in $t$, we have
\begin{align*}
\frac{\partial_{-}\Phi^{(2)}}{\partial s}(\eta_{1},i)&=\lim_{s\rightarrow\eta_{1}-}\frac{\Phi^{(2)}(s,i)-\Phi^{(2)}(\eta_{1},i)}{s-\eta_{1}}\\
&=\lim_{s\rightarrow\eta_{1}-}\frac{1}{s-\eta_{1}}\int_{s}^{\eta_{1}}\Big(\Big(\big(\wt{\mA}-\lambda\mV^{+}+\wt{\mB}S^{+}_{\eta_{2}}\big)\Phi^{(2)}\Big)(t,i)+(\mV^{+}h^{+})(t,i)\Big)dt\\
&=\Big(\big(\wt{\mA}-\lambda\mV^{+}+\wt{\mB}S^{+}_{\eta_{2}}\big)\Phi^{(2)}\Big)(\eta_{1},i)+(\mV^{+}h^{+})(\eta_{1},i)=\frac{\partial_{+}\Phi^{(1)}}{\partial s}(\eta_{1},i)=\frac{\partial_{+}\Phi^{(2)}}{\partial s}(\eta_{1},i),
\end{align*}
and
\begin{align*}
\lim_{s\rightarrow\eta_{1}-}\frac{\partial\Phi^{(2)}}{\partial s}(s,i)&= \lim_{s\rightarrow\eta_{1}-}\Big(\big(\wt{\mA}-\lambda\mV^{+}+\wt{\mB}S^{+}_{\eta_{2}}\big)\Phi^{(2)}+(\mV^{+}h^{+})\Big)(s,i)\\
&=\Big(\big(\wt{\mA}-\lambda\mV^{+}+\wt{\mB}S^{+}_{\eta_{2}}\big)\Phi^{(2)}+(\mV^{+}h^{+})\Big)(\eta_{1},i)\\
&=\frac{\partial_{+}\Phi^{(1)}}{\partial s}(\eta_{1},\cdot)=\lim_{s\rightarrow\eta_{1}+}\frac{\partial\Phi^{(1)}}{\partial s}(s,i)=\lim_{s\rightarrow\eta_{1}+}\frac{\partial\Phi^{(2)}}{\partial s}(s,i),
\end{align*}
where we have used the fact that $\Phi^{(1)}\in C_{0}^{1}(\overline{\sX^{+}_{\eta_{1}}})$ in the penultimate equality.

In conclusion, we have shown that there exists a unique solution $\Phi^{(2)}\in C_{c}(\overline{\sX^{+}_{\eta_{2}}})$ to \eqref{eq:PhiIntEq} when restricted to $\overline{\sX^{+}_{\eta_{2}}}$ and $\Phi^{(2)}\in C_{c}^{1}(\overline{\sX^{+}_{\eta_{2}}})$. By induction, we obtain a unique solution $\Phi\in C_{c}(\overline{\sX_{+}})$ to \eqref{eq:PhiIntEq} and $\Phi\in C_{c}^{1}(\overline{\sX_{+}})$, which completes the proof of the lemma.
\end{proof}

\begin{proof}[Proof of Lemma \ref{lem:QellCompgPlus}.]
Let $g^{+}\in C_{c}(\overline{\sX_+})$ with $\supp g^{+}\subset[0,\eta_{g^{+}}]\times\mathbf{E}_{+}$ for some $\eta_{g^{+}}\in(0,\infty)$. For any $\lambda\in\bR_{+}$, define $R_{\lambda}$ on $C_{0}(\overline{\sX_{+}})$ by
\begin{align*}
R_{\lambda}h^{+}:=\int_{0}^{\infty}e^{-\lambda\ell}\cQ^{+}_{\ell}h^{+}\,d\ell,\quad h^{+}\in C_{0}(\overline{\sX_{+}}).
\end{align*}
The integral on the right-hand side above is well defined since $\cQ_{\ell}^{+}$ is a contraction mapping, for any $\ell\in\bR_{+}$. In order to prove that $\supp\cQ_{\ell}g^{+}\subset[0,\eta_{g^{+}}]\times\mathbf{E}_{+}$, for any $\ell\in\bR_{+}$, it is sufficient to show that $\supp R_{\lambda}g^{+}\subset[0,\eta_{g^{+}}]\times\mathbf{E}_{+}$, for any $\lambda\in(0,\infty)$. Indeed, if the later is true, then for any $(s,i)\in[\eta_{g^{+}},\infty)\times\mathbf{E}_{+}$,
\begin{align*}
\int_{0}^{\infty}e^{-\lambda\ell}\big(\cQ^{+}_{\ell}g^{+}\big)(s,i)\,d\ell=0,\quad\text{for all }\,\lambda\in(0,\infty),
\end{align*}
which implies that (cf. \cite[Lemma 1.1]{Dynkin1965}), $(\cQ^{+}_{\ell}g^{+})(s,i)=0$ for almost every $\ell\in\bR_{+}$. Since $(\cQ^{+}_{\cdot}g^{+})(s,i)$ is continuous on $\bR_{+}$, we have $(\cQ^{+}_{\ell}g^{+})(s,i)=0$ for all $\ell\in\bR_{+}$.

By \cite[Proposition I.2.1]{EthierKurtz2005}), for any $\lambda\in(0,\infty)$, the operator $(\lambda-H^{+}):C_{0}^{1}(\overline{\sX_{+}})\rightarrow C_{0}(\overline{\sX_{+}})$ is invertible and $(\lambda-H^{+})^{-1}=R_{\lambda}$ (so that $R_{\lambda}$ is the resolvent at $\lambda$ of $H^{+}$). Hence, the equation \eqref{eq:LambdaHPlusEq} has a unique solution $\Phi_{\lambda}=R_{\lambda}g^{+}=(\lambda-H^{+})^{-1}g^{+}\in C_{0}^{1}(\overline{\sX_{+}})$. On the other hand, by Lemma \ref{lem:ExistUniPhiIntEq}, \eqref{eq:LambdaHPlusEq} (with $h^{+}$ replaced by $g^{+}$) has a unique solution in $C_{0}^{1}(\overline{\sX_{+}})$ which vanishes in $[\eta_{g^{+}},\infty)\times\mathbf{E}_{+}$. Therefore, $\supp R_{\lambda}g^{+}\subset[0,\eta_{g^{+}}]\times\mathbf{E}_{+}$, which completes the proof of the lemma.
\end{proof}

The proof of Lemma \ref{lem:FhatPlusC01} requires the following additional lemma.

\begin{lemma}\label{lem:QellVanish}
For any $g^{+}\in C_{c}(\overline{\sX_{+}})$ with $\supp g^{+}\subset[0,\eta_{g^{+}}]\times\mathbf{E}_{+}$ for some $\eta_{g^{+}}\in(0,\infty)$, $\lim_{\ell\rightarrow\infty}\|\cQ^{+}_{\ell}g^{+}\|_{\infty}=0$.
\end{lemma}

\begin{proof}
By Lemma \ref{lem:ExistUniPhiIntEq}, when $\lambda=0$, \eqref{eq:LambdaHPlusEq} (or equivalently, \eqref{eq:PhiIntEq}) has a unique solution $\Phi_{0}\in C_{0}^{1}(\overline{\sX_{+}})$ subject to $\supp\Phi_{0}\subset[0,\eta_{g^{+}}]\times\mathbf{E}_{+}$. Note that this does NOT imply  the invertibility of $H^{+}$ (or equivalently, the existence of $0$-resolvent of $H^{+}$).

We first show that $\lim_{\lambda\rightarrow 0+}\|\Phi_{\lambda}-\Phi_{0}\|_{\infty}=0$. From the proof of Lemma \ref{lem:QellCompgPlus}, for any $\lambda\in(0,\infty)$, $\Phi_{\lambda}=R_{\lambda}g^{+}\in C_{0}^{1}(\overline{\sX_{+}})$ is the unique solution to \eqref{eq:LambdaHPlusEq} with $\supp\Phi_{\lambda}\subset[0,\eta_{g^{+}}]\times\mathbf{E}_{+}$. It follows from Lemma \ref{lem:ExistUniPhiIntEq} that $\Phi_{\lambda}$ is the unique solution to \eqref{eq:PhiIntEq}. Hence, for any $s\in[0,\eta_{g^{+}}]$, we have
\begin{align*}
\sup_{(t,i)\in\sX_{s}^{+}}\big|(\Phi_{\lambda}-\Phi_{0})(t,i)\big|=\sup_{(t,i)\in\sX_{s}^{+}}\bigg|\int_{t}^{\eta_{g^{+}}}\!\Big(\big(\wt{\mA}-\lambda\mV^{+}+\wt{\mB}S^{+}\big)(\Phi_{\lambda}-\Phi_{0})+\lambda\mV^{+}\Phi_{0}\Big)(r,i)\,dr\bigg|,
\end{align*}
where we recall $\sX_{s}^{+}=[s,\infty)\times\mathbf{E}_{+}$ and $\sX_{0}^{+}=\sX_{+}$. For any $(r,j)\in[0,\eta_{g^{+}}]\times\mathbf{E}_{-}$, let $\wt{\Phi}_{\lambda,0}^{(r)}$ be the restriction of $\Phi_{\lambda}-\Phi_{0}$ on $\overline{\sX_{r}^{+}}$. By \eqref{eq:STPlusSPlus} and \eqref{eq:STPlusSPlusNorm},
\begin{align*}
\big|\big(S^{+}(\Phi_{\lambda}-\Phi_{0})\big)(r,j)\big|=\Big|\Big(S_{r}^{+}\wt{\Phi}_{\lambda,0}^{(r)}\Big)(r,j)\Big|\leq\|S_{r}^{+}\|_{\infty}\Big\|\wt{\Phi}_{\lambda,0}^{(r)}\Big\|_{\infty}\leq\|S^{+}\|_{\infty}\sup_{(u,k)\in\sX^{+}_{r}}\big|(\Phi_{\lambda}-\Phi_{0})(u,k)\big|.
\end{align*}
Therefore, for any $s\in[0,\eta_{g^{+}}]$, we have
\begin{align*}
\sup_{(t,i)\in\sX_{s}^{+}}\big|(\Phi_{\lambda}-\Phi_{0})(t,i)\big|\leq\lambda\|\mV^{+}\|_{\infty}\|\Phi_{0}\|_{\infty}\big(\eta_{g^{+}}-s\big)+M_{\lambda}\int_{s}^{\eta_{g^{+}}}\sup_{(u,k)\in\sX_{r}^{+}}\big|(\Phi_{\lambda}-\Phi_{0})(u,k)\big|dr,
\end{align*}
where $M_{\lambda}:=\|\wt{\mA}\|_{\infty}+\lambda\|\mV^{+}\|_{\infty}+\|\wt{\mB}\|_{\infty}\|S^{+}\|_{\infty}$. By Gronwall inequality, we obtain that
\begin{align*}
\|\Phi_{\lambda}-\Phi_{0}\|_{\infty}=\sup_{(t,i)\in\sX_{0}^{+}}\big|(\Phi_{\lambda}-\Phi_{0})(t,i)\big|\leq\lambda\|\mV^{+}\|_{\infty}\|\Phi_{0}\|_{\infty}\eta_{g^{+}}e^{M_{\lambda}\eta_{g^{+}}}\rightarrow 0,\quad\text{as }\,\lambda\rightarrow 0+.
\end{align*}

Next, we will show that $\lim_{\ell\rightarrow\infty}\|\cQ^{+}_{\ell}g^{+}\|_{\infty}=0$. Without loss of generality, we assume that $g^{+}$ is nonnegative. Otherwise, we can prove the above statement for the positive and negative part of $g^{+}$, denoted by $g_{p}^{+}$ and $g_{n}^{+}$ respectively. Then, $\|\cQ^{+}_{\ell}g^{+}\|_{\infty}\leq\|\cQ^{+}_{\ell}g^{+}_{p}\|_{\infty}+\|\cQ^{+}_{\ell}g^{+}_{n}\|_{\infty}\rightarrow 0$, as $\ell\rightarrow\infty$. Note that when $g^{+}$ is nonnegative, since $\cQ_{\ell}^{+}$ is positive, we have $\cQ_{\ell}^{+}g^{+}\geq 0$ for any $\ell\in\bR_{+}$.

To begin with, since $\lim_{\lambda\rightarrow 0+}\|\Phi_{\lambda}-\Phi_{0}\|_{\infty}=0$ and $\|\Phi_{0}\|_{\infty}<\infty$. Hence, for any $(s,i)\in\sX_{+}$,
\begin{align}\label{eq:QellPlusPote}
\infty>|\Phi_{0}(s,i)|=\lim_{\lambda\rightarrow 0+}|\Phi_{\lambda}(s,i)|=\lim_{\lambda\rightarrow 0+}\int_{0}^{\infty}e^{-\lambda\ell}\big(\cQ^{+}_{\ell}g^{+}\big)(s,i)\,d\ell=\int_{0}^{\infty}\big(\cQ^{+}_{\ell}g^{+}\big)(s,i)\,d\ell,
\end{align}
where we have used the monotone convergence in the last equality.

Suppose that $\limsup_{\ell\rightarrow\infty}\|\cQ^{+}_{\ell}g^{+}\|_{\infty}>0$, then there exists $\varepsilon_{0}>0$ and $(s_{n},i_{n},\ell_{n})\in\sX_{+}\times\bR_{+}$, $n\in\bN$, with $\lim_{n\rightarrow\infty}\ell_{n}=\infty$, such that $(\cQ^{+}_{\ell_{n}}g^{+})(s_{n},i_{n})\geq\varepsilon_{0}$ for any $n\in\bN$. Without loss of generality, we can assume that $\ell_{n+1}-\ell_{n}>1$ and $i_{n}=i_{0}\in\mathbf{E}_{+}$ for all $n\in\bN$. Moreover, by part (i), $\supp\cQ_{\ell_{n}}^{+}g^{+}\subset[0,\eta_{g^{+}}]$, and so $(s_{n})_{n\in\bN}\subset[0,\eta_{g^{+}}]$, and hence    we may {also} assume that $\lim_{n\rightarrow\infty}s_{n}=s_{0}$ for some $s_{0}\in[0,\eta_{g^{+}}]$.

Since $(\cQ_{\ell}^{+})_{\ell\in\bR_{+}}$ is a strongly continuous contraction semigroup on $C_{0}(\overline{\sX_{+}})$, for any $b>0$,
\begin{align*}
\big\|\cQ^{+}_{\ell+b}g^{+}-\cQ^{+}_{\ell}g^{+}\big\|_{\infty}=\big\|\cQ^{+}_{\ell}\big(\cQ^{+}_{b}g^{+}-g^{+}\big)\big\|_{\infty}\leq\big\|\cQ^{+}_{b}g^{+}-g^{+}\big\|_{\infty}.
\end{align*}
In particular, $(\cQ^{+}_{\cdot}g^{+})(s,i)$ is uniformly continuous on $\bR_{+}$, uniformly for all $(s,i)\in\sX_{+}$. Thus, there exists a universal constant $\delta_{0}\in(0,1)$, such that for any $b\in[0,\delta_{0}]$, $(\cQ^{+}_{\ell_{n}+b}g^{+})(s_{n},i_{0})>\varepsilon_{0}/2$, for all $n\in\bN$, which implies that
\begin{align}\label{eq:IntQellgPlusn}
\int_{\ell_{n}}^{\ell_{n}+\delta_{0}}\big(\cQ^{+}_{\ell}g^{+}\big)(s_{n},i_{0})\,d\ell>\frac{\delta_{0}\varepsilon_{0}}{2}.
\end{align}
On the other hand, by \cite[Propositon I.1.5 (a)]{EthierKurtz2005}), $\int_{0}^{\delta_{0}}\cQ^{+}_{\ell}g^{+}d\ell\in\sD(H^{+})=C_{0}^{1}(\overline{\sX_{+}})$, so that by \cite[Propositon I.1.5 (b)]{EthierKurtz2005}) and \cite[1.2.B]{Dynkin1965},
\begin{align*}
\int_{\ell_{n}}^{\ell_{n}+\delta_{0}}\cQ^{+}_{\ell}g^{+}\,d\ell=\int_{0}^{\delta_{0}}\cQ_{\ell_{n}}^{+}\big(\cQ^{+}_{\ell}g^{+}\big)\,d\ell=\cQ_{\ell_{n}}^{+}\int_{0}^{\delta_{0}}\cQ^{+}_{\ell}g^{+}\,d\ell\in\sD(H^{+}).
\end{align*}
Hence, by \eqref{eq:WHPlus1} and \cite[Propositon I.1.5 (b)]{EthierKurtz2005}), and since $(\cQ_{\ell}^{+})_{\ell\in\bR_{+}}$ is a contraction semigroup,
\begin{align}
&\bigg\|\frac{\partial}{\partial s}\int_{\ell_{n}}^{\ell_{n}+\delta_{0}}\cQ^{+}_{\ell}g^{+}\,d\ell\,\bigg\|_{\infty}=\bigg\|\big(\mV^{+}H^{+}-\wt{\mA}-\wt{\mB}S^{+}\big)\int_{\ell_{n}}^{\ell_{n}+\delta_{0}}\cQ^{+}_{\ell}g^{+}\,d\ell\,\bigg\|_{\infty}\nonumber\\
&\quad\leq\|\mV^{+}\|_{\infty}\bigg\|H^{+}\cQ^{+}_{\ell_{n}}\!\int_{0}^{\delta_{0}}\cQ^{+}_{\ell}g^{+}d\ell\,\bigg\|_{\infty}+\big\|\wt{\mA}+\wt{\mB}S^{+}\big\|_{\infty}\bigg\|\cQ^{+}_{\ell_{n}}\!\int_{0}^{\delta_{0}}\cQ^{+}_{\ell}g^{+}d\ell\,\bigg\|_{\infty}\nonumber\\
&\quad\leq\|\mV^{+}\|_{\infty}\bigg\|\cQ^{+}_{\ell_{n}}H^{+}\!\int_{0}^{\delta_{0}}\cQ^{+}_{\ell}g^{+}d\ell\,\bigg\|_{\infty}+\big\|\wt{\mA}+\wt{\mB}S^{+}\big\|_{\infty}\bigg\|\int_{0}^{\delta_{0}}\cQ^{+}_{\ell}g^{+}d\ell\,\bigg\|_{\infty}\nonumber\\
\label{eq:NormDiffIntQellgPlus} &\quad\leq\|\mV^{+}\|_{\infty}\bigg\|H^{+}\!\int_{0}^{\delta_{0}}\cQ^{+}_{\ell}g^{+}d\ell\,\bigg\|_{\infty}+\big\|\wt{\mA}+\wt{\mB}S^{+}\big\|_{\infty}\bigg\|\int_{0}^{\delta_{0}}\cQ^{+}_{\ell}g^{+}d\ell\,\bigg\|_{\infty}=:M\in(0,\infty).
\end{align}
Combining \eqref{eq:IntQellgPlusn} and \eqref{eq:NormDiffIntQellgPlus}, for any $r\in(-\delta_{0}\varepsilon_{0}/(4M),\delta_{0}\varepsilon_{0}/(4M))$, we have
\begin{align}\label{eq:IntQellgPlusnr}
\int_{\ell_{n}}^{\ell_{n}+\delta_{0}}\big(\cQ^{+}_{\ell}g^{+}\big)(s_{n}+r,i_{0})\,d\ell>\frac{\delta_{0}\varepsilon_{0}}{4},\quad\text{for all }\,n\in\bN.
\end{align}
Let $N\in\bN$ be large enough so that $s_{0}\in(s_{n}-\delta_{0}\varepsilon_{0}/(4M),s_{n}+\delta_{0}\varepsilon_{0}/(4M))$ for all $n\geq N$. Since $\ell_{n+1}-\ell_{n}>0$ and $\delta_{0}\in(0,1)$, the intervals $(\ell_{n},\ell_{n}+\delta_{0})$, $n\in\bN$, are non-overlapping. Therefore, we obtain from \eqref{eq:IntQellgPlusnr} that
\begin{align*}
\int_{0}^{\infty}\big(\cQ^{+}_{\ell}g^{+}\big)(s_{0},i_{0})\,d\ell\geq\sum_{n=N}^{\infty}\int_{\ell_{n}}^{\ell_{n}+\delta_{0}}\big(\cQ^{+}_{\ell}g^{+}\big)(s_{0},i_{0})\,d\ell=\infty,
\end{align*}
which clearly contradicts \eqref{eq:QellPlusPote}. The proof of the lemma is now complete.
\end{proof}

\begin{proof}[Proof of Lemma \ref{lem:FhatPlusC01}.]
Let $g^{+}\in C_{c}^{1}(\overline{\sX_{+}})$ with $\supp g^{+}\subset[0,\eta_{g^{+}}]\times\mathbf{E}_{+}$ for some $\eta_{g^{+}}\in(0,\infty)$, and fix $\ell\in\bR$. Recall that $\wh{F}_{+}$ is defined as in \eqref{eq:FhatPlus}.

We first show that $\wh{F}_{+}(\cdot,\cdot,\cdot,\ell)\in C_{0}(\sX\times(-\infty,\ell])$. For any $i\in\mathbf{E}_{+}$, $s,s'\in\bR_{+}$, and $a,a'\in(-\infty,\ell]$, we have
\begin{align*}
&\Big|\big(\cQ^{+}_{\ell-a}g^{+}\big)(s,i)-\big(\cQ^{+}_{\ell-a'}g^{+}\big)(s',i)\Big|\\
&\quad\leq\Big|\big(\cQ^{+}_{\ell-a}g^{+}\big)(s,i)-\big(\cQ^{+}_{\ell-a}g^{+}\big)(s',i)\Big|+\Big|\big(\cQ^{+}_{\ell-a}g^{+}\big)(s',i)-\big(\cQ^{+}_{\ell-a'}g^{+}\big)(s',i)\Big|\\
&\quad\leq\Big|\big(\cQ^{+}_{\ell-a}g^{+}\big)(s,i)-\big(\cQ^{+}_{\ell-a}g^{+}\big)(s',i)\Big|+\big\|\cQ^{+}_{\ell-a}g^{+}-\cQ^{+}_{\ell-a'}g^{+}\big\|_{\infty}.
\end{align*}
Since $\cQ^{+}_{\ell-a}g^{+}\in C_{0}(\overline{\sX_{+}})$ and $\cQ^{+}_{\ell-\cdot}g^{+}$ is strongly continuous on $(-\infty,\ell]$, we see that $(\cQ^{+}_{\ell-\cdot}g^{+})(\cdot,i)$ is jointly continuous on $\bR_{+}\times(-\infty,\ell]$. Moreover, for any $i\in\mathbf{E}_{-}$,
\begin{align*}
&\Big|\big(S^{+}\cQ^{+}_{\ell-a}g^{+}\big)(s,i)-\big(S^{+}\cQ^{+}_{\ell-a'}g^{+}\big)(s',i)\Big|\\
&\quad\leq\Big|\big(S^{+}\cQ^{+}_{\ell-a}g^{+}\big)(s,i)-\big(S^{+}\cQ^{+}_{\ell-a}g^{+}\big)(s',i)\Big|+\Big|\big(S^{+}\cQ^{+}_{\ell-a}g^{+}\big)(s',i)-\big(S^{+}\cQ^{+}_{\ell-a'}g^{+}\big)(s',i)\Big|\\
&\quad\leq\Big|\big(S^{+}\cQ^{+}_{\ell-a}g^{+}\big)(s,i)-\big(S^{+}\cQ^{+}_{\ell-a}g^{+}\big)(s',i)\Big|+\big\|S^{+}\cQ^{+}_{\ell-a}g^{+}-S^{+}\cQ^{+}_{\ell-a'}g^{+}\big\|_{\infty}\\
&\quad\leq\Big|\big(S^{+}\cQ^{+}_{\ell-a}g^{+}\big)(s,i)-\big(S^{+}\cQ^{+}_{\ell-a}g^{+}\big)(s',i)\Big|+\|S^{+}\|_{\infty}\big\|\cQ^{+}_{\ell-a}g^{+}-\cQ^{+}_{\ell-a'}g^{+}\big\|_{\infty}.
\end{align*}
Since $\cQ^{+}_{\ell-a}g^{+}\in C_{0}(\overline{\sX_{+}})$, the condition $(a^{+})(i)$ implies that $S^{+}\cQ^{+}_{\ell-a}g^{+}\in C_{0}(\overline{\sX_{-}})$. Together with the strong continuity of $\cQ^{+}_{\ell-\cdot}g^{+}$ on $(-\infty,\ell]$ as well as the boundedness of $S^{+}$, we obtain that $(S^{+}\cQ^{+}_{\ell-\cdot}g^{+})(\cdot,i)$ is jointly continuous on $\bR_{+}\times(-\infty,\ell]$. In view of \eqref{eq:FhatPlus}, we obtain that $\wh{F}_{+}(\cdot,i,\cdot,\ell)$ is jointly continuous on $\bR_{+}\times(-\infty,\ell]$ for any $i\in\mathbf{E}$. It remains to show that $\wh{F}_{+}(\cdot,i,\cdot,\ell)$ vanishes at infinity for any $i\in\mathbf{E}$. By Lemma \ref{lem:QellCompgPlus}, $\supp\cQ^{+}_{\ell-a}g^{+}\subset[0,\eta_{g^{+}}]\times\mathbf{E}_{+}$, and the condition $(a^{+})(i)$ implies that $\supp S^{+}\cQ^{+}_{\ell-{a}}g^{+}\subset[0,\eta_{g^{+}}]\times\mathbf{E}_{-}$, so that $\supp\wh{F}_{+}(\cdot,i,a,\ell)\subset[0,\eta_{g^{+}}]$. Moreover, by Lemma \ref{lem:QellVanish}, $\lim_{a\rightarrow -\infty}\cQ^{+}_{\ell-a}g^{+}=0$ strongly, and since $\|S^{+}\|_{\infty}<\infty$, we also have $\lim_{a\rightarrow -\infty}S^{+}\cQ^{+}_{\ell-a}g^{+}=0$ strongly. Hence, $\wh{F}_{+}(\cdot,i,\cdot,\ell)$ vanishes at infinity for any $i\in\mathbf{E}$, and therefore, $\wh{F}_{+}(\cdot,\cdot,\cdot,\ell)\in C_{0}(\sX\times(-\infty,\ell])$.

Next, we will show that $\partial\wh{F}_{+}(\cdot,\cdot,\cdot,\ell)/\partial a$ exists and belongs to $C_{0}(\sX\times(-\infty,\ell])$. Since $g^{+}\in C_{0}^{1}(\overline{\sX_{+}})=\sD(H^{+})$, for any $a\in(-\infty,\ell]$, by \cite[Proposition I.1.5 (b)]{EthierKurtz2005}, we have $\cQ^{+}_{\ell-a}g^{+}\in\sD(H^{+})$, and
\begin{align}\label{eq:DevQella}
\frac{d}{da}\cQ^{+}_{\ell-a}g^{+}=-\cQ^{+}_{\ell-a}H^{+}g^{+}=-H^{+}\cQ^{+}_{\ell-a}g^{+}.
\end{align}
Together with Lemma \ref{lem:DiffQell}, we obtain that $\partial\wh{F}_{+}(\cdot,\cdot,a,\ell)/\partial a$ at any $a\in(-\infty,\ell]$, and
\begin{align*}
\frac{\partial}{\partial a}\wh{F}_{+}(\cdot,\cdot,a,\ell)= -\begin{pmatrix} \cQ^{+}_{\ell-a}H^{+}g^{+} \\ S^{+}\cQ^{+}_{\ell-a}H^{+}g^{+} \end{pmatrix}(\cdot,\cdot).
\end{align*}
Moreover, by \eqref{eq:WHPlus1} and the condition $(a^{+})(i)$, we have $H^{+}g^{+}\in C_{0}(\overline{\sX_{+}})$ with $\supp H^{+}g^{+}\subset[0,\eta_{g^{+}}]\times\mathbf{E}_{+}$. Using arguments similar to those leading to $\wh{F}_{+}(\cdot,\cdot,\cdot,\ell)\in C_{0}(\sX\times(-\infty,\ell])$ above, we conclude that $\partial\wh{F}_{+}(\cdot,\cdot,\cdot,\ell)/\partial a\in C_{0}(\sX\times(-\infty,\ell])$.

Finally, we will show that $\partial\wh{F}_{+}(\cdot,\cdot,\cdot,\ell)/\partial s$ exists and belongs to $C_{0}(\sX\times(-\infty,\ell])$. For any $a\in(-\infty,\ell]$, since $\cQ^{+}_{\ell-a}g^{+}\in C_{0}^{1}(\overline{\sX_{+}})=\sD(H^{+})$, by \eqref{eq:WHPlus}, we have
\begin{align*}
\mV^{-1}\bigg(\frac{\partial}{\partial s}+\wt{\mathsf{\Lambda}}\bigg) \begin{pmatrix} I^{+} \\ S^{+} \end{pmatrix} \cQ_{\ell-a}g^{+}&= \begin{pmatrix} I^{+} \\ S^{+} \end{pmatrix} H^{+}\cQ_{\ell-a}g^{+}.
\end{align*}
Consequently, in view of \eqref{eq:FhatPlus}, $\partial\wh{F}_{+}(s,i,a,\ell)/\partial s$ exists at any $(s,i,a)\in\sX\times(-\infty,\ell]$, and
\begin{align*}
\frac{\partial\wh{F}_{+}}{\partial s}(s,i,a,\ell)= \begin{pmatrix} \displaystyle{\frac{\partial}{\partial s}\big(\cQ^{+}_{\ell-a}g^{+}\big)} \\ \\ \displaystyle{\frac{\partial}{\partial s}\big(S^{+}\cQ^{+}_{\ell-a}g^{+}\big)} \end{pmatrix} (s,i)= \begin{pmatrix} \mV^{+}\cQ^{+}_{\ell-a}H^{+}g^{+}-\big(\wt{\mA}+\wt{\mB}S^{+}\big)\cQ^{+}_{\ell-a}g^{+} \\ \\ \mV^{-}S^{+}\cQ^{+}_{\ell-a}H^{+}g^{+}+\big(\wt{\mC}+\wt{\mD}S^{+}\big)\cQ^{+}_{\ell-a}g^{+} \end{pmatrix} (s,i).
\end{align*}
With similar technique as before, the right-hand sides above, as a function of $(s,i,a)$, belongs to $C_{0}(\sX\times(-\infty,\ell])$.
\end{proof}

\section{Two additional technical lemmas}

In this section, we establish two additional technical lemmas that are used in the proofs of our main theorems. We begin with a lemma regarding the distributions of the first and second jump time of $Z^{2}$ (see also \cite[Section 8.4.2]{RolskiSchmidliSchmidtTeugels1999}).

\begin{lemma}\label{lem:TailDistJumpTildeZ2}
Let $\wt{\gamma}_{1}$ and $\wt{\gamma}_{2}$ be the first and the second jump time of $Z^{2}$, respectively. Then, for any $(s,i,a)\in\sZ$ and $r\in\bR_{+}$,
\begin{align}\label{eq:TailDistJumpTildeZ2}
\wt{\bP}_{s,i,a}\big(\wt{\gamma}_{1}>r\big)=\exp\bigg(\int_{s}^{s+r}\mathsf{\Lambda}_{u}(i,i)\,du\bigg).
\end{align}
In particular,
\begin{align}\label{eq:DistFirstJumpTildeZ2}
\wt{\bP}_{s,i,a}\big(\wt{\gamma}_{1}\leq r\big)&\leq Kr,\\
\label{eq:DistSecondJumpTildeZ2} \wt{\bP}_{s,i,a}\big(\wt{\gamma}_{2}\leq r\big)&\leq K^{2}r^{2}.
\end{align}
\end{lemma}

\begin{proof}
For any $(s,i,a)\in\sZ$ and $r\in\bR_{+}$, by \eqref{eq:Probz}, \eqref{eq:ProcZ}, and \eqref{eq:LawXXStar},
\begin{align*}
&\wt{\bP}_{s,i,a}\big(\wt{\gamma}_{1}>r\big)=\wt{\bP}_{s,i,a}\Big(\inf\big\{t\in\bR_{+}:Z^{2}_{t}\neq Z^{2}_{t-}\big\}>r\Big)\\
&\quad =\bP_{s,(i,a)}\Big(\inf\big\{t\in\bR_{+}:X_{t+s}\neq X_{(t+s)-}\big\}>r\Big)=\bP_{s,(i,a)}\Big(\inf\big\{t\in[s,\infty]:X_{t}\neq X_{t-}\big\}>r+s\Big)\\
&\quad =\bP^{*}_{s,i}\Big(\inf\big\{t\in[s,\infty]:X_{t}^{*}\neq X_{t-}^{*}\big\}>r+s\Big)=\exp\bigg(\int_{s}^{s+r}\mathsf{\Lambda}_{u}(i,i)\,du\bigg).
\end{align*}
Hence, by Assumption \ref{assump:GenLambda} (i),
\begin{align*}
\wt{\bP}_{s,i,a}\big(\wt{\gamma}_{1}\leq r\big)=1-\exp\bigg(\int_{s}^{s+r}\mathsf{\Lambda}_{u}(i,i)\,du\bigg)\leq 1-e^{-Kr}\leq Kr.
\end{align*}
Moreover, for any $\wt{\omega}\in\wt{\Omega}$,
\begin{align*}
\wt{\gamma}_{2}(\wt{\omega})&=\inf\!\Big\{t\!\in\![\wt{\gamma}_{1}(\wt{\omega}),\infty]\!:Z^{2}_{t}(\wt{\omega})\!\neq\!Z^{2}_{t-}(\wt{\omega})\Big\}=\wt{\gamma}_{1}(\wt{\omega})\!+\!\inf\!\Big\{t\in\overline{\bR}_{+}\!:Z^{2}_{(\wt{\gamma}_{1}(\wt{\omega})+t)-}(\wt{\omega})\!\neq\!Z^{2}_{\wt{\gamma}_{1}(\wt{\omega})+t}(\wt{\omega})\Big\}\\
&=\wt{\gamma}_{1}(\wt{\omega})+\inf\Big\{t\in\overline{\bR}_{+}:Z^{2}_{t-}\big(\theta_{\wt{\gamma}_{1}(\wt{\omega})}\wt{\omega}\big)\neq Z^{2}_{t}\big(\theta_{\wt{\gamma}_{1}(\wt{\omega})}\wt{\omega}\big)\Big\}=\wt{\gamma}_{1}(\wt{\omega})+\big(\wt{\gamma}_{1}\circ{\theta}_{\wt{\gamma}_{1}}\big)(\wt{\omega}),
\end{align*}
and thus
\begin{align*}
\big(Z^{1}_{\wt{\gamma}_{1}}\circ{\theta}_{\wt{\gamma}_{1}}\big)(\wt{\omega})=Z^{1}_{\wt{\gamma}_{1}\circ\,\theta_{\wt{\gamma}_{1}}(\wt{\omega})}\big(\theta_{\wt{\gamma}_{1}(\wt{\omega})}\wt{\omega}\big)=Z^{1}_{(\wt{\gamma}_{1}+\wt{\gamma}_{1}\circ\,\theta_{\wt{\gamma}_{1}})(\wt{\omega})}(\wt{\omega})=Z^{1}_{\wt{\gamma}_{2}}(\wt{\omega}).
\end{align*}
Therefore, by \eqref{eq:ProcZ}, the strong Markov property of $\wt{\cM}$ (cf. \cite[Theorem III.9.4]{RogersWilliams1994}), and \eqref{eq:DistFirstJumpTildeZ2},
\begin{align*}
\wt{\bP}_{s,i,a}\big(\wt{\gamma}_{2}\leq r\big)&=\wt{\bP}_{s,i,a}\Big(Z^{1}_{\wt{\gamma}_{2}}\in[s,s+r]\Big)=\wt{\bP}_{s,i,a}\Big(Z^{1}_{\wt{\gamma}_{1}}\in[s,s+r],\,Z^{1}_{\wt{\gamma}_{2}}\in[s,s+r]\Big)\\
&=\wt{\bP}_{s,i,a}\Big(Z^{1}_{\wt{\gamma}_{1}}\in[s,s+r],\,Z^{1}_{\wt{\gamma}_{1}}\circ\theta_{\wt{\gamma}_{1}}\in[s,s+r]\Big)\\
&=\wt{\bE}_{s,i,a}\Big(\1_{\big\{Z^{1}_{\wt{\gamma}_{1}}\in[s,s+r]\big\}}\,\wt{\bP}_{s,i,a}\Big(Z^{1}_{\wt{\gamma}_{1}}\circ\theta_{\wt{\gamma}_{1}}\in[s,s+r]\,\Big|\wt{\sF}_{\wt{\gamma}_{1}}\Big)\Big)\\
&=\wt{\bE}_{s,i,a}\bigg(\1_{\big\{Z^{1}_{\wt{\gamma}_{1}}\in[s,s+r]\big\}}\,\wt{\bP}_{Z^{1}_{\wt{\gamma}_{1}},Z^{2}_{\wt{\gamma}_{1}},Z^{3}_{\wt{\gamma}_{1}}}\Big(Z^{1}_{\wt{\gamma}_{1}}\in[s,s+r]\Big)\bigg)\\
&\leq Kr\,\wt{\bP}_{s,i,a}\Big(Z^{1}_{\wt{\gamma}_{1}}\in[s,s+r]\Big)=Kr\,\wt{\bP}_{s,i,a}\left(\widetilde{\gamma}_{1}\leq r\right)\leq K^{2}r^{2},
\end{align*}
which completes the proof of the lemma.
\end{proof}

The next lemma establishes some relationship between $\wt{\gamma}_{1}$ and $\wt{\tau}^{+}_{\ell}$.

\begin{lemma}\label{lem:GammaTailImpl}
For any $(s,i)\in\sX_{+}$, $a\in\bR$, and $h\in(0,\infty)$,
\begin{align}\label{eq:Gamma1overh}
\1_{\{\wt{\gamma}_{1}>h/v(i)\}}&=\1_{\{\wt{\gamma}_{1}>h/v(i)\}}\1_{\{\wt{\tau}^{+}_{a+h}=h/v(i)\}}\quad\wt{\bP}_{s,i,a}-a.s.\,,\\
\label{eq:hoverGamma1} \1_{\{\wt{\gamma}_{1}\leq h/v(i)\}}&=\1_{\{\wt{\gamma}_{1}\leq h/v(i)\}}\1_{\{\wt{\gamma}_{1}\leq\wt{\tau}^{+}_{a+h}\}}\quad\wt{\bP}_{s,i,a}-a.s.\,.
\end{align}
\end{lemma}

\begin{proof}
For any $(s,i)\in\sX_{+}$, $a\in\bR$, and $h\in(0,\infty)$,
\begin{align*}
&\bigg\{Z^{3}_{t}=a+\int_{0}^{t}v\big(Z^{2}_{u}\big)\,du,\,\text{for all }t\in\bR_{+}\bigg\}\bigcap\bigg\{\wt{\gamma}_{1}>\frac{h}{v(i)}\bigg\}\\
&\quad =\bigg\{Z^{3}_{t}=a\!+\!\!\int_{0}^{t}v\big(Z^{2}_{u}\big)\,du,\,\text{for all }t\in\bR_{+}\bigg\}\bigcap\bigg\{\wt{\gamma}_{1}\!>\!\frac{h}{v(i)}\bigg\}\bigcap\Big\{Z^{3}_{t}=a\!+\!v(i)t,\,\text{for all }t\in[0,\wt{\gamma}_{1}]\Big\}\\
&\quad\subset\bigg\{Z^{3}_{t}<a+h,\,\text{for all }\,t\in\bigg[0,\frac{h}{v(i)}\bigg);\,\,Z^{3}_{h/v(i)}=a+h;\,\,Z^{3}_{t}>a+h,\,\text{for all }\,t\in\bigg(\frac{h}{v(i)},\wt{\gamma}_{1}\bigg]\bigg\}\\
&\quad\subset\bigg\{\wt{\tau}^{+}_{a+h}=\frac{h}{v(i)}\bigg\},
\end{align*}
and
\begin{align*}
&\bigg\{Z^{3}_{t}=a+\int_{0}^{t}v\big(Z^{2}_{u}\big)\,du,\,\text{for all }t\in\bR_{+}\bigg\}\bigcap\bigg\{\wt{\gamma}_{1}\leq\frac{h}{v(i)}\bigg\}\\
&\quad =\bigg\{Z^{3}_{t}=a+\int_{0}^{t}v\big(Z^{2}_{u}\big)\,du,\,\text{for all }t\in\bR_{+}\bigg\}\bigcap\bigg\{\wt{\gamma}_{1}\leq\frac{h}{v(i)}\bigg\}\bigcap\Big\{Z^{3}_{t}\leq a+h,\,\text{for all }t\in[0,\wt{\gamma}_{1}]\Big\}\\
&\quad\subset\Big\{Z^{3}_{t}\leq a+h,\,\text{for all }t\in[0,\wt{\gamma}_{1}]\Big\}\subset\big\{\wt{\gamma}_{1}\leq\wt{\tau}^{+}_{a+h}\big\}.
\end{align*}
Then, \eqref{eq:Gamma1overh} and \eqref{eq:hoverGamma1} follow from \eqref{eq:DistTildeZ3IntTildeZ2}.
\end{proof}


\bibliographystyle{alpha}

\end{document}